\documentclass[reqno]{amsart}

\usepackage{hyperref}
\usepackage{tabularx} 

\usepackage{varwidth}

\usepackage{enumitem}
\usepackage{subcaption}  

\usepackage{esint}
\usepackage{amssymb,centernot}
\usepackage{tikz}
\usepackage{pgfplots}
\pgfplotsset{compat=1.18}

\usepackage{graphicx}
\usepackage{amsrefs,enumitem,pgfkeys}
\usepackage{soul}
\usepackage{dirtytalk}

\usepackage{MnSymbol,wasysym}
\usepackage{fancyhdr}
\usepackage{time}

\usepackage[margin=1in]{geometry}
 
\newcommand\restr[2]{{
  \left.\kern-\nulldelimiterspace 
  #1 
  \vphantom{\big|} 
  \right|_{#2} 
  }}

\usepackage{comment}
\newtheorem{theorem}{Theorem}
\newtheorem{lemma}[theorem]{Lemma}
\newtheorem{corollary}[theorem]{Corollary}
\newtheorem{proposition}[theorem]{Proposition}
\theoremstyle{definition}

\newtheorem{remark}[theorem]{Remark}

\newtheorem{definition}[theorem]{Definition}

\newcommand{\eref}[1]{(\ref{e.#1})}
\newcommand{\tref}[1]{Theorem \ref{t.#1}}
\newcommand{\lref}[1]{Lemma \ref{l.#1}}

\newcommand{\cref}[1]{Corollary \ref{c.#1}}
\newcommand{\fref}[1]{Figure \ref{f.#1}}
\newcommand{\sref}[1]{Section \ref{s.#1}}

\newcommand{\dref}[1]{Definition \ref{d.#1}}

\numberwithin{theorem}{section}
\numberwithin{equation}{section}

\newcommand{\loc}{\textup{loc}}
\newcommand{\Na}{\mathcal{N}}

\newcommand{\Z}{\mathbb{Z}}
\newcommand{\R}{\mathbb{R}}
\newcommand{\Oa}{\mathcal{O}}

\newcommand{\grad}{\nabla}

\usepackage{mathtools}

\newcommand{\pt}{\partial}

\newcommand{\pd}{{\partial}^{\scriptstyle +}}
\newcommand{\pp}{\partial_{\textup{p}}}
\newcommand{\ppp}{\partial_{\textup{p}}^{\scriptstyle +}}

\newcommand{\bca}{\begin{cases}}
\newcommand{\eca}{\end{cases}}

\newcommand{\lb}{\left(}

\newcommand{\rb}{\right)}
\newcommand{\lmb}{\left[}
\newcommand{\rmb}{\right]}
\newcommand{\lma}{\left\{}
\newcommand{\rma}{\right\}}
\newcommand{\Ha}{{\mathcal{H}}}

\newcommand{\Ra}{{\mathcal{R}}}
\newcommand{\Ca}{{\mathcal{C}}}

\newcommand{\Ta}{{\mathcal{T}}}

\newcommand{\csubset}{\subset\joinrel\subset}

\newcommand{\x}{\mathbf{x}}
\newcommand{\y}{\mathbf{y}}

\newcommand{\gd}{\nabla}
\newcommand{\rta}{\rightarrow}

\newcommand{\be}{\begin{equation}}
\newcommand{\ee}{\end{equation}}

\newcommand{\bt}{\begin{thm}}
\newcommand{\et}{\end{thm}}
\newcommand{\bc}{\begin{cor}}
\newcommand{\ec}{\end{cor}}
\newcommand{\bl}{\begin{lem}}
\newcommand{\el}{\end{lem}}
\newcommand{\norm}[1]{\left\lVert#1\right\rVert}
\newcommand{\avg}[1]{\langle#1\rangle}
\newcommand{\normm}[1]{{\left\vert\kern-0.25ex\left\vert\kern-0.25ex\left\vert #1 
    \right\vert\kern-0.25ex\right\vert\kern-0.25ex\right\vert}}

\newcommand{\dist}{\operatorname{dist}}

\newcommand{\weakcv}{\rightharpoonup}

\def\XXint#1#2#3{{\setbox0=\hbox{$#1{#2#3}{\int}$ }
\vcenter{\hbox{$#2#3$ }}\kern-.6\wd0}}

\usetikzlibrary{decorations.pathmorphing, arrows.meta, patterns}

\usepackage{marvosym}

\newcommand{\argmin}{\mathop{\textup{argmin}}}

\def\XXint#1#2#3{{\setbox0=\hbox{$#1{#2#3}{\int}$ }
\vcenter{\hbox{$#2#3$ }}\kern-.6\wd0}}

\newcommand{\ep}{\varepsilon}
\newcommand{\osc}{\mathop{\textup{osc}}}

\title{Homogenization of a vertical oscillating Neumann condition}
\author{William M Feldman}
\address{Department of Mathematics, University of Utah, Salt Lake City, UT 84112}
\email{feldman@math.utah.edu}
\author{Zhonggan Huang}
\address{Department of Mathematics, University of Utah, Salt Lake City, UT 84112}
\email{zhonggan@math.utah.edu}

\date{\today}

\begin{document}

\begin{abstract}
We homogenize the Laplace and heat equations with the Neumann data oscillating in the``vertical" $u$-variable. These are simplified models for interface motion in heterogeneous media, particularly capillary contact lines. The homogenization limit reveals a pinning effect at zero tangential slope, leading to a novel singularly anisotropic pinned Neumann condition. The singular pinning creates an unconstrained contact set generalizing the contact set in the classical thin obstacle problem. We establish a comparison principle for the heat equation with this new type of boundary condition. The comparison principle enables a proof of homogenization via the method of half-relaxed limits from viscosity solution theory. Our work also demonstrates — for the first time in a PDE problem in multiple dimensions — the emergence of rate-independent pinning from gradient flows with wiggly energies. Prior limit theorems of this type, in rate independent contexts, were limited to ODEs and PDEs in one dimension.
\end{abstract}

\maketitle

\tableofcontents

\section{Introduction}\label{section.introduction}

In this paper we study the homogenization limit $\ep \to 0$ of the following heat equation with the Neumann data oscillating in the \say{vertical} $u$-variable
\be
\label{eq.pNNep}
\bca
\pt_t u^\ep = \Delta u^\ep &\textup{ in } B_1\cap \{x_1>0\}\times(0,\infty)\\
\pt_{1} u^\ep = f\lb \frac{u^\ep}{\ep}\rb & \textup{ on } B_1\cap \{x_1=0\}\times(0,\infty).
\eca
\ee
Here $f: \R \to \R$ is a sufficiently regular $1$-periodic function, $B_1$ is the unit ball in $\R^d$ with $d\ge 2$, and $\partial_1 = \partial_{x_1}$. We work in the upper-half unit ball as a model domain and denote
\[
B_1^+:=B_1\cap \{x_1>0\} \ \quad \textup{and} \quad  \ B_1':=B_1\cap \{x_1=0\}.
\]
We also study the homogenization of the steady state problem
\be
\label{eq.homogenizationNeumann}
\bca
\Delta u^\ep = 0 & \textup{ in }B_1^+\\
\pt_1 u^\ep = f\lb \frac{u^\ep}{\ep}\rb & \textup{ on }B_1'.
\eca
\ee
Such problems arise when the graph of $u$ is considered as an interface in contact with a heterogeneous boundary, see \fref{1}. The appeal of this model is that it is simply posed, yet also captures the interplay between multi-dimensionality, homogenization, rate independent pinning, and rate dependent gradient flow. We discuss the importance of these themes further in Section \ref{subsec.literature} below.

\begin{figure}[hbtp]
    \centering
   \begin{tikzpicture}[scale=1.8] 
        \draw[thick] (0,0) --++(90:3cm);
        \fill[pattern=north west lines] (0,0) rectangle ++(-0.1,3); 
        
        \draw[gray, thick, scale=2/6.2829, domain=0:6.2829] plot (\x,{6+0.3*sin(0.8*\x r)-0.5*\x });
        
        \draw[gray, thick] 
            plot[variable=\y, domain=0:3, samples=100] ({0.1*cos(15*\y r)}, \y);
        
        \node[above right, black] at (2,.5) {$\pt_t u^\ep = \Delta u^\ep$};
        \node[black, left] at (-0.2,1.7) {$\partial_1 u^\varepsilon = f\left(\frac{u^\varepsilon}{\varepsilon}\right)$};
        
        \draw[->, >=stealth, thick] (0.5, -0.2) -- (1.5, -0.2) node[midway, below] {$x_1$};
    \end{tikzpicture}
     \captionsetup{width=0.7\textwidth}
    \caption{The graph of $u^\ep$ over $B_1^+$ is a moving interface interacting with an inhomogeneous medium via a Neumann condition at the boundary $B_1'$. }
    \label{fig:1}
    \label{f.1}
\end{figure}
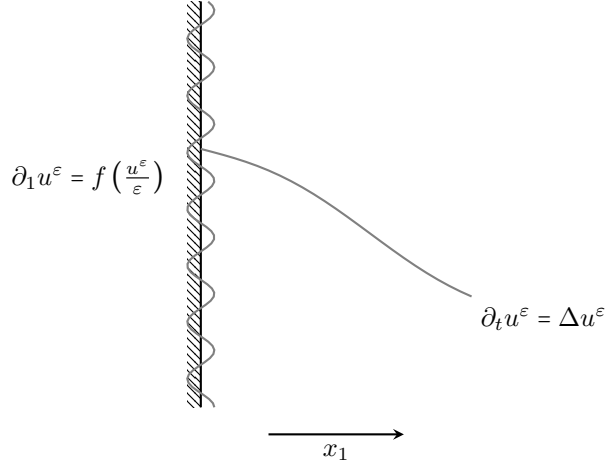

The full description of the homogenization limits of the above two problems requires detailed language from the theory of viscosity solutions. For the purposes of exposition we begin by describing our result at a formal level.  We will present the PDE following the formalism common in the study of rate independent energetic systems (a la \cite{MielkeRoubicek} and see Section \ref{subsec.literature} below). The homogenization limit of \eqref{eq.pNNep} is a heat equation with a singular pinned Neumann condition
\be\label{eq.pGDwithmotionlaw}
\bca
\pt_t u= \Delta u  & \textup{ in }B_1^+\times(0,\infty)\\
\pt_1 u \in \pt \Ra(\pt_t u;  \gd' u ) & \textup{ on } B_1'\times(0,\infty).
\eca
\ee
Here $\displaystyle\gd' u$ is the tangential gradient of $u$ on $B_1'$, and $\pt \Ra(\tau; p)$ is the subdifferential in $\tau$ of the following function 
\be\label{eq.thedissipationcoefficient}
\Ra(\tau; p):=
\bca
L^*(p) \tau & \textup{ if }\tau \ge0\\
L_*(p) \tau & \textup{ if }\tau < 0.
\eca
\ee
Here $L_*$ and $L^*$ are certain homogenized coefficients depending on the tangential slope $p\in  \R^{d-1}$. 

We will show the following simple formula for the homogenized coefficients
\be\label{eq.explicithomogenizedcoefficientinintro}
L_\ast(p)=\lb\min f \rb 1_{\{p=0\}} + \avg{f} 1_{\{p \neq 0\}} \ \hbox{ and } \ L^\ast(p)=\lb\max f \rb1_{\{p=0\}} + \avg{f} 1_{\{p \neq 0\}}.
\ee
Here $\displaystyle \avg{f}:=\int_0^1 f(u)du $ is the average value of $f$.

The rate independent evolution law at the Neumann boundary in \eqref{eq.pGDwithmotionlaw} can be described heuristically as:
\begin{itemize}
\item  When the tangential gradient $\gd'u\ne0$ the inner normal derivative $\pt_1 u = \avg{f}$ is {pinned} exactly at the average value. When the tangential gradient $\gd'u=0$ then the inner normal derivative
\[
\pt_1 u\in [\min f,\max f]
\]
is pinned in a nontrivial interval.
\item  If $\pt_t u>0$, then $\pt_1 u \ge \avg{f}$ and if further $\gd' u=0$ then $\pt_1 u = \max f$ (in a certain weak sense).
\item  If $\pt_t u<0$, then $\pt_1 u \le \avg{f}$ and if further $\gd' u=0$ then $\pt_1 u = \min f$ (in a certain weak sense).
\end{itemize}
 Due to the discontinuity of the homogenized coefficients $L_*$ and $L^*$, our description here is not completely accurate, especially the final two bullet points which we are only able to interpret in the viscosity solutions / comparison principle sense and have no classical sense. See Definition \ref{d.definevissoltopGD}, especially conditions \ref{condition(b)} and \ref{condition(c)}, for a precise description in the language of viscosity solution theory.

Our central main result is the homogenization limit from \eqref{eq.pNNep} to \eqref{eq.pGDwithmotionlaw}. We will study solutions with a fixed Dirichlet data on the upper part of the parabolic boundary of $B_1^+ \times (0,\infty)$:
\begin{equation}\label{e.ppp-1}
    \ppp (B_1\times(0,\infty)) := \big[(\partial B_1 \cap \{x_1 \geq 0\}) \times (0,\infty)\big] \cup \big[\overline{B_1^+} \times \{t=0\}\big].
\end{equation}
 For the introduction we write the result briefly, see Theorem \ref{t.homogenizationofparabolic} below for a more detailed statement including precise assumptions on the Dirichlet boundary data and heterogeneity $f$.

\begin{theorem}\label{t.homogenizationofparabolicformal}
Suppose $f$ is 1-periodic and regular on $\R$. Then a sequence $u^\ep$ of solutions to \eqref{eq.pNNep}, with a fixed Dirichlet condition on the parabolic boundary $\ppp (B_1^+\times[0,\infty))$, converge locally uniformly on $(x,t) \in \overline{B_1^+}\times[0,\infty)$ to the unique continuous viscosity solution $u$ to \eqref{eq.pGDwithmotionlaw} with the same data on the parabolic boundary.

\end{theorem}

A central element of the proof of this theorem is the comparison principle / uniqueness for \eqref{eq.pGDwithmotionlaw}.  The important role played by comparison principle is well known in the homogenization of nonlinear elliptic and parabolic problems. In particular, the proof of Theorem \ref{t.homogenizationofparabolicformal} uses the half relaxed limit approach of Barles and Perthame \cites{BarlesPerthame1,BarlesPerthame2} which relies on  comparison principle for semicontinuous sub and supersolutions.

\begin{theorem}\label{t.comparisonprincipleformal}
    Suppose $u$ is a subsolution and $v$ is a supersolution to \eqref{eq.pGDwithmotionlaw}, as in Definition \ref{d.definevissoltopGD}. If $u\le v$ on the parabolic boundary $\ppp (B_1^+\times[0,\infty))$ (as in \eref{ppp-1}) then $u\le v$ on the whole space-time domain $\overline{B_1^+}\times[0,\infty)$.

\end{theorem}

There are two major difficulties in proving the comparison principle Theorem \ref{t.comparisonprincipleformal}. The first is the failure of uniform obliqueness, and as a result the existence of free degenerate regions in \eqref{eq.pGDwithmotionlaw}
\[
\{\pt_1 u \ne \avg{f}\}\subset\{\gd' u = 0\}.
\]
As we shall see later that these sets are not empty and in the proof we need to discuss whether the location of the touching points are in such degenerate regions. Outside the degenerate set the solutions satisfy standard Neumann boundary condition, but on the degenerate set the solutions satisfy a Dirichlet boundary condition with the boundary data depending only on the time variable. This makes the proof quite delicate as we don't know either the shape of the degenerate region or the values on such regions. The lack of uniform obliqueness also fails the construction of doubling test functions by Barles in \cite{BARLES1999191} and therefore we use the inf / sup-convolution type of arguments (See Appendix \ref{appendix.tangentialregularization} and also our previous work \cite{feldman2024regularitytheorygradientdegenerate}). It is still interesting to explore if there is a doubling variable argument for the proof of comparison principle of \eqref{eq.pGDwithmotionlaw}. 

Another difficulty in proving the comparison principle is that the boundary condition in \eqref{eq.pGDwithmotionlaw} is essentially a differential inclusion. As we shall see in the following discussions, there is no uniqueness for the steady state equation. The crucial ingredients that ensure the uniqueness in the parabolic flow are the dynamic slope conditions \ref{condition(b)} and \ref{condition(c)} in Definition \ref{d.definevissoltopGD}. As we mentioned before, these conditions, especially condition \ref{condition(c)}, can only be made precise by using viscosity solution notions.

The steady state equation for \eqref{eq.pGDwithmotionlaw}, which is also the homogenization limit of the elliptic problems \eqref{eq.homogenizationNeumann}, takes the form
\be\label{eq.generalhomogenizedequidef}
\bca
\Delta u = 0 & \textup{ in }B_1^+\\
\pt_1 u \in [L_*(\gd' u),L^*(\gd' u)] & \textup{ on }B_1'.
\eca
\ee
Thanks to the explicit formula \eqref{eq.explicithomogenizedcoefficientinintro} for $L_*$ and $L^*$ the above equation has the following equivalent form
\be\label{eq.generalhomogenizedequidefequiform}
\bca
\Delta u = 0 & \textup{ in }B_1^+\\
(\pt_1 u-\avg{f}) |\gd' u|=0 \textup{ and }\pt_1 u\in [\min f,\max f] & \textup{ on }B_1'.
\eca
\ee
There is, in general, non-uniqueness of solutions to Dirichlet boundary value problems for \eqref{eq.generalhomogenizedequidef} and \eqref{eq.generalhomogenizedequidefequiform}. This nonuniqueness also occurs before homogenization in \eqref{eq.homogenizationNeumann}, which is not so obvious, but it is in fact one consequence of the homogenization result that we describe next.

Despite the general non-uniqueness we can still determine the {unique} homogenization limits of special solutions to \eqref{eq.homogenizationNeumann}. On one hand by Perron's method, we can construct {maximal subsolutions} 
\[
u_{\textup{max}}^\ep(x)=\max\{u(x) \ ; \ u\textup{ is a subsolution to \eqref{eq.homogenizationNeumann}}\}
\]
and symmetrically {minimal supersolutions}
\[
u_{\textup{min}}^\ep(x)=\min\{u(x) \ ; \ u\textup{ is a supersolution to \eqref{eq.homogenizationNeumann}}\}.
\]
We call both $u_{\textup{max}}^\ep$ and $u_{\textup{min}}^\ep$ the {extremal} solutions.  

On the other hand, we can consider the corresponding energy of \eqref{eq.pNNep} and \eqref{eq.homogenizationNeumann}
\be\label{eq.theenergyepsilonforlaminarcase}
E_\ep(u,B_1^+):= \int_{B_{1}^+} \frac{1}{2}|\gd u|^2 + \int_{B_1'} \int_0^{u(x')} f (r/\ep) dr dx'.
\ee
We define {global energy minimizers} $u_{\textup{glb}}^\ep$ satisfying
\[
E_\ep(u_{\textup{glb}}^\ep,B_1^+) \le E_{\ep}(v, B_1^+),
\] 
{for any appropriate test function }\(v\){ such that }\(v=u_{\textup{glb}}^\ep=g\textup{ on }\pt B_1 \cap \{x_1\ge0\}\).

Again, it is not obvious whether these notions provide actually distinct solutions. In the case of standard Neumann data $\partial_1u = f(x)$, they are all the same. We will see, via our homogenization result below, that these three solutions of \eqref{eq.homogenizationNeumann} are often distinct, see also Proposition \ref{prop.generalexistenceoffacets} later.

Typical arguments from $\Gamma$-convergence theory \cite{braides2006handbook} show that the energies $E_\ep$ in \eqref{eq.theenergyepsilonforlaminarcase} $\Gamma$-converge to the following energy
\be\label{eq.homogenizedenergy}
E_*(u,B_1^+):=\int_{B_{1}^+} \frac{1}{2}|\gd u|^2 + \int_{B_1'} u(x') \avg{f}  dx'.
\ee
In particular, the homogenized PDE associated with energy minimizing solutions of \eqref{eq.homogenizationNeumann} is the standard Neumann problem
\be\label{eq.standarNeumann}
\bca
 \Delta u = 0 & \textup{ in }B_1^+\\
 \pt_1 u = \avg{f} & \textup{ on }B_1'.
\eca
\ee
See Lemma \ref{l.localuniformconvergenceofglobalminimizers} below for a precise statement and proof.

The homogenization for the extremal solutions $u_{\textup{max}}^\ep$ and $u_{\textup{min}}^\ep$ is trickier as there is no $\Gamma$-convergence type theory for them. In our second main theorem, we show that the homogenization limits of the extremal solutions are exactly extremal solutions to the general homogenized equation \eqref{eq.generalhomogenizedequidef}. We write the results in a non-technical way, see Section \ref{section.extremalsteadystates} for a more detailed statement.

\begin{theorem}\label{t.homogenizationofspecialsolutions}
The extremal solutions of \eqref{eq.homogenizationNeumann} converge as $\ep \to 0$ to the extremal solutions of \eqref{eq.generalhomogenizedequidef}.

\end{theorem}

 We summarize the homogenization results for the steady states in Figure \ref{fig:illustratecriticalpoints}.

\begin{figure}[hbtp]
     \centering
\begin{tikzpicture}[scale=1]
    \begin{axis}[
        axis lines = none, 
        xlabel = {$x$},
        ylabel = {$y$},
        domain=-1.6:2,
        ymin=-9, ymax=9, 
        samples=100,
        enlargelimits,
        width=16cm,
        xtick=\empty, 
        ytick=\empty, 
        ]
        \addplot[
            gray!80,
            thick
        ] {sin(deg(x/0.16))+1*x^2};

        \coordinate (p1) at (-1.19391, {sin(deg((-1.19391)/0.16))+1*(-1.19391)^2});
        \coordinate (p2) at (-0.239075, {sin(deg((-0.239075)/0.16))+1*(-0.239075)^2});
        \coordinate (p3) at (0.716945, {sin(deg((0.716945)/0.16))+1*(0.716945)^2});
        \coordinate (p4) at (1.66914, {sin(deg((1.66914)/0.16))+1*(1.66914)^2});

        \addplot[only marks, mark=square*, black] coordinates {(-1.19391, {sin(deg((-1.19391)/0.16))+1*(-1.19391)^2})};
        \addplot[only marks, mark=square*, black] coordinates {(-0.239075, {sin(deg((-0.239075)/0.16))+1*(-0.239075)^2})};
        \addplot[only marks, mark=square*, black] coordinates {(0.716945, {sin(deg((0.716945)/0.16))+1*(0.716945)^2})};
        \addplot[only marks, mark=square*, black] coordinates {(1.66914, {sin(deg((1.66914)/0.16))+1*(1.66914)^2})};

        \coordinate (q1) at (-0.795146, {sin(deg((-0.795146)/0.16))+1*(-0.795146)^2});
        \coordinate (q2) at (0.264907, {sin(deg((0.264907)/0.16))+1*(0.264907)^2});
        \coordinate (q3) at (1.3268, {sin(deg((1.3268)/0.16))+1*(1.3268)^2});

        \addplot[only marks, mark=*, black] coordinates {(-0.795146, {sin(deg((-0.795146)/0.16))+1*(-0.795146)^2})};
        \addplot[only marks, mark=*, black] coordinates {(0.264907, {sin(deg((0.264907)/0.16))+1*(0.264907)^2})};
        \addplot[only marks, mark=*, black] coordinates {(1.3268, {sin(deg((1.3268)/0.16))+1*(1.3268)^2})};
      
        \node at (axis cs:-1.6, 4) {\( E_\ep \)};

        \draw[->, black, thick, dotted] (p1) -- (axis cs:-1.19391, -3);  
        \draw[->, black, thick, dotted] (p2) -- (axis cs:-0.239075, -3);  
         \draw[->, black, thick, dotted] (p4) -- (axis cs:1.66914, -3);      

        \node at (axis cs:-0.239075, -3.5) {\( \pt_1 u =\avg{f}\)};
        \node at (axis cs:-1.19391-0.1, -3.5) {Minimal supersolution to};
        \node at (axis cs:-1.19391-0.1, -3.5-0.7) {\( \pt_1 u \le L^*(\gd' u)\)};
        \node at (axis cs:1.5+0.2, -3.5) {Maximal subsolution to};
         \node at (axis cs:1.5+0.2, -3.5-0.7) {\( \pt_1 u \ge L_*(\gd' u)\)};
        \node at (axis cs:.3, -6-0.1) {\( L_\ast(\gd' u)\le \pt_1 u \le L^\ast(\gd' u)\)};
        \node at (axis cs:-1.19391, 1.5) {\(u_\textup{min}^\ep\)};
        \node at (axis cs:1.66914, 3) {\( u_\textup{max}^\ep\)};
        \node at (axis cs:-0.239075, 0) {\(u_\textup{glb}^\ep\)};

        \draw[decorate,decoration={brace,mirror,amplitude=15pt}] 
            (axis cs:-1.5,-4.5-0.2) -- (axis cs:2,-4.5-0.2);
        
    \end{axis}
\end{tikzpicture}
 \captionsetup{width=0.7\textwidth}
    \caption{This figure formally illustrates the critical points of the energy $E_\ep$ as defined in \eqref{eq.theenergyepsilonforlaminarcase} and their homogenization. On the left and right are the extremal solutions, and they homogenize exactly to the extremal solutions of $L_\ast(\gd' u)\le \pt_1 u \le L^\ast(\gd' u)$. The global energy minimizers $u_\textup{glb}^\ep$ homogenizes exactly to the standard Neumann problem $\pt_1 u = \avg{f}$.}
    \label{fig:illustratecriticalpoints}
\end{figure}
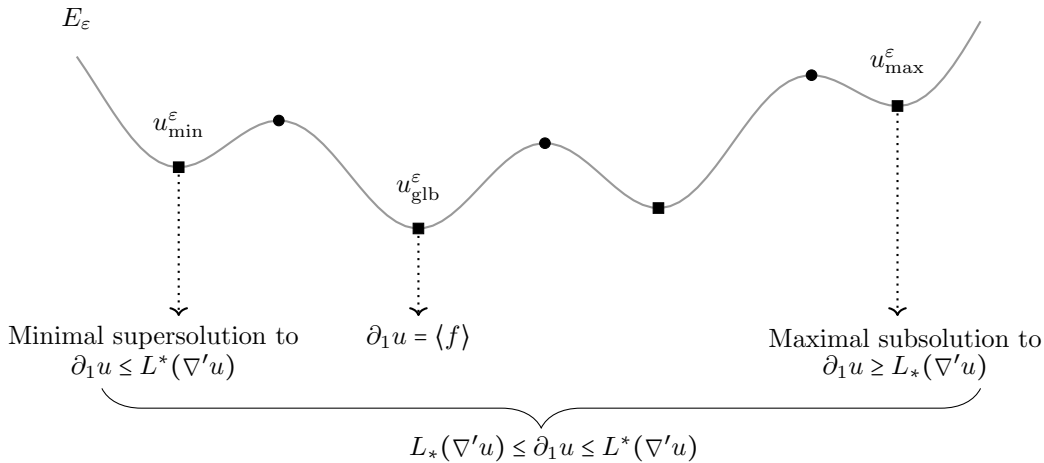

One essential question, which we alluded to already, is: are the extremal solutions of \eqref{eq.generalhomogenizedequidefequiform} actually distinct from the energy minimizing solution satisfying the standard Neumann boundary condition $\pt_1 u = \avg{f}$? We will establish that, indeed, these solutions are distinct for many choices of Dirichlet boundary datum. In fact, this phenomenon also demonstrates the existence of the degenerate regions, which we call free {facets/contact sets}
\[
\Ca(u):=\{\pt_1 u \ne \avg{f}\}\subset \{\gd' u = 0\}\subset B_1'.
\]
See \eqref{eq.definitionforCa} for a precise definition of $\Ca(u)$ in terms of viscosity solution theory.  We prove the following result in Section \ref{section.specialcases}.
\begin{proposition}\label{prop.generalexistenceoffacets}
 There is a non-empty set of boundary data $ \mathcal{F}\subset C(\pt B_1\cap \{x_1\ge0\})$, open in the uniform topology, such that, for all $g\in \mathcal{F}$, the unique minimal supersolution $u_g$ to \eqref{eq.generalhomogenizedequidef} with $u_g=g$ on $\pt B_1\cap \{x_1\ge0\}$ has nontrivial relatively open contact set $\Ca(u_g) \neq \emptyset$.
 
\end{proposition}

{Note that, in combination with the homogenization result for extremal solutions in Theorem \ref{t.homogenizationofspecialsolutions}, this shows that $u_{\max}^\ep$, $u_{\min}^\ep$, and $u^\ep_{\textup{glb}}$ are all distinct for a fixed boundary data $g$ as in Proposition \ref{prop.generalexistenceoffacets}.  In other words our result allows to understand the non-uniqueness of the $\ep$ problem \eqref{eq.homogenizationNeumann} in terms of the non-uniqueness of the homogenized problem \eqref{eq.generalhomogenizedequidef}, which is often more accessible.}

We also show in the same section that the homogenized boundary condition in the parabolic evolution \eqref{eq.pGDwithmotionlaw} is, also, non-trivially distinct from the standard Neumann boundary condition. Specifically we show that when the Dirichlet boundary data satisfies a certain strong monotonicity property in time, the solution to the homogenized problem \eqref{eq.pGDwithmotionlaw} will converge to the extremal steady states as time goes to infinity. As just discussed, these extremal steady states are not, in general, solutions to the standard Neumann problem. See Theorem \ref{t.relationbetweenextremsteadystatesandparabolicflow} for the precise description and proof.

\subsection{Pinning, contact set and rate-independent system}

The parabolic homogenization in Theorem \ref{t.homogenizationofparabolicformal} shows the emergence of some novel phenomena from a macroscopic viewpoint:
\begin{enumerate}
    \item \emph{Pinning and rate-independent motion}: Formally there is no motion when the slope 
    $$\pt_1 u \in (L_*(\gd' u),L^*(\gd' u)),$$ 
    that is, when 
    \[
    \gd'u=0 \textup{ and } \pt_1u \in (\min f, \max f).
    \]
    When the interface moves, it follows a rate-independent principle: the state of the contact slope $\pt_1 u$ does not depend on the magnitude of $\pt_t u$, i.e. the rate of motion. Note that the initial problem \eqref{eq.pNNep} has no intrinsic frictional hysteresis, this phenomenon arises in the homogenizaiton limit.
    \item \emph{Singular anisotropy and free contact set}: The free contact set
    \[
    \mathcal{C}(u) = \{\pt_1 u \ne \avg{f}\} \subset \{\gd' u =0\},
    \]
    plays a central role in the homogenized problem. As described in \cite{feldman2024regularitytheorygradientdegenerate} this free region is related to and generalizes the role of the contact set in the thin obstacle problem. Thus a thin free boundary arises in the limiting homogenized problems \eqref{eq.pGDwithmotionlaw} and \eqref{eq.generalhomogenizedequidef}.
 
\end{enumerate}

\subsection{Literature}\label{subsec.literature}

Although our techniques in this paper are primarily based on comparison principle, the problem \eqref{eq.pNNep} has a natural gradient flow structure. In fact the phenomenon of rate independent hysteresis is most commonly studied in the energetic context, see the book \cites{MielkeRoubicek} for more background and references. In order to explain further this connection,  we introduce some of the underlying notations and ideas from the theory of gradient systems. We will be somewhat imprecise, in particular we will not detail the functional spaces or the forcing by boundary data. General gradient systems take the form
\[0 \in \partial_{\dot{u}}R(u(t),\dot{u}(t)) + \partial_u E(t,u(t))+\partial_tE(t,u(t))\]
where $E$ is the energy functional and $R$ is the dissipation rate functional. In the energetic formulation of \eqref{eq.pNNep} the energy is $E_\ep$, as defined in \eqref{eq.theenergyepsilonforlaminarcase}, and the dissipation takes the following $L^2$-form
\[
R_\ep(\pt_t u) := \frac{1}{2}\int_{B_1^+} |\pt_t{u}|^2 \ dx.
\]
The homogenized system \eqref{eq.pGDwithmotionlaw} also has a formal energetic structure.  The homogenized energy $E_*$ takes the form \eqref{eq.homogenizedenergy}, simply the usual $\Gamma$-limit of the $E_\ep$. However something much more subtle happens in the dissipation rate functional, formally speaking the dissipation rate functional is given by
\[
R(\pt_t u;\gd' u) = \frac{1}{2}\int_{B_1^+} |\pt_t{u}|^2 \ dx + \int_{B_1'}  \Ra(\pt_t u(x') ; \gd' u(x')) dx',
\]
where $\Ra(\tau;p)$ is the homogenized quantity defined in \eqref{eq.thedissipationcoefficient}.  The appearance of the $L^1$-type term in the dissipation rate is indicative of rate independent pinning. The origin of such rate independent pinning terms from homogenization of wiggly energies has been expected via analysis of simple ODE models since \cites{james1996hysteresis,james1996wiggly,mielke1999mathematical} and even earlier works on dry friction. In fact this microscopic origin of macroscopic hysteresis is one of the motivations for the theory of gradient systems with rate independent dissipations. What is very unusual in this problem is the singularly anisotropic dependence on the tangential gradient in the rate functional. 

Various energetic convergence theories for gradient systems, including with rate independent / $L^1$ type dissipations, have been established in the literature \cites{SandierSerfaty,Mielke2012,LieroMielkePeletierRenger2017}. However the rigorous passage to the limit from wiggly energy microscopic model to macroscopic model with $L^1$-type dissipation rate has only been fully addressed in one-dimensional ODE models \cites{Mielke2012,Bonaschi2016,dondl2019gradient}.  We establish such a limit theorem in a multi-dimensional PDE setting for the first time. Our techniques, however, do not strongly use the formal energetic structure of the limit problem. In fact, due to the discontinuity of $\Ra(\tau ; p)$ in $p$, we have been unable to rigorously establish an energetic formulation of the limit problem \eqref{eq.pGDwithmotionlaw}.  This is one of the main reasons we don't use energetic techniques to analyze the homogenization limit.  We are very interested whether it is possible to give a rigorous energetic interpretation of the homogenized system \eqref{eq.pGDwithmotionlaw} and study the homogenization limit using evolutionary $\Gamma$-convergence techniques.

As previously mentioned, a central element in our proof of Theorem~\ref{t.homogenizationofparabolicformal} is the comparison principle (Theorem~\ref{t.comparisonprincipleformal}) for solutions of~\eqref{eq.pGDwithmotionlaw}. Similar homogenization challenges arise in other models studied in the literature. For instance, in the context of anisotropic Bernoulli-type problems, elliptic comparison principles were established in~\cites{feldman2019free,feldman2021limit,feldman2024convex}, enabling homogenization results for inhomogeneous one-phase Bernoulli problems~\cites{feldman2021limit,caffarelli2007capillary,caffarelli2007homogenization,caffarellimellet}. These results are in elliptic / stationary settings.  In~\cites{CaffarelliLeeMellet,caffarelli2007flame}, Caffarelli, Lee and Mellet analyzed a reaction-diffusion equation modeling flame propagation, which, under a certain singular limit, converges to the parabolic version of the Bernoulli problem also known as the flame propagation free boundary problem. They consider several scalings, but in the case where the interface width is thinner than the heterogeneities and pinning occurs their work applies only to the one-dimensional case. The challenges in these studies are closely analogous to those in our parabolic model, and more direct connections can be drawn through limiting or linearization procedures (see~\cite{feldman2024regularitytheorygradientdegenerate}).  In comparison to these previous results, this paper is the first to handle the parabolic setting in multiple dimensions. The comparison principle Theorem \ref{t.comparisonprincipleformal} is really the key new tool that allows for the homogenization result in higher dimensions.

As a PDE problem, the homogenized equation \eqref{eq.generalhomogenizedequidefequiform} can be viewed as an  elliptic PDE with gradient degeneracy. Such problems have attracted interest, especially spurred by a result of Silvestre and Imbert \cites{imbert2013c1alpha}, with further developments for nonlocal PDE in \cites{araújo2023fractional,prazeres2021interior} and others. 
There is also a connection with the Signorini or thin obstacle problem, which was described in more detail in the previous work of the authors \cite{feldman2024regularitytheorygradientdegenerate}. The contact set in the elliptic problem \eqref{eq.generalhomogenizedequidefequiform} is closely related to the contact set in the thin obstacle problem. It is a kind of ``unconstrained" analogue of the thin obstacle problem in the language of \cite{figalli2018overviewunconstrainedfreeboundary}. This connection was exploited to study the regularity of solutions to \eqref{eq.generalhomogenizedequidefequiform} in \cites{feldman2024regularitytheorygradientdegenerate}.  
In this work we have derived \eqref{eq.generalhomogenizedequidefequiform} and its parabolic analogue \eqref{eq.pGDwithmotionlaw} via a natural homogenization procedure, and this invites further possible investigation into connections with thin obstacle problems and gradient degenerate elliptic problems.

 Pinning and the hysteresis caused by pinning are central important phenomenon in the study of propagation of interfaces in heterogeneous media. Among others this includes models of capillary contact lines, domain boundaries in random magnetic materials, adhesion of thin films, and others. There have been many works studying pinning and de-pinning of interface motions in heterogeneous media, to name a few examples in slightly related PDE models \cites{barles2011homogenization,dirr2006pinning,dirr2011pinning,courte2021proof,dondl2017ballistic,feldman2021limit,kardar1998nonequilibrium,MR3479884,MR3511342,Xia2015,kim2008,CaffarelliLeeMellet}. Generally speaking the study of stationary (pinned) interfaces in heterogeneous media can be more challenging than that of moving (de-pinned) interfaces. When the front is de-pinned the interface moves through the medium ``seeing" the entire medium, and the large-scale averaging behavior is more accessible.  The additional challenges in our work compared to much of the literature discussed above are (1) the large scale forcing $\partial_1u$ depends nonlocally on the shape, especially the tangential gradient $\gd'u$ of the interface causing the singular anisotropy of the limiting PDE in contrast to works like \cites{dirr2006pinning,dirr2011pinning,courte2021proof,dondl2017ballistic} where the forcing is an external constant. This nonlocal geometric dependence is also one of the reasons for the existence of facets/contact sets that specify the free regions having distinct pinning phenomena; (2) the boundary condition is both quasistatic and rate independent so the comparison principle and uniqueness present a new challenge as compared to finite velocity models as in \cite{kim2008}.

\subsection{Organization of the paper}
In Section \ref{section.preliminaries} we introduce some notations, viscosity touching and crossing and the notions of half relaxed limits. In Section \ref{section.generalsteadystates} we start with a $\Gamma$-convergence result for the homogenization of energy minimizers. Then we study the corrector problem.  We introduce the homogenized pinning interval and prove the existence of plane-like correctors satisfying the strong Birkhoff property. In Section \ref{section.extremalsteadystates} we show the homogenization of extremal solutions of the elliptic problem \eqref{eq.homogenizationNeumann}. In Section \ref{section.homogenizedparabolicflow} we make precise the notion of viscosity solution to the homogenized evolution \eqref{eq.pGDwithmotionlaw}. In Section \ref{section.homoparabolic} we prove the homogenization of \eqref{eq.pNNep} using the half-relaxed limit method. In Section \ref{section.comparisonprinciple} we prove the comparison principle of \eqref{eq.pGDwithmotionlaw}. Finally in Section \ref{section.specialcases} we show the existence of facets / contact sets for certain boundary conditions.

\subsection{Acknowledgments} This material is based upon work supported by the National Science Foundation under Award No. DMS-2407235, both authors were partially supported by this grant. The authors also acknowledge the research environment provided by the support of the NSF RTG grant Award No. 2136198.

\section{Preliminaries}\label{section.preliminaries}

\subsection{Notations}\label{subsection.notations}
\begin{itemize}
    \item If not particularly defined, $\alpha\in (0,1)$ will be a constant that represents the H\"{o}lder exponents that may change from line to line.
    \item For a $\Z^d$-periodic function $f=f(x',u)$ on $\R^{d-1}\times\R$, we denote the average as $\avg{f}:=\int_{(0,1]^{d-1}}\int_{0}^1 f(y',r) dr dy'$. In particular, if $f=f(u)$ we have $\avg{f}=\int_{0}^1 f(r) dr$.
    \item Define 
    \[
    \R_+^d:=\R^d \cap \{x_1>0\} \quad \textup{and} \quad \pt \R_+^d = \R^d \cap\{x_1=0\}.
    \]
    \item Suppose $\Omega\subset \R^d$ is open (or $\Omega\subset \overline{\R_+^d}$ is relatively open), we write 
    $$
    \Omega':= \Omega\cap \{x_1=0\}\quad \textup{and} \quad \Omega^+:=\{x_1>0\}.
    $$ 
    In particular, we write $B_1^+:=B_1\cap \{x_1>0\}$ and $B_1':=B_1\cap \{x_1=0\}$. We write the \say{exterior boundary} as 
    $$
    \pd \Omega=\pd \Omega^+:=\lb\pt \Omega^+\rb\setminus \Omega'.
    $$
    \item For a space-time cylindrical domain $$U=\Omega\times (t_1,t_2) \ (\textup{or }\Omega\times (t_1,t_2]) \subset \R^d\times \R$$ with $t_1<t_2$, we write 
    \[
    U':=\Omega'\times(t_1,t_2] \quad \textup{and} \quad  U^+:=\Omega^+\times(t_1,t_2],
    \]
    and we define the following two types of parabolic boundaries
    \[
    \pp U:= \overline{U}\setminus \lb \Omega\times(t_1,t_2]\rb \quad \textup{and} \quad  \ppp U=\ppp U^+:=\pp U^+ \setminus U'.
    \]
    \item We write $D_{T}:= B_1 \times (0,T]$ for $T\in(0,\infty)$ and $D_\infty:=B_1\times(0,\infty)$. In particular, we have
    \[
    D_{T}'=B_1' \times (0,T] \quad \textup{and} \quad  \ D_{T}^+= B_1^+\times (0,T].
    \]
    We also write
    \[
     D_{\infty}'=B_1' \times (0,\infty) \quad \textup{and} \quad  \ D_{\infty}^+= B_1^+\times (0,\infty).
    \]

\item For two sets $A$, $B$ we denote the symmetric difference
\[
A\Delta B:= (A\setminus B) \cup (B\setminus A).
\]
\item A function $u:\R^d\rta [-\infty,\infty)$ is upper semicontinuous if for every $x\in U$
\[
\limsup_{y\rta x}u(y)\le u(x).
\]
A lower semicontinuous function is defined symmetrically. We emphasize here that upper semicontinuous functions are allowed to take negative infinity value.
\end{itemize}

\subsection{Touching, crossing and half relaxed limits}\label{appendix.halfrelaxedlimits}

We introduce here the notion of {touching} and {crossing}, and their behavior under {half relaxed limits}. Since our problem involves a boundary condition, we mainly consider in a relatively open domain $U\subset \R_+^d\cup \pt\R_+^d$, and denote $U^+=U\cap \R_+^d$, $U'=U\cap \pt\R_+^d$.

\begin{definition}\label{def.touching}
    We say that a smooth function $\psi$ \emph{touches} an upper semicontinuous function $u: \overline{U}\rta[-\infty,\infty)$ (strictly) from above at $x_0\in U$ if there is an open domain $V\subset \R^d$ that contains $x_0$ so that $u-\psi$ attains its (strict) maximal value 0 at $x_0$ in $V\cap U$. We say $\psi$ \emph{touches} a lower semicontinuous function $v$ (strictly) from below at $x_0$ if $-\psi$ touches $-v$ (strictly) from above at $x$.
\end{definition}

It is often more appropriate to consider the notion of {crossing}, or {parabolic touching}, especially in parabolic equations with non-proper zeroth order terms \cite{UsersGuide}. 
\begin{definition}\label{def.crossing}
    In a cylindrical domain $U\times(t_1,t_2)$, we say that a smooth function $\psi$ \emph{crosses} an upper semicontinuous function $u:\overline{U}\times [t_1,t_2]\rta[-\infty,\infty)$ (strictly) from above at $(x_0,t_0)\in U\times(t_1,t_2)$ if there is an open domain $V\subset \R^d$ that contains $x_0$ and a small $r>0$ so that $u-\psi$ attains its (strict) maximal value 0 at $(x_0,t_0)$ in the space-time domain $(V\cap U)\times(t_0-r,t_0]$. The \emph{crossing} from below is defined symmetrically for lower semicontinuous functions.
\end{definition}

For a sequence of upper semicontinuous functions $u_n$ that are bounded from above in a closed set $K\subset\R^d$, we define the {upper half relaxed limit}
\be
{\limsup_{n\rta\infty}}^* u_n(x) = \inf_{n\ge 1} \max_{|y-x|\le 1/n, y\in K} u_n(y), 
\ee
and symmetrically the {lower half relaxed limit} for a sequence of lower semicontinuous functions $v_n$ that are bounded from below in $K$
\be
\underset{n\to \infty}{{\liminf}_\ast} v_n(x) = \sup_{n\ge 1} \min_{|y-x|\le 1/n, y\in K} v_n(y).
\ee
The application of half relaxed limits in viscosity solution limit problems was first introduced by Barles and Perthame \cites{BarlesPerthame1,BarlesPerthame2}.

In the following we prove two technical lemmas about the perturbation properties of touching and crossing under the notion of half relaxed limits. 

\begin{lemma}\label{l.stabilityoftouching}
   Let $u^*$ be the upper half relaxed limit of $u_n$ in $\overline{U}$, then $u^*$ is upper semicontinuous. If $u^*$ reaches a strict maximum at $x_0$ in $B_r(x_0)\cap U$ for some radius $r>0$, then there is a subsequence $u_{n_j}$ of $u_n$ and a sequence of maximum points $x_{n_j}$ of $u_{n_j}$ in $B_r(x_0)\cap U$ that converge to $x_0$ as $j\rta\infty$, and 
   \[
   \lim_{j\rta \infty} u_{n_j}(x_{n_j}) = u^*(x_0).
   \]
   Similar results also hold for the lower half relaxed limits by symmetry.
\end{lemma}

\begin{proof}

See \cite{Silvestre2015}*{Lemma 3.5}.

\end{proof}

\begin{lemma}\label{l.stabilityofcrossing}
   Let $u^*$ be the upper half relaxed limit of $u_n$ in $\overline{U}\times[t_1,t_2]$. If $u^*$ reaches a strict maximum at $(x_0,t_0)$ in $\lb B_r(x_0)\cap U\rb\times(t_0-r,t_0]$ for $(x_0,t_0)\in U \times (t_1,t_2)$ and some radius $r>0$, then there is a subsequence $u_{n_j}$ of $u_n$ and a sequence of maximum points $(x_{n_j},t_{n_j})$ of $u_{n_j}$ in $\lb B_{r}(x_0)\cap U\rb\times(t_0-r,t_{n_j}]$ that converge to $(x_0,t_0)$ as $j\rta\infty$, and 
   \[
   \lim_{j\rta \infty} u_{n_j}(x_{n_j},t_{n_j}) = u^*(x_0,t_0).
   \]
   Similar results also hold for the lower half relaxed limits by symmetry.

\end{lemma}

\begin{proof}

We assume that $u^*(x_0,t_0)=0$. By definition of upper half relaxed limit, we can find a subsequence $u_n$, not relabeled, and a sequence of points $(y_n,s_n)\in U\times(t_1,t_2)$ such that
\[
\lim_{n\rta\infty}(y_n,s_n)=(x_0,t_0) \textup{ and } \lim_{n\rta\infty} u_n(y_n,s_n) = u^*(x_0,t_0).
\]
We define for each $n$, $(x_n,t_n)$ to be a maximum point of $u_n$ in
\[
\overline{\lb B_{r}(x_0)\cap {U}\rb}\times[t_0-r,s_n].
\]
By compactness, there is a limit point $(x_0',t_0')$ of the sequence $(x_n,t_n)$, and therefore
\[
u^*(x_0',t_0') \ge \limsup_{n\rta\infty} u_n(x_n,t_n) \ge \limsup_{n\rta\infty} u_n(y_n,s_n) =u^*(x_0,t_0)=0.
\]
This shows that $t_0'\ge t_0$ and because $t_n\le s_n\rta t_0$, we know that $t_0'=t_0$. Because $u^*$ is strictly negative in $\lb B_r(x_0)\cap U\rb\times(t_0-r,t_0]$ except for $(x_0,t_0)$, we know that $x_0'=x_0$. This shows that $(x_n,t_n)$ has to converge to $(x_0,t_0)$ and the proof is complete.

\end{proof}

\section{General steady states}\label{section.generalsteadystates}

 In this section we study the homogenization of the following more general steady state problem
\be\label{eq.homogenizationNeumanngeneral1}
\bca
\Delta u^\ep = 0 & \textup{ in }B_1^+\\
\pt_1 u^\ep = f(x'/\ep,u^\ep/\ep) & \textup{ on }B_1'.
\eca
\ee
Here $f(y,z)$ belongs to $C^\alpha(\R^{d-1}\times\R)$ for some $\alpha>0$ and is $\Z^{d-1} \times \Z$-periodic. We remark here that although we include more general inhomogeneity in the steady state problem, the parabolic homogenization remains open due to the lack of comparison principles similar to Theorem \ref{t.comparisonprincipleformal}.  

\begin{remark}
    There are two main types of solutions to \eqref{eq.homogenizationNeumanngeneral1}: distributional weak solutions and viscosity solutions. Under the assumption $f\in C^\alpha(\R^{d-1}\times\R)$ the two notions are equivalent and are both classical solutions by applying the regularity results in \cites{Fern2022,milakis2006regularity,nittka2011regularity} (See also Appendix \ref{appendix.someestimates}).
\end{remark}

First, in Section~\ref{subsec.globalenergyminimizer}, we show the homogenization of energy minimizers by a $\Gamma$-convergence type argument. Then we consider the homogenization of general solutions of the PDE. This requires the introduction and classification of the plane-like correctors and the pinning interval, which is covered in Section \ref{subsec.planelikecorrector} and Section~\ref{subsec.homogenizationofgeneralvissol}. In the special case of laminar media the pinning interval takes a particularly simple form, and can be exactly computed, we carry this computation out in Section~\ref{subsec.laminarcasehomogenization}.

\subsection{The global energy minimizers}\label{subsec.globalenergyminimizer}

In this subsection we consider the energy minimizing solutions of \eqref{eq.homogenizationNeumanngeneral1}. Specifically, consider the energy functional
\be\label{eq.theenergyepsilon}
E_\ep(u,B_1^+):= \int_{B_{1}^+} \frac{1}{2}|\gd u|^2 + \int_{B_1'} \int_0^{u(x')} f_\ep (x',r) dr dx',
\ee
where $f_\ep(x',r)=f(x'/\ep,r/\ep)$ is a $\Z^d$-periodic and H\"older continuous function.

\begin{lemma}\label{l.localuniformconvergenceofglobalminimizers}
As $\ep\rta0$ the energy $E_\ep$ subject to the boundary constraint $u=g$ on $\pd B_1$, $\Gamma$-converges (See definitions in \cite{braides2006handbook}) to
\[
E_0(u,B_1^+):=\int_{B_{1}^+} \frac{1}{2}|\gd u|^2 + \int_{B_1'} \avg{f} u(x') dx',
\]
where $\avg{f}:=\int_{[0,1)^{d-1}}\int_0^1 f(y',r) dr dy'$. Moreover, if $g$ is continuous then the corresponding {global energy minimizers} $u^\ep$ converge uniformly on $\overline{B_1^+}$ to the problem
\be
\bca
\Delta u = 0 & \textup{ in }B_1^+\\
u=g& \textup{ on }\pd B_1\\
\pt_1 u = \langle f \rangle&\textup{ on }B_1'.
\eca
\ee

\end{lemma}

\begin{proof}
    Let us first show that $E_\ep$ $\Gamma$-converges to $E_0$ as $\ep \rta 0$ in the weak topology of $H_g^1(B_1^+)$, which is the subspace of $H^1(B_1^+)$ that have trace $g$ on $$\pt^+B_1=\pt B_1\cap \{x_1\ge0\}.$$
    To see this we consider the lower floor function modulo $\ep>0$, $\lfloor {a}\rfloor_\ep:= \sup\{z\in \ep\Z \ ; \ z\le a\}$. Let $u_\ep \weakcv u$ in $H^1(B_1^+)$ then
    \be\label{eq.toshowthegamaconvergencewenneedsomting}
    \begin{split}
         E_\ep(u_\ep) &= \int_{B_1^+} \frac{1}{2}|\gd u_\ep|^2 + \int_{B_1'} \int_0^{u_\ep(x')} f_\ep(x',r)dr dx'\\
    &=\int_{B_1^+}  \frac{1}{2}|\gd u_\ep|^2 + \int_{B_1'} \int_0^{u_\ep(x')-\lfloor u_\ep(x')\rfloor_\ep} f_\ep(x',r)dr + \int_0^1f(x'/\ep,v) dv \lfloor u_\ep(x')\rfloor_\ep dx'\\
    &=\int_{B_1^+}  \frac{1}{2}|\gd u_\ep|^2 + \int_{B_1'} \int_0^1f(x'/\ep,v) dv\, u_\ep(x') dx' + O(\ep).
    \end{split}
     \ee
On the other hand, let $v_\ep$ minimize the following energy on $H^1(B_1^+)$ with $\int_{B_1^+} v_\ep (x) dx =0$
\[
\Tilde{E}_\ep(v,B_1^+):= \int_{B_{1}^+} \frac{1}{2}|\gd v|^2 + \int_{B_1'}v(x') \int_0^{1} f (x'/\ep,r) dr  dx'.
\]
Then $v_\ep$ solves
\[
\bca
\Delta v_\ep = 0&\textup{ in }B_1^+\\
\pt_{\Vec{n}} v_\ep = 0& \textup{ on }\pd B_1^+\\
\pt_1 v_\ep = \int_0^1f (x'/\ep,r)dr& \textup{ on }B_1',
\eca
\]
where $\pt_{\Vec{n}}$ is the inner normal derivative $\pd B_1^+$. Because $f$ is a bounded function, we have by standard elliptic theory
\[
\sup_{\ep>0} \norm{v_\ep}_{H^1(B_1^+)}<\infty.
\]
This shows that, by Banach-Alaoglu theorem, $v_\ep$ weakly converge to the {unique} $v$ that satisfies $\int_{B_1^+} v(x)dx=0$ and solves
\be\label{eq.v0}
\bca
\Delta v = 0&\textup{ in }B_1^+\\
\pt_{\Vec{n}} v = 0& \textup{ on }\pd B_1^+\\
\pt_1 v = \avg{f}& \textup{ on }B_1'.
\eca
\ee
Moreover, by applying the H\"{o}lder estimate in \cite{nittka2011regularity}*{Proposition 3.6}, we also have
\be\label{eq.uniformholderofvep}
\sup_{\ep>0} \norm{v_\ep}_{C^\alpha(B_1^+)}<\infty,
\ee
which, by applying Arzela-Ascoli theorem, shows that $v_\ep$ converge uniformly on $\overline{B_1^+}$.

Now, substituting $v_\ep$ into \eqref{eq.toshowthegamaconvergencewenneedsomting} we obtain
\begin{align}
         E_\ep(u_\ep)&=\int_{B_1^+}  \frac{1}{2}|\gd u_\ep|^2 + \int_{B_1'} \int_0^1f(x'/\ep,r) dr\, u_\ep(x') dx' + O(\ep)\notag\\
    &=\int_{B_1^+}  \frac{1}{2}|\gd u_\ep|^2 + \int_{B_1'} \pt_1 v_\ep(x')\, u_\ep(x') dx' + O(\ep)\notag\\
    &=\int_{B_1^+}  \frac{1}{2}|\gd u_\ep|^2 - \int_{B_1^+} \gd v_\ep(x)\cdot \gd u_\ep(x) dx + O(\ep)\notag\\
    &=\int_{B_1^+}  \frac{1}{2}|\gd (u_\ep-v_\ep)|^2 - \frac{1}{2}\int_{B_1^+}|\gd v_\ep|^2 + O(\ep)\notag\\
    &=\int_{B_1^+}  \frac{1}{2}|\gd (u_\ep-v_\ep)|^2 + \frac{1}{2}\int_{B_1'} \int_0^1f(x'/\ep,r) dr\, v_\ep(x') dx' + O(\ep)\label{e.energyhom-align}\\
    \end{align}
After sending $\ep\rta 0$, we obtain, recalling $u$, $v$ are the weak limits of $u_\ep$, $v_\ep$,
   \[
   \int_{B_1^+}  \frac{1}{2}|\gd (u-v)|^2 \le \liminf_{\ep\rta0} \int_{B_1^+}  \frac{1}{2}|\gd (u_\ep-v_\ep)|^2 .
   \]
The third term on the right of \eref{energyhom-align} converges to zero. Lastly consider the second term on the right of \eref{energyhom-align}.  Note that $\int_0^1f(x'/\ep,r) dr$ weakly converges to $\avg{f}$ in $L^2(B_1')$. Combining  \eqref{eq.uniformholderofvep} we have
   \[
  \frac{1}{2}\int_{B_1'} \pt_1 v_\ep(x')\, v_\ep(x') dx' \to  \frac{1}{2}\int_{B_1'} \avg{f}\, v(x') dx' \ \hbox{ as } \ \ep \to 0.
   \]
 We mention that quantitative homogenization results of $v_\ep \weakcv v$ is known in \cite{shen2018boundary}. The above arguments show that
     \[
     \begin{split}
         \liminf_{\ep\rta0} E_\ep(u_\ep) & \ge  \int_{B_1^+}  \frac{1}{2}|\gd (u-v)|^2+\frac{1}{2}\int_{B_1'} \avg{f}\, v(x') dx' \\
         &=E_0(u).
     \end{split}
     \]
     For every $u\in H_g^1(B_1^+)$ we always have the recovery sequence $\hat{u}_\ep=u$ and 
     \[
     E_0(u) = \lim_{\ep\rta0} E_\ep(u).
     \]
By the standard theory of $\Gamma$-convergence, the global minimizers of $E_\ep$ weakly converge to the global minimizers of $E_0$. 

To show the uniform convergence of $u^\ep$ on $\overline{B_1^+}$ we consider the classical solution $h$ satisfying
\[
\bca
\Delta h = 0 & \textup{ in } B_1^+\\
h=g & \textup{ on }\pd B_1^+\\
\pt_1 h=0 & \textup{ on }B_1'.
\eca
\]
The existence of $h\in C(\overline{B_1^+})\cap C^{2,\alpha}(B_1^+\cup B_1')$ is obtained by using Perron's method and standard elliptic regularity theory. Now $u^\ep-h$ satisfies the conditions in Lemma \ref{l.usefulholderestimate} and $u^\ep-h=0$ on $\pd B_1^+$, therefore $u^\ep-h$ has uniform bounded $C^\alpha(B_1^+)$-norm, which implies that $u^\ep$ converges uniformly on $\overline{B_1^+}$ by Arzela-Ascoli theorem.

\end{proof}

\begin{remark}
   The above proof of $\Gamma$-convergence does not seem to work in the case that $f$ is periodic with respect to a general lattice on $\R^{d-1} \times \R$. However, by using the estimate in Lemma \ref{l.usefulholderestimate}, one can still show a similar homogenization result for the global minimizers.
\end{remark}

\subsection{Plane-like correctors}\label{subsec.planelikecorrector}
The global minimizer theory does not capture the homogenization of all the local minimizers or steady states nor does it capture the macroscopic evolution in parabolic flow. Instead of a single large-scale slope specified by the average value of $f$, there is an interval of pinned slopes. The pinning interval is described via a (global) corrector problem which we describe next.

\begin{definition}\label{def.corrector}
Call $v \in C(\R^d_+\cup \pt \R_+^d)$ a \emph{corrector} if it solves
\be\label{eq.correctorforgeneralsemilinearproblem}
\bca
\Delta v = 0 &\textup{ in }\R_+^{d}\\
\pt_1 v = f(x',v) &\textup{ on }\R_+^d,
\eca
\ee
and 
\[
\sup_{\R_+^d} |v(x)-(\mu,p)\cdot x|<\infty,
\]
for some $(\mu,p)\in \R\times \R^{d-1}$. We call $(\mu,p)$ the \emph{effective slope} of $v$.

\end{definition}

Correctors are one of the fundamental concepts in the theory of homogenization. In classical homogenization theory the corrector equation is invariant with respect to ``vertical" translations -- i.e. adding a constant. In interface problems, such as we consider here, the vertical translation invariance is lost and many new challenges arise. Such challenges have been encountered and addressed before in several related models on homogenization of moving interfaces \cites{caffarelli_delallave_2001,moser_1986,caffarelli_2013,junginger_2009,rabinowitz_stredulinsky_2011,dirr2006pinning,DIRR_KARALI_YIP_2008}. The techniques and ideas trace back to the the fundamental contributions of Aubry and LeDaeron \cite{AUBRY1983381} and Mather \cite{MATHER1982457}, which is now often called Aubry-Mather theory.

It actually fits better with the philosophy of viscosity solutions to split the notion of corrector into two, a subsolution and a supersolution notion.

\begin{definition}\label{def.supercorrector}
    Call $v \in \textup{LSC}(\R^d_+ \cup \partial \R^d_+;(-\infty,+\infty])$ a \emph{supercorrector} with effective slope $(\mu,p) \in \R \times \R^{d-1}$ if
    \begin{enumerate}[label = (\roman*)]
        \item $v$ is a viscosity supersolution to \eqref{eq.correctorforgeneralsemilinearproblem}
        \item There exists $C>0$ so that
        \[v\ge (\mu,p)\cdot x-C \ \hbox{ in } \ \R^d_+ \cup \partial \R^d_+.\]
        \item $v(0) < + \infty$.
    \end{enumerate}
    We also call the triple $(v,\mu,p)$ a supercorrector.
   
\end{definition}
\begin{definition}\label{def.subcorrector}
    Call $v \in \textup{USC}(\R^d_+ \cup \partial \R^d_+;[-\infty,+\infty))$ a \emph{subcorrector} with effective slope $(\mu,p) \in \R \times \R^{d-1}$ if
    \begin{enumerate}[label = (\roman*)]
        \item $v$ is a viscosity subsolution to \eqref{eq.correctorforgeneralsemilinearproblem}
        \item There exists $C>0$ so that
        \[v\le (\mu,p)\cdot x+C \ \hbox{ in } \ \R^d_+ \cup \partial \R^d_+.\]
        \item $v(0) > - \infty$.
    \end{enumerate}
    We also call the triple $(v,\mu,p)$ a subcorrector.
\end{definition}

\begin{remark}
    These definitions are exactly suited to the perturbed test function type argument for homogenization, found below in Theorem~\ref{t.homogenizationofgeneralsemilinear}. The finiteness at the origin and the upper / lower linear bounds will come naturally from the touching test function.
\end{remark}

We define the homogenized coefficients
\be\label{eq.subcoefficient}
Q_\ast(p;f):=\inf \{\mu_+ \ ; (u,\mu_+, p) \textup{ is a subcorrector to \eqref{eq.correctorforgeneralsemilinearproblem}}\},
\ee
and
\be\label{eq.supercoefficient}
Q^\ast(p;f):=\sup \{\mu_- \ ; (u,\mu_-, p) \textup{ is a supercorrector to \eqref{eq.correctorforgeneralsemilinearproblem}}\}.
\ee
We will usually drop the dependence on $f$ and write $Q^*(p)$ or $Q_*(p)$. The interval $[Q_\ast(p),Q^\ast(p)]$ is called the {pinning interval}.

Let us make a couple of remarks about invariance properties of these definitions.

\begin{remark}\label{r.translationinvarianceofhomogenizedcoefficients}
    We remark here that the homogenized coefficients $Q_\ast(p;f)$ and $Q^\ast(p;f)$ are invariant under the translation $f(x',v)\mapsto f(x'+n',v+n_1)$ for some vector $(n',n_1)\in \R^{d-1}\times\R$. This is because given a subcorrector $(u,\mu,p)$ that corresponds to $f$, we can define $v(x_1,x')=u(x_1,x'+n')-n_1$ and obtain $(v,\mu,p)$ a subcorrector that corresponds to $f(x'+n',v+n_1)$. The supercorrector case is symmetrical.
\end{remark}

\begin{remark}\label{r.normalizeaveragezero}
    The homogenized coefficients satisfy $Q^\ast(p;f - a) = Q^\ast(p;f)-a$ and $Q_\ast(p;f - a) = Q_\ast(p;f)-a$ for any constant $a$ because there is a 1-1 correspondence between the semi-correctors corresponding to $f$ and those corresponding to $f-a$ by subtracting $ax_1$. This fact can be used to normalize $\avg{f}=0$, as we will do for convenience later.
\end{remark}

Our main result of this section establishes some important properties of the pinning interval and shows the existence of correctors. The correctors, by construction, will satisfy a certain periodicity property.
\begin{definition}\label{def.birkhoffproperty}
  We call a function $v$ defined on the strip $\{0\le x_1 \le T\}$ for some $T\in (0,\infty]$ to satisfy the \emph{Birkhoff property} if for any $(k,s_1),(k,s_2)\in \Z^{d-1}\times \Z$ such that $k\cdot p - s_1 \le 0\le k\cdot p-s_2$, we always have 
\[
v(x+k)-s_1 \le v(x) \le v(x+k)-s_2.
\]
In particular, we have for any $k\in \R^{d-1}$
\be\label{eq.birkhoffproperty}
v(x+k)-\lceil k\cdot p \rceil \le v(x) \le v(x+k)-\lfloor k\cdot p\rfloor.
\ee
\end{definition}

\begin{theorem}\label{t.characterizecorrectors}
The functions $Q_\ast$ and $Q^\ast$ as defined in \eqref{eq.subcoefficient} and \eqref{eq.supercoefficient} are lower and upper semicontinuous respectively on $\R^{d-1}$, and for all $p\in \R^{d-1}$
\[
\min f \leq Q_\ast(p)\le \avg{f} \le Q^\ast(p)\leq \max f.
\]
Moreover, for any $\mu\in [Q_\ast(p),Q^\ast(p)]$ , there exists a corrector $(u,\mu,p)$ such that $u$ also satisfies the {Birkhoff property} \eqref{eq.birkhoffproperty} and there is a constant $C>0$ so that
\be
\sup_{x_1\ge0} |u(x)- \mu x_1 - p\cdot x'| \le C,
\ee
where $C$ depends only on dimension $d$ and $f$.
\end{theorem}

We decompose the proof of Theorem \ref{t.characterizecorrectors} into four technical lemmas. We first establish the following bound for the homogenized coefficients $Q_\ast$ and $Q^\ast$.
\begin{lemma}\label{l.boundforthecoefficientshowingnontrivialness} For every tangent vector $p\in \R^{d-1}$ both $Q_*(p)$ and $  Q^*(p)$ are contained in the interval $[\min f , \max f].$
\end{lemma}

\begin{proof}
Note that,  $p\cdot x +(\max f) x_1$ is a subcorrector and $p\cdot x +(\min f) x_1$ is a supercorrector. Thus the sets in the definitions \eqref{eq.subcoefficient} and \eqref{eq.supercoefficient} are non-trivial, and the upper bound $ Q_*(p) \le \max f$ and the lower bound $\min f\le Q^*(p)$ are established.
    
For the other side we only show that $Q^\ast(p)\le \max f$ as the case for $Q_\ast$ is symmetrical. Suppose there is a supercorrector $(v,\mu,p)$ such that $\mu>\max f$. We pick $T>0$ large and $\mu-\max f\gg\eta>0$ small, to be specified below, and consider for some real number $s$ the following auxiliary harmonic function
\[
\phi_s(x_1,x'):= (\max f+\eta)x_1+p\cdot x' -\eta|x'|^2 + \eta(d-1)x_1^2 + s+v(0).
\]
Notice that $v$ is a supercorrector so for some constant $C$ we have
\[
v(x_1,x') \ge \mu x_1 + p\cdot x'+C.
\]
Choosing $T>0$ large enough, we have
\[
(\max f)T + p\cdot x' + v(0)< \mu T + p\cdot x'+C \le v(T,x').
\]
This implies that for small $\eta>0$ 
\[
\phi_0(T,x')=(\max f+\eta)T+ \eta(d-1)T^2 +p\cdot x'-\eta|x'|^2  +v(0)<v(T,x').
\]
Also because $v$ is lower bounded by $\mu x_1 + p\cdot x'+C$, there is a large radius $R>0$ independent of $s$ such that
$$
\phi_s(x_1,x')\le \phi_0(x_1,x')<v(x_1,x')
$$ 
for all $0\le x_1 \le T$, $s\le 0$ and $|x'|\ge R$. Let $s_\ast$ be the supremum of $s$ such that $\phi_s<v$ on $\{0\le x_1\le T,|x'|\le R\}$, then $s_\ast\le 0$ because $\phi_0(0)=v(0)$ and $s_\ast>-\infty$ because $v$ is lower bounded in $\{0\le x_1\le T,|x'|\le R\}$. On the other hand, by the harmonicity of $\phi_{s_\ast}$ and superharmonicity of $v$, the function $\phi_{s_\ast}$ will touch $v$ from below at a point $x_0'\in \{x_1=0, |x'|\le R\}$. This establishes a contradiction to the supercorrector condition of $v$ at $x=x_0'$ as $\pt_1\phi_{s_\ast}(x_0')>\max f$.
    
\end{proof}

\begin{lemma}\label{l.semicontinuityofcoefficients}
    The functions $Q_\ast$ and $Q^\ast$ as defined in \eqref{eq.subcoefficient} and \eqref{eq.supercoefficient} are lower and upper semicontinuous respectively on $\R^{d-1}$.
\end{lemma}

\begin{proof}
We only focus on $Q_\ast$ as the case of $Q^\ast$ is symmetrical. Let $p_n\rta p\in \R^{d-1}$ be a converging sequence of tangential vectors and $(V_{n},Q_\ast(p_n)+\ep_n,p_n)$ a corresponding sequence of subcorrectors to \eqref{eq.truncateforgeneralsemilinearcorrector} with $0<\ep_n \to 0$. Notice that for any $C\in \Z$, we have $$(V_{n}+C,Q_\ast(p_n)+\ep_n,p_n)$$ is still a subcorrector, and therefore for each $n$, we can choose appropriate $C_n\in\Z$ such that $\Tilde{V}_n:=V_{n}+C_n$ are bounded from above by $p_n\cdot x + \lb Q_\ast(p_n)+\ep_n\rb x_1 +2$, which means that $\Tilde{V}_n$ are locally uniformly bounded from above. This shows that if we write
\[
W:={\limsup_{n\rta\infty}}^\ast \Tilde{V}_n  \quad \textup{and} \quad  \tau=\liminf_{n\rta\infty} Q_\ast(p_n),
\]
we obtain a triple $(W,\tau, p)$ that is a subcorrector By Lemma \ref{l.stabilityoftouching}. By definition of $Q_\ast$ this shows that
\[
\liminf_{n\rta\infty} Q_\ast(p_n)=\tau\ge Q_\ast(p).
\]
\end{proof}

{Next we show that the energy minimizing slope is always pinned. The proof uses the $\Gamma$-convergence established in Lemma~\ref{l.localuniformconvergenceofglobalminimizers}.}

\begin{lemma}\label{l.averageisinsidethepinninginterval}
For all $p\in \R^{d-1}$
\[
Q_\ast(p)\le \avg{f} \le Q^\ast(p).
\]
\end{lemma}

\begin{proof}
We without loss normalize $\avg{f}=0$ by applying Remark \ref{r.normalizeaveragezero}. It then suffices to show that $0\in [Q_\ast(p),Q^\ast(p)]$.  We first find the global minimizer $u_\ep$ to the energy \eqref{eq.theenergyepsilon} at scale $\ep>0$ with boundary data $g=q\cdot x$ on $\pd B_1^+=\pt B_1\cap \{x_1\ge 0\}$ for some tangent vector $q\in \R^{d-1}$. By Lemma \ref{l.localuniformconvergenceofglobalminimizers}, we know that $u_\ep$ converges locally uniformly to $u_q$ solving
\[
\bca
\Delta u_q = 0 & \textup{ in }B_1^+\\
u_q=q\cdot x &\textup{ on }\pd B_1^+\\
\pt_1 u_q = \avg{f}=0& \textup{ on }B_1'.
\eca
\]
By standard elliptic theory we know that $u_q(x)\equiv q\cdot x$.

We now define for any small $\delta>0$ the following harmonic function
$$
\phi(x)=q\cdot x + \delta x_1 +\frac{\delta^2}{d-1}|x'|^2 -\delta^2 x_1^2.
$$
This function touches $u_q$ strictly from above at $x=0$ with $\pt_1 \phi(0)=\delta>0$ and by local uniform convergence of $u_\ep$ and Lemma \ref{l.stabilityoftouching}, we can find small constants $C_\ep=o_\ep(1)$ such that $\phi_\ep=\phi+C_\ep$ touches $u_\ep$ strictly from above at some $y_\ep\in B_1'$ with $|y_\ep|=o_\ep(1)$. Consider
\[
v_\ep (y) = \frac{u_\ep(\ep y + y_\ep)-u_\ep(y_\ep)}{\ep},
\]
and
\[
\psi_\ep (y) = \frac{\phi_\ep(\ep y + y_\ep)-\phi_\ep(y_\ep)}{\ep}.
\]
By smoothness of $\phi$, we know that
\[
v_\ep(y) \leq \psi_\ep(y) \leq \grad \phi(0)\cdot y + o_\ep(1)
\]
is locally bounded from above. Therefore for a choice of subsequence $\ep_j$ so that
\[
s:=\lim_{j\rta\infty} \lb \frac{u_{\ep_j}(x_{\ep_j})}{\ep_j} - \left\lfloor\frac{u_{\ep_j}(x_{\ep_j})}{\ep_j} \right\rfloor\rb \in [0,1]
\]
exists, we obtain
\[
v_0:={\limsup_{j\rta\infty}}^\ast v_{\ep_j} - s
\]
is a well-defined subcorrector to \eqref{eq.correctorforgeneralsemilinearproblem}  with upper bound $q\cdot y+\delta y_1$, which also touches $v_0$ from above at $x=0$. Since $\pt_1 \phi(0)=\delta$, we know that $\delta\ge Q_\ast(q)$. As $\delta>0$ can be chosen arbitrarily small we have shown that 
$$ 
Q_\ast(q)\le0.
$$
Similar arguments show that $0\le Q^\ast(q)$.

\end{proof}

Now we proceed to the construction of the correctors. It is convenient to consider the following equation truncated in $\{0\le x_1 \le T\}$ for some $T>0$
\be\label{eq.truncateforgeneralsemilinearcorrector}
\bca
\Delta v = 0 &\textup{ in }\{0<x_1<T\}\\
v=p\cdot x &\textup{ on }\{x_1=T\}\\
\pt_1 v = f(x',v) &\textup{ on }\{x_1=0\}.
\eca
\ee
We take $v_+^{T}$ to be the maximal subsolution defined as
\be\label{eq.maximalsubsolfortruncated}
v_+^{T}(x)=\sup\{v(x) \ ; \ v \textup{ is a subsolution to }\eqref{eq.truncateforgeneralsemilinearcorrector}\}.
\ee
and similarly 
\be\label{eq.minsupsolfortruncated}
v_-^{T}(x)=\inf\{v(x) \ ; \ v \textup{ is a supersolution to }\eqref{eq.truncateforgeneralsemilinearcorrector}\}.
\ee

\begin{remark}\label{r.thetruncatedsatisfybirkhoff}
In general, the solutions $v_+^{T}$ and $v_-^{T}$ are not periodic functions but they do satisfy the {Birkhoff property} as in Definition \ref{def.birkhoffproperty}. This is because when $k\cdot p - s\le 0 $ with $(k,s)\in \Z^{d-1}\times \Z$, the function $v(x+k)-s$ is a subsolution to \eqref{eq.truncateforgeneralsemilinearcorrector} whenever $v$ is a subsolution, which shows that
\[
v_+^T(x+k)-s \le v_+^T(x).
\]
Symmetrically we have $v_+^T(x)\le v_+^T(x+k)-s$ whenever $k\cdot p -s\ge 0$. An anologous argument shows that $v_-^T$ also satisfies the Birkhoff property.

\end{remark}

For the convenience of the following discussions, we denote $\mu_+^T \ge -\norm{f}_{L^\infty}$ as follows
\be\label{eq.linearupperboundonvplusT}
\mu_+^T:=\sup\{\mu \ ; \ p\cdot x +\mu (x_1-T) \ge v_+^{T}\},
\ee
and $\mu_-^T \le \norm{f}_{L^\infty}$ to be
\be\label{eq.linearlowerboundonvminusT}
\mu_-^T:=\inf\{\mu \ ; \ p\cdot x +\mu(x_1-T) \le v_-^{T}\}.
\ee
By a similar argument in the proof of Lemma \ref{l.boundforthecoefficientshowingnontrivialness}, both $|\mu_\pm^T|$ are bounded uniformly by $\norm{f}_{L^\infty}$ for all $T>0$. 

Notice that by \eqref{eq.linearupperboundonvplusT}, the functions $v_+^T+\lfloor\mu_+^T T\rfloor$ are uniformly bounded from above by the linear function $p\cdot x +\mu_+^T  x_1$. We define
\be\label{eq.vplusmplus}
V_+:={\limsup_{T\rta\infty}}^\ast (v_+^T+\lfloor\mu_+^T T\rfloor) \quad \textup{and} \quad  \ m_+:=\limsup_{T\rta\infty} \mu_+^T,
\ee
and symmetrically
\be\label{eq.vminusmminus}
V_-:=\underset{T\rta\infty}{{\liminf}_\ast} (v_-^T+\lceil\mu_-^T T\rceil)\quad \textup{and} \quad  \ m_-:=\liminf_{T\rta\infty} \mu_-^T.
\ee
By Lemma \ref{l.stabilityoftouching}, $(V_+,m_+,p)$ is a subcorrector and $(V_-,m_-,p)$ is a supercorrector to \eqref{eq.correctorforgeneralsemilinearproblem}.

\begin{lemma}\label{l.regularityofcorrector}
For any $\mu\in [Q_\ast(p),Q^\ast(p)]$ , there exists a corrector $(u,\mu,p)$ such that $u$ also satisfies the Birkhoff property \eqref{eq.birkhoffproperty} and there is a constant $C>0$ so that
\be\label{eq.uniformdistancetoaffinefunctionofcorrectors}
\sup_{x_1\ge0} |u(x)- \mu x_1 - p\cdot x'| \le C,
\ee
where $C$ depends only on dimension $d$ and $f$.
\end{lemma}

The proof of this lemma also shows the following result.

\begin{corollary}\label{l.corollaryofregularityofcorrector}
    The truncated solutions $v_+^T+\lfloor\mu_+^T T\rfloor$ and $v_-^T+\lceil\mu_-^T T\rceil$ in \eqref{eq.vplusmplus} and \eqref{eq.vminusmminus} converges locally uniformly on $\R_+^d\cup \pt \R_+^d$ to $V_+,V_-$ respectively. 
\end{corollary}

\begin{proof}[Proof of Lemma \ref{l.regularityofcorrector}]
Let us first show the special case when $\mu=Q_\ast(p)$. The case for $\mu=Q^\ast(p)$ is symmetrical. We make the following two claims on $(V_+,m_+,p)$ as defined in \eqref{eq.vplusmplus}
\begin{enumerate}
    \item[(A)] $m_+=Q_\ast(p)$;
    \item[(B)] $(V_+,m_+,p)$ is a corrector to \eqref{eq.correctorforgeneralsemilinearproblem} that satisfies the bound \eqref{eq.uniformdistancetoaffinefunctionofcorrectors}.
\end{enumerate}

To show claim (A), we observe that for a subcorrector $(u,\mu,p)$ to \eqref{eq.correctorforgeneralsemilinearproblem} and all $C\in \Z$ such that $$C>\sup_{\R_+^d} u(x)-(\mu,p)\cdot x,$$ we have, by maximality of $v_+^T$, for some $s\in[0,1)$
\[
v_+^{T}\ge u-C-\mu T -s,
\]
where $s$ is taken so that $s+\mu T$ is an integer. This implies, by \eqref{eq.linearupperboundonvplusT}, that for every $T>0$ and $0\le x_1\le T$
\[
p\cdot x +\mu_+^T  x_1 \ge u-C-\mu T +\mu_+^T T-s.
\]
At $x_1=0$ we have
\[
 \mu_+^T  \le  \mu +\frac{C-u(0)+1}{T},
\]
which implies that 
$$
m_+=\limsup_{T\rta\infty} \mu_+^T \le \mu.
$$
for every subcorrector $(u,\mu,p)$. Combining the fact that $(V_+,m_+,p)$ is also a subcorrector, we have finished the proof of Claim (A).

To show claim (B), we first observe that $V_+\le p\cdot x +\mu_+^T  x_1$ is defined by the upper half relaxed limit of the truncated subcorrectors $v_+^T+\lfloor\mu_+^T T\rfloor$. For notational convenience, we denote 
\[
U_T(x):=v_+^T+\lfloor\mu_+^T T\rfloor-p\cdot x -\mu_+^T  x_1 \quad \textup{and} \quad  \ c_T=\lfloor\mu_+^T T\rfloor-\mu_+^T T.
\]
By Lemma \ref{l.usefulholderestimate} and Lemma \ref{l.usefulC1alphaestimate}, $U_T$ is a classical solution of
\be\label{eq.tiltedeqiationforUT}
\bca
\Delta U_T =0 &\textup{ in }\{0<x_1<T\}\\
U_T = c_T & \textup{ on }\{x_1=T\}\\
\pt_1 U_T = f(x',U_T+p\cdot x)-\mu_+^T & \textup{ on }\{x_1=0\}.
\eca
\ee
Notice that $U_T\le 0$ according to the definition of $\mu_+^T$ in \eqref{eq.linearupperboundonvplusT}. 

To show that $V_+$ is also a supercorrector, we just need to show a uniform lower bound on $U_T$ independent of $T>0$. 

By Remark \ref{r.thetruncatedsatisfybirkhoff} and the Birkhoff property \eqref{eq.birkhoffproperty}, we know that for $k\in \Z^{d-1}$
\be\label{eq.birkhoffofUT}
U_T(x+k)+k\cdot p- \lceil k\cdot p\rceil\le U_T(x)\le U_T(x+k)+k\cdot p- \lfloor k\cdot p\rfloor,
\ee
and so $|U_T(x+k)-U_T(x)|\le 1$ for all $x\in \{0\le x_1 \le T\}, \ k\in \Z^{d-1}$. We claim that the uniform boundedness follows if there is a constant $C>0$ depending only on dimension $d$ and $f$ so that for all $T>0$ there is a point $x_T'\in \{x_1=0\}$ such that
\be\label{eq.localuniformboundforUT}
\osc_{x_1=0,|x'-x_T'|\le 2} U_T \le C.
\ee
This is because $U_T$ is harmonic in $\{0<x_1< T\}$, bounded on $\{x_1=T\}$, and if \eqref{eq.localuniformboundforUT} holds then combining the Birkhoff property \eqref{eq.birkhoffofUT} we know that $U_T$ is uniformly bounded on $\{0\le x_1\le T\}$ independent of $T>0$.

{Now we return to prove \eqref{eq.localuniformboundforUT}.} By the definition $\mu_+^T$, for all $T>0$ there exists a point $x_T'\in \{x_1=0\}$ such that
\[
U_T(x_T') \ge -3.
\]
Now we solve for an auxiliary function $\eta$ satisfying
\[
\bca
\Delta \eta = 0 &\textup{ in }\{0<x_1 <2, |x'-x_T'|<5\}\\
\eta =0  &\textup{ on }\pd \{0<x_1 <2, |x'-x_T'|< 5\}\\
\pt_1 \eta = \pt_1 U_T  & \textup{ on } \{x_1=0,|x'-x_T'|< 5\}.
\eca
\]
Notice that $|\pt_1 U_T|\le 5 \norm{f}_{L^\infty}$, which means that $\eta$ is uniformly bounded independent of $T$. On the other hand we have
by Harnack inequality there is a constant $C>0$ depending on dimension such that
\[
\max_{\{0\le x_1 \le1, |x'-x_T'|\le 2\}} \lb \eta-U_T \rb \le C \min_{\{0\le x_1 \le1, |x'-x_T'|\le 2\}} \lb \eta-U_T  \rb,
\]
as $\eta-U_T $ is a positive harmonic function on $\{0< x_1 <2, |x'-x_T'|< 5\}$ satisfying zero Neumann boundary condition on $\{x_1=0\}$. This implies that
\[
\osc_{\{0\le x_1 \le1, |x'-x_T'|\le 2\}} U_T \le C\lmb \max_{\{0\le x_1 \le1, |x'-x_T'|\le 2\}}  U_T  + \max_{\{0\le x_1 \le1, |x'-x_T'|\le 2\}} |\eta| \rmb,
\]
on left hand side of which we have $$\osc_{\{x_1=0, |x'-x_T'|\le 2\}} U_T \le \osc_{\{0\le x_1 \le1, |x'-x_T'|\le 2\}} U_T,$$ which implies that
\[
\osc_{\{x_1=0, |x'-x_T'|\le 2\}} U_T \le C\lmb 3  + \max_{\{0\le x_1 \le1, |x'-x_T'|\le 2\}} |\eta| \rmb \le C'
\]
for some positive constant $C'>0$ independent of $T$. Combining Lemma \ref{l.usefulholderestimate} and \ref{l.usefulC1alphaestimate}, this shows that the limit $V_+$ is a classical solution to \eqref{eq.correctorforgeneralsemilinearproblem} with the distance from $Q_\ast(p) x_1 + p\cdot x$ uniformly bounded by a constant $C>0$ depending only on dimension $d$ and $f$.

For $\mu\in [Q_\ast(p),Q^\ast(p)]$, we consider for some constant $c\in \Z$ the supercorrector 
\[
W^\mu(x) = V_-(x) - (Q^\ast(p)-\mu) x_1 + c
\]
and the subcorrector $W_\mu \le W^\mu$ defined as
\[
W_\mu(x) = V_+(x) + (\mu - Q_\ast(p)) x_1.
\]
The existence of a corrector $(V_\mu,\mu,p)$ {with $W_\mu \leq V_\mu \leq W^\mu$} follows from the Perron's method. As $V_\mu$ is defined by taking maximal values of subsolutions (or minimal values of supersolutions) in between $W_\mu,W^\mu$, we know that $V_\mu$ also satisfies the Birkhoff property as in Definition \ref{def.birkhoffproperty} by a similar argument in Remark \ref{r.thetruncatedsatisfybirkhoff}.

\end{proof}

\begin{proof}[Proof of Theorem \ref{t.characterizecorrectors}]
    The proof is complete by combining Lemma \ref{l.boundforthecoefficientshowingnontrivialness}, Lemma \ref{l.semicontinuityofcoefficients}, Lemma \ref{l.averageisinsidethepinninginterval} and Lemma \ref{l.regularityofcorrector}.
\end{proof}

\subsection{Homogenization of general viscosity solutions}\label{subsec.homogenizationofgeneralvissol}
Using the correctors and the pinning interval defined in the previous subsection, we can now present a homogenization result for general viscosity solutions.
\begin{theorem}\label{t.homogenizationofgeneralsemilinear}
 Let $\ep_k \to 0$ and $u_k$ be a sequence of solutions to \be\label{eq.homogenizationNeumanngeneral2}
\bca
\Delta u_k = 0 & \textup{ in }B_1^+\\
\pt_1 u_k = f(x'/\ep_k,u_k/\ep_k) & \textup{ on }B_1'
\eca 
\ee
such that $u_k \to u$ locally uniformly in $B_1^+\cup B_1'$. Then $u$ solves
\be\label{eq.generalhomogenized2}
\bca
\Delta u = 0 & \textup{ in }B_1^+\\
Q_\ast(\gd' u) \le \pt_1 u \le Q^\ast(\gd' u) & \textup{ on }B_1',
\eca
\ee
where the pinning interval $[Q_\ast(p;f),Q^\ast(p;f)]$ is defined in \eqref{eq.subcoefficient} and \eqref{eq.supercoefficient}. 
\end{theorem}
\begin{proof}[Proof of Theorem \ref{t.homogenizationofgeneralsemilinear}] 
We only show the supersolution condition, the subsolution condition is proved by a symmetrical argument. 
    
Let $\phi$ be any smooth function that touches the limit function $u$ from below at some $x_0\in B_1'$, $\Delta \phi(x_0)>0$ and $u\ge \phi$ in $\overline{B_\delta^+(x_0)}$ for some $\delta< 1-|x_0|$. Then by Lemma \ref{l.stabilityoftouching} there exists a sequence of points $x_k\rta x_0$ and constants $c_k\rta 0$ as $k\rta\infty$ such that $\phi_k:=\phi+c_k$ touches $u_k$ from below at $x_k\in B_1'$ for each $k$, and $u_k\ge \phi_k$ in $\overline{B_{\delta/2}^+(x_k)}$.

By choosing subsequences, we can further assume the existence of
$$
s=\lim_{k\rta\infty}\lb u_k(x_k)/\ep_k-\lfloor u_k(x_k)/\ep_k\rfloor\rb\in [0,1].
$$
Let     
\[
v_k(x):= \frac{u_k(\ep_k x + x_k)-u_k(x_k)}{\ep_k}\quad \textup{and} \quad 
\psi_k(x):=\frac{\phi_k(\ep_k x + x_k)-\phi_k(x_k)}{\ep_k}.
\]
By smoothness of $\phi_k$ near $x_0$, we know that for any fixed $R>0$, the sequence of functions $v_k$ is bounded from below in $\overline{B_R^+}$ because
\[
v_k(x) \geq \psi_k(x) \geq \grad \phi_k(0)\cdot x + o_k(1).
\]
Therefore the lower half relaxed limit of $v_k$ 
\[
v_\ast(x) := \underset{k\rta\infty}{{\liminf}_\ast} v_k(x)
\]
exists as a lower-semicontinuous function $\R^d_+ \cup \partial \R^d_+ \to (-\infty,+\infty]$. Notice that for sufficiently large $k$, $v_k$ is by assumption a supersolution to $$\pt_1 v_k \le f (v_k + u_k(x_k)/\ep_k) \textup{ on }B_R'$$ which by Lemma \ref{l.stabilityoftouching} implies that $$\pt_1 v_\ast \le f(v_\ast+s)$$ in the viscosity sense on $\{x_1=0\}$. On the other hand, we know that the linear function $l(x):=\gd \phi(x_0)\cdot x$ touches $v_\ast$ from below at exactly $x=0$. This shows that $(v_\ast,\pt_1\phi(x_0),\gd' \phi(x_0))$ is a supercorrector to \eqref{eq.correctorforgeneralsemilinearproblem}, and the proof is complete by applying Theorem \ref{t.characterizecorrectors}.
 
\end{proof}

\subsection{The homogenized coefficients in the laminar case}\label{subsec.laminarcasehomogenization}

Let us now compute the precise coefficients $Q_\ast $ and $Q^\ast$ in the special case that $f$ is laminar, i.e. $f(x',u)= f(u)$ is a $1$-periodic $C^\alpha$ function only of $u\in \R$. The corrector equation becomes
\be\label{eq.correctorlaminar}
\bca
\Delta v =0 & \text{in }\R_+^d,\\
\pt_1 v = f(v) & \text{on }\R^{d-1},
\eca
\ee
satisfying 
\[
\sup_{\R_+^d} |v(x) - (\mu,p)\cdot x | <\infty,
\]
for some \((\mu,p)\in\R\times\R^{d-1}\). By Remark \ref{r.normalizeaveragezero}, we assume
\[
\avg{f}=\int_0^1 f(v) dv = 0.
\]

\begin{lemma}\label{l.computethelaminarcoefficient}
    Suppose $f(x',v)\equiv f(v)$ is 1-periodic in $v$ and $f\in C^\alpha(\R)$ for some $\alpha>0$, then
    \be
    Q_\ast(p)=L_\ast(p):=\bca
\min f &\textup{ if }p = 0,\\
\avg{f}=0 &\textup{ if }p\ne0,
    \eca
    \ee
    and symmetrically
     \be
    Q^\ast(p)=L^\ast(p):=\bca
\max f &\textup{ if }p = 0,\\
\avg{f}=0 &\textup{ if }p\ne0,
    \eca
    \ee
    where $\avg{f}=\int_0^1 f(r)dr $.
\end{lemma}

By Corollary \ref{l.corollaryofregularityofcorrector}, it suffices to analyze the truncated versions \eqref{eq.truncateforgeneralsemilinearcorrector}, which in the current case takes the form
\be\label{eq.correctorlaminartruncated}
\bca
\Delta v = 0 &\textup{ in }\{0<x_1<T\}\\
v=p\cdot x &\textup{ on }\{x_1=T\}\\
\pt_1 v = f(v) &\textup{ on }\{x_1=0\}.
\eca
\ee
Let $v_+^T$ and $v_-^T$ be the maximal subsolutions and minimal supersolutions respectively. We notice that if \(p=0\), the extremal solutions $v_+^T$ and $v_-^T$ are tangentially invariant linear functions. It suffices to solve for 1-variable solution $v=v(x_1)$ that satisfies
\[
\bca
v''(x_1) = 0 &\textup{ for }x_1\in(0,T)\\
v(T)=0 & v'(0)=f(v(0)).
\eca
\]
Notice that a solution will take the form $v(x_1)=-\frac{v(0)}{T} (x_1-T)$, where $v(0)$ satisfies
\be\label{eq.horizontalpinningequation}
-\frac{v(0)}{T}= f(v(0)).
\ee
Let $v(0)=w_+^T$ be the maximal solution and $v(0)=w_-^T$ be the minimal solution to the above equation, then we know that
\[
v_\pm^T(x_1,x')=-\frac{w_\pm^T}{T}(x_1-T).
\]
On the other hand, we have the following characterization of $w_\pm^T$.
\begin{lemma}\label{l.exactoscillationwhentangentialgradientiszero}
Suppose $f$ is continuous and 1-periodic in \eqref{eq.horizontalpinningequation}, then as $T\rta \infty$, we have 
\[
\lim_{T\rta \infty} f(w_+^T) =\lim_{T\rta \infty}-\frac{w_+^T}{T} = \min f,
\]
and 
\[
\lim_{T\rta \infty} f(w_-^T)=\lim_{T\rta \infty}-\frac{w_-^T}{T} = \max f.
\]    
\end{lemma}

\begin{proof}
    We only show the $w_+^T$ case as the $w_-^T$ case is symmetrical. On one hand, we know that 
    \[
    -\frac{w_+^T}{T} = f(w_+^T) \ge \min f,
    \]
and hence 
\[
w_+^T \le - (\min f )T.
\]
Because $-\frac{v}{T}$ decays linearly and $f$ remains bounded, this shows that when $v>- (\min f )T$, we always have $-\frac{v}{T}< f(v)$. On the other hand, when $- (\min f )T- 3\le v \le - (\min f )T-1$, we have 
\[
\min f + \frac{1}{T} \le -\frac{v}{T} \le \min f + \frac{3}{T}.
\]
By the 1-periodicity of $f$ there is a $v_\ast\in [- (\min f )T- 3, - (\min f )T-1]$ such that $f(v_\ast)=\min f$. Now, we obtain the inequality
\[
-\frac{v_\ast}{T} \ge \min f + \frac{1}{T}>f(v_\ast).
\]
By using the intermediate value theorem, we know that there must be a number $\Tilde{v}\ge v_\ast \ge - (\min f )T- 3$ such that \eqref{eq.horizontalpinningequation} holds for $v(0)=\Tilde{v}$. By maximality of $w_+^T$ we have
\[
w_+^T\ge \Tilde{v} \ge - (\min f )T- 3.
\]
This shows that 
\[
f(w_+^T) = -\frac{w_+^T}{T} \le \min f + \frac{3}{T} \rta \min f
\]
as $T\rta\infty$.
\end{proof}

For \(p\neq 0\), because \(f\) is 1-periodic in \(u\), \(v_{\pm}^T(x-p/|p|^2) + 1 = v_{\pm}^T(x)\), with tangential translation invariance:
\[
v_{\pm}^T(x - q) = v_{\pm}^T(x) \ \hbox{ for } \  \ q \perp p \ \textup{and} \ e_1.
\]
This implies \(\eta:=v_{\pm}^T - p\cdot x\) is a \(\frac{1}{|p|}\)-periodic 2-variable function, with continuity derived from \(v_{\pm}^T(x) - v_{\pm}^T(x-sp/|p|^2) = s\). The problem \eqref{eq.correctorlaminartruncated} is thus reduced to:
\be\label{eq.the2variableproblem}
\bca
\pt_1^2\eta + \pt_2^2\eta = 0 & \text{in }\{0<x_1<T\},\\
\eta = 0 & \text{on }\{x_1=T\},\\
\pt_1\eta = f(\eta + |p|x_2) & \text{on }\{x_1=0\},
\eca
\ee
with \(\eta\in C(\{0\leq x_1\leq T\})\), \(1/|p|\)-periodic in $x_2$. By the $C^\alpha$-regularity of \(f\), we know that $\eta$ is a classical solution to \eqref{eq.the2variableproblem} according to Lemma \ref{l.usefulholderestimate} and \ref{l.usefulC1alphaestimate}.

Now we show that \(\eta\) has bounded distance (independent of \(T\)) from \(\mu_+^T(x_1-T)\) with 
\[
\lim_{T\rta \infty} \mu_+^T = \langle f \rangle=0,
\]
where $\mu_+^T$ is defined in \eqref{eq.linearupperboundonvplusT}.

By Corollary \ref{l.corollaryofregularityofcorrector} $\eta$ has uniform distance independent of $T$ from an affine function $\mu_+^T(x_1-T)$ with $\mu_+^T$ having a limit $\mu_+^\infty=\lim_{T\rta\infty} \mu_+^T$. Let 
\[
\langle \eta \rangle (x_1 ):=|p|\int_0^{1/|p|} \eta(x_1,x_2) dx_2,
\]
we have, by \eqref{eq.the2variableproblem},
\[
\langle \eta \rangle(x_1) = |p|\int_0^{1/|p|} f(\eta(0,x_2)+|p|x_2) dx_2 (x_1 -T).
\]
By Lemma \ref{l.regularityofcorrector} and Corollary \ref{l.corollaryofregularityofcorrector}, we have
\[
\mu_+^\infty=\lim_{T\rta \infty}\mu_+^T = \lim_{T\rta \infty} |p|\int_0^{1/|p|} f(\eta(0,x_2)+|p|x_2) dx_2 .
\]
\begin{lemma}\label{l.exactoscillationwhentangentialgradientisnonzero}
    For all \(T>0\) and any viscosity solution \(\eta\) to \eqref{eq.the2variableproblem}, we have
    \[
    |p|\int_0^{1/|p|} f(\eta(0,x_2)+|p|x_2) dx_2 = \langle f \rangle=0,
    \]
    where \(\langle f \rangle:=\int_0^1 f(z) dz=0\).
\end{lemma}

\begin{proof}
By applying \(\pt_2 \eta\) to \eqref{eq.the2variableproblem} we have 
\[
\begin{split}
   0&=\int_{0}^T\int_0^{1/|p|} \pt_2 \eta(x_1,x_2) \Delta \eta(x_1,x_2) dx_2 dx_1\\
   &= -\int_{0}^T\int_0^{1/|p|} \gd \pt_2\eta(x_1,x_2) \cdot \gd \eta(x_1,x_2) dx_2 dx_1 - \int_0^{1/|p|} \pt_1 \eta(0,x_2) \pt_2 \eta(0,x_2) dx_2\\
   &=-\int_{0}^T\int_0^{1/|p|} \frac{1}{2}\pt_2|\gd \eta |^2(x_1,x_2) dx_2dx_1- \int_0^{1/|p|} \pt_1 \eta(0,x_2) \pt_2 \eta(0,x_2) dx_2\\
   &=- \int_0^{1/|p|} \pt_1 \eta(0,x_2) \pt_2 \eta(0,x_2) dx_2.
\end{split}
\]
On the other hand, by making the substitution 
\[z = \eta(0,x_2)+|p|x_2 \ \hbox{ and } \ dz = (\pt_2 \eta(0,x_2)+|p|) dx_2\]
we get
\[
\begin{split}
    \langle f\rangle&=\int_0^1 f(z) dz\\ 
   &= \int_0^{1/|p|} f(\eta(0,x_2)+|p|x_2) (\pt_2\eta(0,x_2) + |p|) dx_2\\
&=\int_0^{1/|p|}\pt_1 \eta(0,x_2) \pt_2\eta(0,x_2) dx_2  + |p| \int_0^{1/|p|} f(\eta(0,x_2)+|p|x_2) dx_2\\
&=|p| \int_0^{1/|p|} f(\eta(0,x_2)+|p|x_2) dx_2.
\end{split}
\]
    
\end{proof}

\begin{proof}[Proof of Lemma \ref{l.computethelaminarcoefficient}]
    The result follows from Lemma \ref{l.exactoscillationwhentangentialgradientiszero}, Lemma \ref{l.exactoscillationwhentangentialgradientisnonzero}.
\end{proof}

\section{Homogenization of extremal solutions in laminar media}\label{section.extremalsteadystates}

In this section, we study the homogenization of the extremal solutions, that is, the minimal supersolutions and maximal subsolutions of
\begin{equation}\label{e.laminar-hom-section}
    \bca
\Delta u = 0 & \textup{ in }B_1^+\\
\pt_1 u = f(u/\ep) & \textup{ on }B_1'.
\eca
\end{equation}

\begin{definition}\label{def.extremalsolutiontolaminarhomogenization}
Given a boundary data $g\in C(\partial B_1 \cap \{x_1 \geq 0\})$ we define the \emph{extremal solutions} of \eqref{e.laminar-hom-section}: the \emph{minimal supersolution}
\[
u_{g,\min{}}^\ep(x) := \inf\{u(x) \ : \ u \textup{ is a supersolution to \eqref{e.laminar-hom-section} and }u\ge g \textup{ on }\pt B_1\cap\{x_1\ge0\}\},
\]
and the \emph{maximal subsolution}
\[
u_{g,\max{}}^\ep(x) := \sup\{u(x) \ : \ u \textup{ is a subsolution to \eqref{e.laminar-hom-section} and }u\le g \textup{ on }\pt B_1\cap\{x_1\ge0\}\}.
\]
\end{definition}

We remark here that for each fixed $\ep$ the extremal solutions $u_{g,\min / \max}^\ep$ are classical solutions to \eqref{e.laminar-hom-section} according to Perron's method, Lemma \ref{l.usefulholderestimate} and Lemma \ref{l.usefulC1alphaestimate}.

 We show that extremal solutions to \eqref{e.laminar-hom-section} converge, respectively, to extremal solutions of
\begin{equation}\label{e.laminar-limit-hom-section}
     \bca
\Delta u = 0 & \textup{ in }B_1^+\\
\pt_1 u \in [L_\ast(\grad'u),L^\ast(\grad'u)] & \textup{ on }B_1'.
\eca
\end{equation}
Here the homogenized coefficients, as derived in Section \ref{section.generalsteadystates}, are
\be\label{eq.normalizedlowercoefficientforlaminar}
L_\ast(p)=\lb\min f\rb 1_{\{p=0\}} + \avg{f} 1_{\{p \neq 0\}} \ \hbox{ and } \ L^\ast(p)=\lb\max f\rb 1_{\{p=0\}} + \avg{f} 1_{\{p \neq 0\}}.
\ee
  To reduce notation, for the rest of the section we will keep the normalization
\[
\avg{f}=0.
\]
This normalization can be achieved by adding an affine function as explained in Remark \ref{r.normalizeaveragezero}.

\begin{definition}
    An upper semicontinuous function $u$ is called a \emph{subsolution} to \eqref{e.laminar-limit-hom-section} if for any smooth function $\phi$ touching $u$ from above at some $x\in B_1^+\cup B_1'$, either $\Delta \phi(x) \ge 0$ or $x\in B_1'$ and 
    \[
    \pt_1 \phi(x) \ge L_\ast(\gd' \phi(x)).
    \]
\end{definition}

\begin{definition}
    A lower semicontinuous function $u$ is called a \emph{supersolution} to \eqref{e.laminar-limit-hom-section} if for any smooth function $\phi$ touching $u$ from below at some $x\in B_1^+\cup B_1'$, either $\Delta \phi(x) \le 0$ or $x\in B_1'$ and 
    \[
    \pt_1 \phi(x) \le L^\ast(\gd' \phi(x)).
    \]
\end{definition}

\begin{remark}\label{r.equivalenceofthreeviscosityinequalities}
    The following three viscosity inequality constraints are equivalent for a lower semicontinuous function $u$ (the arguments are symmetric for upper semicontinuous functions):
    \begin{enumerate}[label=(\roman*)]
            \item $\pt_1 u \le L^\ast(\gd' u)$.
        \item $\pt_1 u \le \max f, \ \min \{\pt_1 u, |\gd' u|\} \le 0$.
        \item $\pt_1 u \le \max f, \   |\gd' u|\pt_1 u\le0$. 
    \end{enumerate}
       
\end{remark}

\begin{definition}\label{def.extremalsolutiontogradientdegenerate}
Given a boundary data $g\in C(\partial B_1 \cap \{x_1 \geq 0\})$ we define the \emph{extremal solutions} of \eqref{e.laminar-limit-hom-section}: the \emph{minimal supersolution}
\[
u_{g,\min{}}(x) := \inf\{u(x) \ : \ u \textup{ is a supersolution to \eqref{e.laminar-limit-hom-section} and }u\ge g \textup{ on }\pt B_1\cap\{x_1\ge0\}\},
\]
and the \emph{maximal subsolution}
\[
u_{g,\max{}}(x) := \sup\{u(x) \ : \ u \textup{ is a subsolution to \eqref{e.laminar-limit-hom-section} and }u\le g \textup{ on }\pt B_1\cap\{x_1\ge0\}\}.
\]
\end{definition}

Our main result of this section is the homogenization: extremal solutions of \eqref{e.laminar-hom-section} converge to extremal solutions of \eqref{e.laminar-limit-hom-section}.
\begin{theorem}\label{t.homogenizationofextremalsteadystates}
    Let $g\in C(\partial B_1 \cap \{x_1 \geq 0\})$ and $u^\ep=u_{g,\min}^\ep$ are minimal supersolutions to \eqref{e.laminar-hom-section} with $u^\ep = g$ on $\partial^+B_1=\partial B_1 \cap \{x_1 \geq 0\}$. Then 
    \[u^\ep \to u_{g,\min{}} \ \hbox{ uniformly on } \  \overline{B_1^+},\]
    where $u_{g,\min{}}$ is the minimal supersolution of \eqref{e.laminar-limit-hom-section} with boundary data $g$ on $\partial^+B_1$. The analogous statement for maximal subsolutions holds by symmetry.
\end{theorem}

\begin{proof}[Proof of Theorem \ref{t.homogenizationofspecialsolutions}]
  The proof is done by combining Remark \ref{r.equivalenceofthreeviscosityinequalities}, Theorem \ref{t.homogenizationofextremalsteadystates}.
\end{proof}

Notice that it is difficult to directly check the minimality of the limit of $u_{g,\min}^\ep$. To overcome this difficulty we propose an alternative characterization of minimal supersolutions to \eqref{e.laminar-limit-hom-section} by local viscosity solution testing properties. Such characterization appears also in \cites{feldman2021limit,feldman2024regularitytheorygradientdegenerate} and is a suitable criterion for using the test function type argument in nonlinear homogenization as discussed in \cite{Evans1992}.

\subsection{A characterization of the extremal homogenized steady states in the laminar case}
In this subsection, we introduce an equivalent characterization for the minimal supersolutions as defined in Definition \ref{def.extremalsolutiontogradientdegenerate}. This characterization is for preparation of the nonlinear homogenization argument that will appear in Section \ref{subsection.homogenizationofminimalsupersolutionstoNNep}

\begin{theorem}\label{t.characterizeminimalsuperslutiontohomoextremalsolutions}
Let $\xi\in C(\overline{B_1^+})$ be a harmonic function on $B_1^+$ such that the following conditions hold on $B_1'$:
   \begin{enumerate}[label = (\arabic*)]
    \item \label{condition(1)}\(\min\{\pt_1 \xi, |\gd' \xi|\} = 0\) 
    \item \label{condition(2)}$\pt_1 \xi \le \max f$ 
    \item  \label{condition(3)}If a smooth function of the form $\eta(x)\equiv \eta(x_1)$ touches $\xi$ from above at $z\in B_1'$ and for some open domain $\Omega\csubset \R^d$ containing $z$ such that $\Omega\cap \overline{B_1^+} \csubset B_1^+\cup B_1'$ we have $\eta \ge  \xi$ in $\Omega\cap \overline{B_1^+}$ and $\eta>\xi$ on $\pt \Omega \cap \overline{B_1^+}$, then $\pt_1 \eta(z)=\pt_1\eta(0) \ge  \max f $.
   \end{enumerate}
Then $\xi $ is equal to the minimal supersolution $u_{\xi,\min}$ to \eqref{e.laminar-limit-hom-section} that satisfies $u_{\xi,\min}=\xi$ on $\pt B_1\cap \{x_1\ge0\}$.
\end{theorem}

Before showing the theorem, let us introduce the following auxiliary lemma.
\begin{lemma}\label{l.partitionofB1prime}
    Given any supersolution $v$ to \eqref{e.laminar-limit-hom-section}, there is a partition 
    \[
    B_1'= \Ca_v \sqcup \Na_v \sqcup \Gamma_v,
    \]
where $$\Ca_v:=\{x\in B_1' \ ; \ \exists \textup{ smooth }\phi \textup{ touching }v \textup{ from below at }x \textup{ such that }\pt_1\phi(x)>0\}$$
is relatively open in $B_1'$, $\Na_v:= B_1'\setminus \overline{\Ca_v}$ and $\Gamma_v$ the common boundary of $\Ca_v$ and $\Na_v$.
    
\end{lemma}

\begin{proof}
    Any such supersolution $v$ is also a supersolution to $$\min\{\pt_1 v,|\gd' v|\} \le 0,$$
    and hence we can apply the results in Section 2 in  \cite{feldman2024regularitytheorygradientdegenerate}.
\end{proof}

\begin{proof}[Proof of Theorem \ref{t.characterizeminimalsuperslutiontohomoextremalsolutions}]

We only discuss the minimal supersolution case, as the maximal subsolution case is symmetrical. We also denote by $g$ the restriction of $\xi$ on $\pt B_1\cap \{x_1\ge0\}$.

 \emph{Step 1: Minimal supersolution satisfies the three conditions.}  
By Remark \ref{r.equivalenceofthreeviscosityinequalities} and the standard Perron's method, the minimal supersolution \( u \) satisfies conditions \ref{condition(1)} and \ref{condition(2)}. The condition \ref{condition(3)} is carried out by a similar argument, but due to the unusual form, we still perform the argument here. To verify condition \ref{condition(3)}, assume there exists a smooth function \( \eta_0(x) = \eta_0(x_1) \) such that:
\[
\partial_1 \eta_0(0) < \max f,
\]
and \( \eta_0 \) touches \( u = u_{g,\min}\) from above at a point \( z \in \Omega \) in \( \Omega \cap \overline{B_1^+} \csubset B_1^+ \cup B_1' \), with \( \eta_0 > u \) on \( \partial \Omega \cap \overline{B_1^+} \). Consider the function:
\[
(\max f) x_1 + \eta_0(0) - \tau,
\]
which is a supersolution to \eqref{e.laminar-limit-hom-section} for each constant \( \tau \). When \( \tau = 0 \), since \( \partial_1 \eta_0(0) < \max f \), there exists a small \( \delta > 0 \) such that:
\[
(\max f) x_1 + \eta_0(0) > \eta_0(x_1) \quad \text{on} \quad 0 < x_1 \leq \delta.
\]
In particular, on \( \partial \Omega \cap \{ 0 \leq x_1 \leq \delta \} \), we have:
\[
(\max f) x_1 + \eta_0(0) \geq \eta_0(x_1) > u(x).
\]
This implies:
\[
(\max f) x_1 + \eta_0(0) > u \quad \text{on} \quad \partial \big( \Omega \cap \{ x_1 < \delta \} \big) \cap \{ x_1 \geq 0 \}.
\]
By continuity, we can choose \( \tau_0 > 0 \) small enough so that:
\[
(\max f) x_1 + \eta_0(0) - \tau_0 > u \quad \text{on} \quad \partial \big( \Omega \cap \{ x_1 < \delta \} \big) \cap \{ x_1 \geq 0 \}.
\]
Taking the minimum of \( u \) and \( (\max f) x_1 + \eta_0(0) - \tau_0 \) inside \( \Omega \cap \overline{B_1^+} \) yields a strictly smaller supersolution to \eqref{e.laminar-limit-hom-section} with boundary data \( g \), contradicting the minimality of \( u \). Thus, condition \ref{condition(3)} holds.

\emph{Step 2: the three conditions imply minimality.}  
To prove the reverse, we show the comparison principle between a supersolution \( v \) to \eqref{e.laminar-limit-hom-section} such that $v\ge g$ on $\pt B_1\cap \{x_1\ge0\}$ and a subsolution \( \xi \) satisfying conditions (1) and (3). Following the argument in \cite{feldman2024regularitytheorygradientdegenerate}*{Section 8.2}, we use tangential convolutions and harmonic lifts (Also see Appendix \ref{appendix.tangentialregularization}). These methods allow us to assume, \( v \) to be a supersolution, semiconvex on \( B_1' \), and harmonic in \( B_1^+ \), while \( \xi \) to be a subsolution, semiconcave on \( B_1' \), and harmonic in \( B_1^+ \). Specifically we know that $v$ satisfies
\be\label{eq.equationforasupersolutioninpage22}
\bca
\Delta v = 0 & \textup{ in }B_1^+\\
v \ge g  & \textup{ on }\pt B_1\cap \{x_1\ge0\}\\
\pt_1 v \le L^\ast(\gd' v)=(\max f)1_{\{\gd ' v =0\}} & \textup{ on }B_1',
\eca
\ee
and $\xi$ satisfies condition \ref{condition(3)} and
\be
\bca
\Delta \xi = 0 & \textup{ in }B_1^+\\
\xi \le g  & \textup{ on }\pt B_1\cap \{x_1\ge0\}\\
\pt_1 \xi \ge 0 & \textup{ on }B_1'.
\eca
\ee
We denote for $h>0$ the strict perturbation 
\[
v_h:=v -h x_1 + 2h.
\]

Suppose $v_h$ touches $\xi$ from above at some point \( x_0 \in B_1' \). By semi-convexity of $v_h$ and semi-concavity of $\xi$, both $v_h$ and $\xi$ are differentiable at $x_0$. By condition \ref{condition(1)}, we have:
\[
\partial_1 \xi(x_0) \geq 0,
\]
which implies that 
\be\label{eq.x0notinGammavNv}
x_0 \notin \mathcal{N}_v \sqcup \Gamma_v 
\ee
because in this case \( \partial_1 v_h(x_0) < 0 \) by \eqref{eq.equationforasupersolutioninpage22} and Lemma \ref{l.partitionofB1prime}. Let \( \Omega \subset \Ca_v \) be the component containing \( x_0 \).
We claim by condition \ref{condition(3)}, there must exist another touching point \( y \in \partial' \Omega \cap B_1' \subset \Gamma_v \cap B_1' \).  For \( \delta, \ep > 0 \), define:
\[
\Omega_{\delta, \ep} := \big\{ x \in \Omega \, ; \, \text{dist}(x, \partial' \Omega) > \ep, \, \text{dist}(x, \partial' B_1') > \delta \big\}.
\]
Because $v_h>\xi $ on $\pt B_1\cap \{x_1\ge0\}$, we know that when \( \delta, \ep > 0 \) are small we always have
\[
\max_{\partial' \Omega_{\delta, \ep}} (\xi - v_h) = \max_{(\partial' \Omega_{0, \ep}) \cap B_{1-\delta}'} (\xi - v_h).
\]
If \(\displaystyle \max_{\partial' \Omega_{\delta, \ep}} (\xi - v_h) < 0 \) (i.e., no touching point \( y \) exists on \( (\partial' \Omega_{0, \ep}) \cap B_{1-\delta}' \)), then because
\[
\pt_1 v_h = \pt_1 v -h < \max f
\]
and $v$ is constant on \( \Omega \), we have for sufficiently small \( \tau > 0 \):
\[
(\max f - \tau) x_1 + v(x_0) \geq v \quad \text{on} \quad \Omega_{\delta, \ep} \times [0, \tau].
\]
This produces a 1-variable function \( (\max f - \tau/2) x_1 + v(x_0) \) that violates condition \ref{condition(3)} on \( \Omega_{\delta, \ep} \times [0, \tau] \). Thus, for fixed small \( \delta > 0 \) and each small \( \ep > 0 \), there exists 
\[
y = y_\ep \in (\partial' \Omega_{0, \ep}) \cap B_{1-\delta}' 
\]
such that \( \xi \) touches \( v \) from below at \( y_\ep \). Taking a subsequence as \( \ep \to 0 \), we obtain a point \( y_0 \in (\partial' \Omega) \cap B_{1-\delta}' \) where \( \xi \) touches \( v \) from below, which is a contradiction since we already showed above in \eqref{eq.x0notinGammavNv} that there are no touching points in $\Gamma_v$.
\end{proof}

\subsection{Homogenization of the minimal supersolutions}\label{subsection.homogenizationofminimalsupersolutionstoNNep}

According to Theorem \ref{t.homogenizationofgeneralsemilinear} and Theorem \ref{t.characterizeminimalsuperslutiontohomoextremalsolutions}, to prove Theorem \ref{t.homogenizationofextremalsteadystates}, we need to show that any locally uniform limit $u$ of a subsequence $u_k:= u^{\ep_k}$ on $B_1^+\cup B_1'$ is harmonic in $B_1^+$ and satisfies the three viscosity solution conditions in Theorem \ref{t.characterizeminimalsuperslutiontohomoextremalsolutions}. We split the proof into two lemmas.

We first show that a minimal supersolution $u$ satisfies the condition \ref{condition(1)} in Theorem \ref{t.characterizeminimalsuperslutiontohomoextremalsolutions}. This is essentially proving that 
\[
\pt_1 u \ge \avg{f}=0,
\]
which incorporates the effects of global minimizers to minimal supersolutions.

\begin{lemma}\label{l.extremesteadystateslimitcondition1}
    Let $u_k:=u^{\ep_k}$ be a sequence of minimal supersolutions to \eqref{e.laminar-hom-section} with $\ep_k \to 0$ and $u_k \to u$ locally uniformly in $B_1^+\cup B_1'$. Then $u$ satisfies condition \ref{condition(1)} in Theorem \ref{t.characterizeminimalsuperslutiontohomoextremalsolutions}.
\end{lemma}

\begin{remark}
    The proof of this lemma also shows that in the general $f(x',v)$ case (instead of merely this laminar case $f(v)$), the minimal supersolution $u$ to the homogenized equation \eqref{eq.homogenizationNeumanngeneral1} satisfies $$\avg{f}\le \pt_1 u \le Q^\ast(\gd' u ; f).$$
\end{remark}

\begin{proof}
  To show condition \ref{condition(1)}, by Theorem \ref{t.homogenizationofgeneralsemilinear}, Theorem \ref{t.characterizecorrectors}, Lemma \ref{l.computethelaminarcoefficient} and Remark \ref{r.equivalenceofthreeviscosityinequalities}, it suffices to show the subsolution condition
\[
\pt_1 u \ge 0.
\]
Suppose, for contradiction, that this is not the case. Then there exists a smooth function \(\phi\) that touches \(u\) from above at some point \(x_0 \in B_1'\), satisfying 
\[
\Delta \phi(x_0) < 0 \quad \textup{and} \quad   \pt_1 \phi(x_0) < 0,
\]
and \(u \le \phi\) in \(\overline{B_\delta^+(x_0)}\) for some \(0 < \delta < 1 - |x_0|\). By the smoothness of \(\phi\), we may choose \(\delta\) small enough so that 
\[
\Delta \phi(x)  < 0 \quad \textup{and} \quad  \pt_1 \phi(x) < 0
\]
for all \(x \in \overline{B_\delta^+(x_0)}\).

Without loss of generality, we can make the touching to be strict by applying a small perturbation (e.g., \(\phi + \tau|x - x_0|^2\) for a small \(\tau > 0\)) so that 
\be\label{eq.perturbationtomakestricttouching}
\phi(x_0)=u(x_0)\quad \textup{and} \quad \phi > u \text{ in } \overline{B_\delta^+(x_0)} \setminus \{x_0\}.
\ee
By Lemma \ref{l.stabilityoftouching}, there exist sequences of points \(x_k \to x_0\) and constants \(C_k \to 0\) as \(k \to \infty\) such that 
\[
\phi_k := \phi + C_k
\]
touches \(u_k\) from above at \(x_k \in B_1'\) for each \(k\), and \(u_k \ge \phi_k\) in \(\overline{B_{\delta/2}^+(x_k)}\).

Now, we solve the following global energy minimization problem:
\[
\psi_k := \argmin_\psi \lma E_{\ep_k}\left(\psi, B_{\delta/2}^+(x_k)\right) \ ; \  \psi = \phi_k \text{ on } \pt B_{\delta/2}(x_k) \cap \{x_1 \ge 0\} \rma,
\]
where \(E_{\ep_k}\) is defined in \eqref{eq.theenergyepsilon}.
By \lref{localuniformconvergenceofglobalminimizers}, global minimizers $$\psi_k \in C(\overline{B_{\delta/2}^+(x_k)}) \cap C_{\loc}^{1,\alpha}(B_{\delta/2}^+(x_k) \cup B_{\delta/2}'(x_k))$$ converge uniformly to \(\psi_0\), solving
\begin{equation}\label{e.psi0-eqn}
    \begin{cases}
\Delta \psi_0 = 0 & \text{in } B_{\delta/2}^+(x_0), \\
\psi_0 = \phi & \text{on } \pt B_{\delta/2}^+(x_0) \cap \{x_1 \ge 0\}, \\
\pt_1 \psi_0 = 0 & \text{on } B_{\delta/2}'(x_0).
\end{cases}
\end{equation}
Since $\phi$ is a strict supersolution of \eref{psi0-eqn}, we obtain 
\[
\psi_0 < \phi \  \text{ in } \ \overline{B_{\delta/4}'(x_k)}
\]
for all sufficiently large \(k > 0\). By the uniform convergence of \(\psi_k\), this implies 
\[
\psi_k < \phi_k \text{ in } \overline{B_{\delta/4}'(x_k)},
\]
for all sufficiently large \(k > 0\). This leads to a contradiction to the minimality of \(u_k\) if we consider the minimum of \(u_k\) and \(\psi_k\), which produces a strictly smaller supersolution to \eqref{eq.homogenizationNeumann}.
\end{proof}

\begin{lemma}\label{l.extremesteadystateslimitcondition2}
  Let $u_k:=u^{\ep_k}$ be a sequence of minimal supersolutions to \eqref{e.laminar-hom-section} (the $\ep$-problem $\pt_1 u = f(u/\ep)$) with $\ep_k \to 0$ and $u_k \to u$ locally uniformly in $B_1^+\cup B_1'$. Then $u$ satisfies condition \ref{condition(3)} in Theorem \ref{t.characterizeminimalsuperslutiontohomoextremalsolutions}.
\end{lemma}

\begin{proof}
 To show condition \ref{condition(3)}, we assume there exists a smooth one-variable function \(\eta = \eta(x_1)\) that touches \(u\) from above at a point \(z \in B_1'\), satisfying 
\[
\pt_1 \eta(0) < \max f,
\]
and such that for some small open domain \(\Omega \csubset \R^d\) containing \(z\) with \(\Omega \cap \overline{B_1^+} \csubset B_1^+ \cup B_1'\), we have 
\[
\eta \ge u \text{ in } \Omega \cap \overline{B_1^+}, \ \eta(0)=u(z) \ \textup{ and } \  \eta > u \text{ on } \pt \Omega \cap \overline{B_1^+}.
\]
Notice that, since \(\pt_1 \eta(0) < \max f\), for small \(\tau > 0\), the function 
\[
(\max f) x_1 + \eta(0) - \tau
\]
satisfies 
\[
(\max f) x_1 + \eta(0) - \tau > u \text{ on } \pt (\Omega \cap \{x_1 < \tau\}) \cap \overline{B_1^+} =: \pt \Omega^\tau \cap \overline{B_1^+}.
\]
Let \(\tau \gg h > 0\) be small, and define the fattened domain 
\[
\Omega_h^\tau := \bigcup_{x \in \Omega^\tau} B_h(x).
\]
By the continuity of \(u\) and \((\max f) x_1 + \eta(0) - \tau\), we also have 
\[
(\max f) x_1 + \eta(0) - \tau > u \text{ in } \overline{\Omega_h^\tau \setminus \Omega^\tau} \cap \overline{B_1^+}.
\]
By the local uniform convergence of \(u_k\) to \(u\) and Lemma \ref{l.stabilityoftouching}, for large \(k\), there exist points \(z_k \in \Omega \cap B_1'\) with \(z_k \to z\) and constants \(c_k \to 0\) such that 
\[
(\max f) x_1 + \eta(0) + c_k
\]
touches \(u_k\) from above at \(z_k\) in \(\overline{\Omega_h^\tau} \cap \overline{B_1^+}\), and 
\[
(\max f) x_1 + \eta(0) + c_k > u_k \text{ in } \overline{\Omega_h^\tau \setminus \Omega^\tau} \cap \overline{B_1^+}.
\]
Moreover, for all \(0 < \nu \le \tau/2\), the inequality 
\be\label{eq.strictorderinginlemma4.11}
(\max f) x_1 + \eta(0) + c_k - \nu > u_k
\ee
still holds in \(\overline{\Omega_h^\tau \setminus \Omega^\tau} \cap \overline{B_1^+}\). Since \(\tau > 0\) is chosen independently of \(\ep_k\), for large \(k\), we have \(\tau \gg \ep_k\), where \(\ep_k\) is the period of \(f_{\ep_k}:=f(\cdot/\ep_k)\). This implies that we can choose \(0 < \nu_k \le \tau/2\) such that 
\[
f_{\ep_k}(\eta(0) + c_k - \nu_k) = \max f,
\]
making 
\[
(\max f) x_1 + \eta(0) + c_k - \nu_k
\]
a supersolution to \eqref{eq.homogenizationNeumann} for \(\ep = \ep_k\). This contradicts the fact that \(u_k\)'s are minimal supersolutions because \eqref{eq.strictorderinginlemma4.11} and
\[
\begin{split}
u_k(z_k)  &=\eta(0)+c_k\\
   &>\eta(0) + c_k - \nu_k
\end{split}
\]
shows that the minimum of $u_k$ and $(\max f) x_1 + \eta(0) + c_k - \nu_k$ is a supersolution strictly smaller than $u_k$.

\end{proof}

Let us now prove Theorem \ref{t.homogenizationofextremalsteadystates}. 

\begin{proof}[Proof of Theorem \ref{t.homogenizationofextremalsteadystates}]

We first show that any sequence $u^{\ep_l}=u_{g,\min}^{\ep_l}$ of minimal supersolutions contains a subsequence $u_k$ that converges uniformly on $\overline{B_1^+}$ to $u\in C(\overline{B_1^+})$. To see this, we let $h$ to be the solution to
\[
\bca
\Delta h  = 0 & \textup{ in }B_1^+\\
h = g & \textup{ on }\pt B_1 \cap \{x_1\ge 0\}\\
\pt_1 h = 0 & \textup{ on }B_1'.
\eca
\]
By Perron's method and standard elliptic regularity theory
$$h\in C(\overline{B_1^+})\cap C^{2,\alpha}(B_1^+\cup B_1').$$
Therefore, we know that, by Perron's method again, the functions $v_k:=u_k -h $ are harmonic in $B_1^+$, $v_k=0$ on $\pd B_1^+ = \pt B_1\cap \{x_1\ge0\}$ and
\[
\pt_1 v_* \le \max f \quad \textup{and} \quad  \pt_1 v^* \ge \min f.
\]
By Lemma \ref{l.usefulholderestimate}, we know that $v_k$'s are uniformly bounded in $C^\alpha(B_1^+)$-norm. By using Arzela-Ascoli theorem, we know that $v_k$, after passage to a subsequence, uniformly converges to some $\Tilde{v}\in C(\overline{B_1^+})$. This shows that $u_k$, also after passage to a subsequence, converges to $u=\Tilde{v}+h\in C(\overline{B_1^+})$.

Now, observe that condition \ref{condition(2)} follows from the fact that $\pt_1 u_k \le \max f $ for all $k$ and Lemma \ref{l.stabilityoftouching}. Condition \ref{condition(1)} and \ref{condition(3)} are proved in Lemma \ref{l.extremesteadystateslimitcondition1} and Lemma \ref{l.extremesteadystateslimitcondition2} respectively. The proof is then complete by applying Theorem \ref{t.homogenizationofgeneralsemilinear} and Theorem \ref{t.characterizeminimalsuperslutiontohomoextremalsolutions}.

\end{proof}

\section{The heat equation with singularly anisotropic pinned Neumann condition}\label{section.homogenizedparabolicflow}

In this section we make precise definition of the homogenized parabolic flow \eqref{eq.pGDwithmotionlaw}. For notational convenience, we will denote the homogenized coefficients as
\[
L_*(\gd'u) = m 1_{\{\gd'u=0\}} \quad \textup{and} \quad L^*(\gd'u) = M 1_{\{\gd'u=0\}},
\]
for some real numbers $m<0<M$. In the case of the homogenized equation \eqref{eq.pGDwithmotionlaw} we correspondingly take
\[
m=\min f, \ M = \max f \ \textup{ and } \   \avg{f}=0.
\]
Recall the space-time domain are defined as follows
\[
D_\infty^+:=B_1^+\times(0,\infty) \quad \textup{and} \quad D_\infty':=B_1'\times(0,\infty).
\]
We also write $$U\csubset B_1^+\cup B_1'$$ to be a relatively open domain, and we also denote 
$$
U'=U\cap B_1', \ U^+=U\cap B_1^+\ \textup{ and } \   \pd U= \pt U\setminus U'.
$$

\begin{definition}\label{d.definevissoltopGD}
An upper semicontinuous function \(u:\overline{D_{\infty}^+}\rta[-\infty,\infty)\) is called a viscosity subsolution to \eqref{eq.pGDwithmotionlaw} if for any smooth function $\phi$ crossing $u$ from above (See Definition \ref{def.crossing}) at $(x_0,t_0)\in D_\infty^+\cup D_\infty'$ in the cylindrical neighborhood $U\times (t_0-r,t_0+r)$, we have either (i)  $$\pt_t \phi(x_0,t_0)\le \Delta\phi(x_0,t_0) $$
    or (ii) $\pt_t \phi(x_0,t_0) > \Delta\phi(x_0,t_0)$, $(x_0,t_0) \in D_\infty'$, and the following conditions hold: 
    \begin{enumerate}[label=(\alph*)]
    
    \item \label{condition(a)} (\emph{Local stability}) The inner normal derivative satisfies
    \[
    \pt_1 \phi(x_0, t_0) \ge L_\ast(\gd' \phi(x_0, t_0)).
    \]
    \item \label{condition(b)} (\emph{Dynamic slope condition: transversal case}) If \(\pt_t \phi(x_0, t_0) > 0\), then 
    \[
    \pt_1 \phi(x_0, t_0) \ge 0.
    \]
    \item \label{condition(c)} (\emph{Dynamic slope condition: laminar case}) If \(\phi > u\) on \(\pt^+ U  \times \{t_0\}\) and satisfies $$\grad'\phi \equiv 0 \textup{ and } 
 \pt_t \phi(x_0, t_0)  > 0$$  then 
    \[
    \pt_1 \phi(x_0, t_0) \ge M.
    \]
    \end{enumerate}
A viscosity supersolution is defined symmetrically by reversing the crossing and inequalities, and replacing \(M\) and \(L_\ast\) with \(m\) and \(L^\ast\), respectively. A viscosity solution is defined as a continuous function on \(\overline{D_{\infty}^+}\) that is both a viscosity subsolution and a viscosity supersolution.
\end{definition}

\begin{remark}
    The dynamic slope conditions \ref{condition(b)} and \ref{condition(c)} form a rate-independent motion law (see, for example, \cite{MielkeRoubicek}), which is crucial in proving the comparison principle. Heuristically the motion law for a subsolution means that if $\pt_t u(x_0,t_0) >0$ for some $(x_0,t_0)\in B_1'\times(0,\infty)$ then with some strictness condition the maximal inner normal slope is saturated \emph{somewhere} on the component of the contact set containing $(x_0,t_0)$.
\end{remark}

The proof of our main theorem on parabolic homogenization, \tref{homogenizationofparabolic}, follows the idea of half-relaxed limit. An essential piece of the half-relaxed limit strategy is a semicontinuous comparison principle. Thus, one of the central results of this work, is a comparison principle / uniqueness property for \eqref{eq.pGDwithmotionlaw}.

Let us reformulate Theorem \ref{t.comparisonprincipleformal} here.

\begin{theorem}\label{t.comparisonprinciple}
    Let $u$ be a subsolution and $v$ a supersolution of \eqref{eq.pGDwithmotionlaw} in the sense of \dref{definevissoltopGD}. If $u\le v$ on the parabolic boundary $\ppp {D_{\infty}^+}$ then $u\le v$ on the whole space-time domain $ \overline{D_{\infty}^+}$.
\end{theorem}

\begin{remark}
    This comparison principle justifies that the viscosity solution conditions in \dref{definevissoltopGD} are sufficient to characterize the homogenized limit problem \eqref{eq.pGDwithmotionlaw}.
\end{remark}
 
The proof of \tref{comparisonprinciple} is postponed to Section \ref{section.comparisonprinciple}.  First we will apply \tref{comparisonprinciple} to establish the homogenization result \tref{homogenizationofparabolic} in the following section.

\section{Homogenization of the parabolic flow in the laminar case}\label{section.homoparabolic}

In this section, we consider the viscosity theoretic homogenization of the parabolic flow \eqref{eq.pNNep} as $\ep\rta0^+$
\be\tag{\ref{eq.pNNep}}
\bca
\pt_t u^\ep = \Delta u^\ep &\textup{ in } B_1^+\times(0,\infty)\\
\pt_1 u^\ep = f(u^\ep/\ep) & \textup{ on }B_1'\times (0,\infty),
\eca
\ee
where $f\in C^\alpha(\R)$ is a 1-periodic function.

We normalize $f$ by using Remark \ref{r.normalizeaveragezero} so that
\[
\avg{f}= \int_0^1 f(v) dv =0,
\]
and we take 
\[
m=\min f \quad \textup{and} \quad M=\max f
\]
in the \dref{definevissoltopGD}.

Recall in Section \ref{subsection.notations} we denote for all $T\in (0,\infty)$ the following space-time domains
\[
D_{T}^+:=B_1^+\times (0,T]\quad \textup{and} \quad  D_{T}':=B_1'\times(0,T]
\]
and we call
\[
\ppp D_{T}^+ := \overline{D_{T}^+}\setminus \lb D_{T}^+ \cup D_{T}'\rb
\]
the (positive) {parabolic boundary}. 

Let us restate Theorem \ref{t.homogenizationofparabolicformal} in a more precise way.

\begin{theorem}\label{t.homogenizationofparabolic}
Let $T\in(0,\infty)$, $f\in C^\alpha(\R)$ be a 1-periodic function for some $\alpha>0$ and $g\in C(\ppp D_\infty^+)$. If $u^\ep$ are solutions to \eqref{eq.pNNep} that satisfy $u^\ep=g$ on $\ppp D_\infty^+$, then $u^\ep$ converge uniformly on $\overline{D_T^+}$ to the unique solution $u$ to \eqref{eq.pGDwithmotionlaw} in the sense of Definition \ref{d.definevissoltopGD} that shares the same boundary data $g$ on $\ppp D_\infty^+$.

\end{theorem}

We show that the upper and lower half relaxed limits of $u^\ep$, $u^*$ and $u_*$, are, respectively, a sub and supersolution of \eqref{eq.pGDwithmotionlaw}. Then we apply the comparison principle \tref{comparisonprinciple} to show that $u^* \leq u_*$ and conclude that $u^* = u_*$.

\begin{remark}\label{r.regularityofparaatscaleep}
    For any $g\in C(\ppp D_{\infty}^+)$ and $\ep>0$, by Perron's method and standard parabolic regularity theory, there exists a classical solution of \eqref{eq.pNNep}
    $$u^\ep\in C(\overline{D_\infty^+})\cap C_{\textup{loc}}^{2+\alpha,1+\alpha/2}(D_{\infty}^+ )\cap C_{\textup{loc}}^{1+\alpha,1/2+\alpha/2}(D_{\infty}^+\cup D_\infty' )$$ satisfying the boundary condition $u^\ep=g$ on $\ppp D_{\infty}^+$. See Lemma \ref{l.usefulholderestimatepara} and Lemma \ref{l.usefulC1alphaestimatepara} for more details.
\end{remark}

Let us first consider an $L^\infty$ bound for $u^\ep$, which guarantees the existence of the half relaxed limits.
\begin{lemma}\label{l.uniformboundednessforuepparabolcproblem}
    Let $u^\ep$ be a viscosity solution to \eqref{eq.pNNep} and $u^\ep = g(x,t)$ on the parabolic boundary for some $g\in C\lb\ppp D_{T}^+\rb$, then there is a constant $C>0$ independent of $\ep$ such that
    \[
\norm{u^\ep}_{L^\infty\lb\overline{D_{T}^+}\rb} \le C \lb \norm{g}_{L^\infty\lb\ppp D_{T}^+\rb} + \norm{f}_{L^\infty\lb\R\rb}\rb.
    \]
\end{lemma}

\begin{proof}
This follows by comparison principle using supersolutions of \eqref{eq.pNNep} of the form
\[\|g\|_{L^\infty}+ \|f\|_{L^\infty} (1-x_1)\]
and symmetrically defined subsolutions.
\end{proof}

Given a fixed Dirichlet boundary data $g$ on the parabolic boundary $\ppp D_{T}^+$, we define the following two functions, the {upper half relaxed limit} of $u^\ep$
\be\label{eq.uupperstar}
u^\ast := {\limsup_{\ep\rta0}}^\ast u^\ep 
\ee
and the {lower half relaxed limit} of $u^\ep$
\be\label{eq.ulowerstar}
u_\ast:=\underset{\ep\rta0}{{\liminf}_\ast} u^\ep.
\ee
We now show that on the parabolic boundary $\ppp D_T^+$ we have $u^*=u_*=g$. See Section \ref{section.preliminaries} for the precise definitions of half relaxed limits.

\begin{lemma}\label{l.boundaryorderingg}
    Suppose $u^\ep$ are solutions as described in Lemma \ref{l.uniformboundednessforuepparabolcproblem} with boundary data $g\in C(\ppp D_T^+)$, then
    \[
    u^*= g = u_* \textup{ on }\ppp D_T^+.
    \]
\end{lemma}

\begin{proof}
    By using standard parabolic theory, we can find classical solutions $H_\pm$ to the following heat equation
    \be
    \bca
    \pt_t H_\pm = \Delta H_\pm & \textup{ in }D_T^+\\
    \pt_1 H_\pm = \pm \norm{f}_{L^\infty} & \textup{ on }D_T'\\
    H_\pm = g & \textup{ on }\ppp D_T^+.
    \eca
    \ee
    By comparison principle
    \[
    H_+\le u^\ep \le H_-,
    \]
    and therefore $u_*=u^* =g $ by the continuity of $H_\pm$.
\end{proof}

In the following three subsections we show that the half relaxed limits $u^*$,$u_*$ are respectively the viscosity sub/supersolutions to \eqref{eq.pGDwithmotionlaw} in the sense of \dref{definevissoltopGD}. In Section \ref{subsec.proofoftheorem1.1}, we finish this section by proving Theorem \ref{t.homogenizationofparabolic} assuming the comparison principle in Theorem \ref{t.comparisonprinciple}.

\subsection{Local stability condition}\label{subsec.localstability}

\begin{lemma}\label{l.conditiona}
    The function $u^\ast = \limsup^*u^\ep$ solves
   \[
   \partial_1 u^{\ast}(x,t) \geq L_*(\gd' u^*) \ \hbox{  on } \ B_1' \times (0,T)
   \]
   in the sense of condition \ref{condition(a)} from \dref{definevissoltopGD}. Similarly, $u_\ast = \liminf_*u^\ep$  satisfies the symmetrical supersolution condition.
\end{lemma}

\begin{remark}
    The proof of this lemma also works for the general semilinear case $f(x',v)$ that is periodic with respect to $\Z^{d-1}\times \Z$. In this case similar arguments show
    \[
    \pt_1 u^* \ge Q_*(\gd' u^* ; f) \quad \textup{and} \quad  \pt_1 u_* \le Q^\ast(\gd' u_* ; f)
    \]
    in the viscosity sense, where $Q_*$ and $Q^*$ are homogenized coefficients as defined in \eqref{eq.subcoefficient} and \eqref{eq.supercoefficient}. It is an interesting open question to characterize similar conditions to \dref{definevissoltopGD} \ref{condition(b)} and \ref{condition(c)} that provide a comparison principle for the general semilinear case.
\end{remark}

In order to prove Lemma \ref{l.conditiona} we need to use the following technical lemma, which we prove after the proof of Lemma \ref{l.conditiona}.

\begin{lemma}\label{l.intermediateofconditiona}
    Suppose $v:\overline{\R_+^d}\times\R\rta[-\infty,\infty)$ is a viscosity subsolution to
\be\label{eq.parabolictouchingwithsmooth4}
\begin{cases}
\partial_tv \leq \Delta v & \text{ in } {\R_+^d}\times \R\\
\partial_1 v \ge f(v +c) & \text{ on } \pt{\R_+^d}\times\R \\
v(y, t) \le p \cdot y & \text{ in } \overline{\R_+^d}\times\R\\
v(0,0)=0\ \textup{ and } \    v(y,t)=-\infty \textup{ for }t>0, & 
\end{cases}
\ee
where $c$ is an arbitrary constant, then 
\[
p_1 \ge L_*(p')=(\min f) 1_{\{p'=0\}}.
\]
\end{lemma}

\begin{proof}[Proof of Lemma \ref{l.conditiona}]

We prove that $u^\ast$ satisfies the subsolution condition \ref{condition(a)} for \( u^\ast \), the supersolution condition \ref{condition(a)} for \( u_\ast \) follows from a symmetric argument. Suppose that $\psi$ is a smooth test function that crosses \( u^\ast \) strictly from above (by using standard perturbations) at \( (x_0, t_0) \) in $U\times (t_0-r,t_0+r)$ with 
\be\label{eq.lemmacaseasuperheatoftestpsi}
\partial_t \psi(x_0, t_0)>\Delta \psi(x_0, t_0) .
\ee
By the smoothness of \( \psi \), we may constrain to smaller domain $U$ and radius $r>0$ so that this strict supersolution property \eqref{eq.lemmacaseasuperheatoftestpsi} holds also in $U\times(t_0-r,t_0+r) $ that contains \( (x_0, t_0) \), where $U\csubset B_1^+\cup B_1'$ is relatively open. 

By Lemma \ref{l.stabilityofcrossing}, we can find a subsequence (not relabelled) of $\ep\rta 0^+$, $c_\ep\rta0$ and a sequence of points $$(x_\ep,t_\ep)\in U\times (t_0-r,t_0+r)$$ that converge to $(x_0,t_0)$ so that for each $\ep$ the function $\psi+c_\ep$ crosses $u^{\ep}$ from above at $(x_\ep,t_\ep)$ in $U\times (t_0-r,t_0+r)$. Because $\psi$ is a strict supersolution and $u^{\ep}$ is a subsolution to the heat equation, the strong comparison principle implies that $(x_\ep,t_\ep) \in U'\times (t_0-r,t_0+r)$. This shows that $x_0\in U'=U\cap B_1' $.

Now we aim to show that
\[
\partial_1 \psi(x_0, t_0) \ge L_\ast(\nabla' \psi(x_0, t_0)).
\]
Consider the (parabolic) rescalings:
\[
v_\ep(y, t) := \frac{u^\ep(\ep y + x_\ep, \ep^2 t + t_\ep) - u^\ep(x_\ep, t_\ep)}{\ep}\quad \text{and} \quad
\psi_\ep(y, t) := \frac{\psi(\ep y + x_\ep, \ep^2 t + t_\ep) - \psi(x_\ep, t_\ep)}{\ep}.
\]
Note that
\[\psi_\ep(y,t) \leq p \cdot y  + O(\ep|y|^2+\ep |t|),\]
where \( p := \nabla \psi(x_0, t_0) \).  So, for any \( R > 0 \), there exists a small \( \ep_0 > 0 \) such that for all \( 0 < \ep < \ep_0 \), the $v_\ep$ solve in the classical sense
\be\label{eq.parabolictouchingwithsmooth}
\begin{cases}
  \pt_t v_\ep = \Delta v_\ep & \text{ in } {B_R^+} \times (-R^2, 0] \\
\partial_1 v_\ep = f\lb v_\ep + \frac{u^{\ep}(x_{\ep}, t_{\ep})}{\ep}\rb & \text{ on } B_R' \times (-R^2, 0] \\
v_\ep(y, t) \le p \cdot y +O(\ep R^2) & \text{ in } \overline{B_R^+} \times [-R^2, 0]\\
v_\ep(0,0)=\psi_\ep(0,0)=0, & 
\end{cases}
\ee
where the final two properties follow since $\psi_\ep$ crosses $v_\ep$ from above at $(0,0)$.

Choose a subsequence \( \ep_j \to 0 \) so that the limit exists
\[
\lim_{j \to \infty} \left[ \frac{u^{\ep_j}(x_{\ep_j}, t_{\ep_j})}{\ep_j} - \left\lfloor \frac{u^{\ep_j}(x_{\ep_j}, t_{\ep_j})}{\ep_j} \right\rfloor \right] = c\in[0,1].
\]
For each fixed radius $R>0$ the following upper half-relaxed limit, by the third condition in \eqref{eq.parabolictouchingwithsmooth}, is well-defined on $\overline{B_R^+} \times [-R^2, 0]$
\[
v^\ast := {\limsup_{j\rta \infty}}^\ast v_{\ep_j}.
\]
Because $R>0$ is arbitrarily chosen, the limit function $v^*$ is well-defined on all $\overline{\R_+^d}\times(-\infty,0]$.
Note that \( p \cdot y \) touches \( v^\ast \) from above at \( (0, 0) \) in \( \overline{\R_+^d}\times(-\infty,0] \) so \( v^\ast(y) \in [-\infty,\infty) \) on \( \overline{\R_+^d}\times(-\infty,0]  \) and $v^*(0,0)=0$. By extending $v^* = -\infty$ on $\overline{\R_+^d}\times (0,\infty)$, we can make $v^*$ an upper semicontinuous function on $\overline{\R_+^d}\times \R$. Moreover, if we also extend $v_\ep=-\infty$ on $\overline{B_R^+}\times (0,\infty)$, then
\[
v^*={\limsup_{j\rta \infty}}^\ast v_{\ep_j}
\]
on whole $\overline{\R_+^d}\times\R$.

By stability of touching in Lemma \ref{l.stabilityoftouching} and that $v_\ep$ are classical solutions to \eqref{eq.parabolictouchingwithsmooth}, \( v^\ast \) solves \eqref{eq.parabolictouchingwithsmooth4} with $p=\gd \psi(x_0,t_0)$. The proof is done by applying Lemma \ref{l.intermediateofconditiona}.

\end{proof}

\begin{proof}[Proof of Lemma \ref{l.intermediateofconditiona}]
We prove this lemma by constructing a subcorrector $V$ to Definition \ref{def.subcorrector} having slope $p$. To that end, we consider the following sequence of functions with $r>0$
$$
w_r(y,t) := v^*(y,t/r),
$$
and their upper half relaxed limit (which is well-defined because $v^\ast(y, t) \le p \cdot y $)
$$
w(y,t):={\limsup_{r\rta0^+}}^\ast w_r(y,t).
$$
The function $w$ is upper semicontinuous and takes the form
\[
w(y,t)=\bca
V_0(y) & t=0\\
V_1(y) & t<0\\
-\infty &t >0,
\eca
\]
with $V_0\ge V_1$. 

We claim that $V=V_0$ is a subcorrector to Definition \ref{def.subcorrector} with slope $p$. By the bound $V_0(0)\ge0$ and $V_0(y)\le p\cdot y$, it suffices to show that $V_0$ is a subsolution. Suppose $\eta(y)$ is a smooth function touching $V_0$ strictly from above at $y_0\in \overline{\R_+^d}$, then we can extend 
\[
\eta(y,t)=\eta(y)+t^2
\]
and know that $\eta(y,t)$ touches $w$ strictly from above at $(y_0,0)$. By Lemma \ref{l.stabilityoftouching} again, there are 
\[
\textup{constants }c_r\rta0\textup{ and points }(y_r,t_r)\rta (y_0,0)
\]
such that $\eta+c_r$ touches $w_r$ strictly from above at $(y_r,t_r)$. Note that the functions $w_r$ satifies
$$
\begin{cases}
r\partial_t w_r \leq \Delta w_r & \text{ in } {\R_+^d}\times \R \\
\partial_1 w_r\ge f(w_r + c) & \text{ on } \pt{\R_+^d}\times\R.
\end{cases}
$$
This implies that either 
$$
r\pt_t\eta(y_r,t_r)=2rt_r \le \Delta \eta(y_r,t_r)
$$
or $y_r\in \pt \R_+^d$ and
$$\pt_1 \eta(y_r,t_r) \ge f(\eta(y_r,t_r)+c_r +c).$$
Sending $r\rta0^+$ shows that $V_0$ is indeed a subcorrector to Definition \ref{def.subcorrector} with slope $p$. Note that the translating constant $c$ does not change the homogenized coefficients due to Remark \ref{r.translationinvarianceofhomogenizedcoefficients}.

\end{proof}

\subsection{Dynamic slope condition: transversal case}\label{subsec.dy1}
In this subsection we prove that the half relaxed limits $u^*$, $u_*$ satisfy the condition \ref{condition(b)} in \dref{definevissoltopGD}.

\begin{lemma}\label{l.conditionb}
   The function $u^\ast = \limsup^*u^\ep$ solves
   \[\hbox{if } \ \partial_tu^*(x,t) > 0 \ \hbox{ then } \ \partial_1 u^{\ast}(x,t) \geq 0 \ \hbox{  on } \ B_1' \times (0,T)\]
   in the sense of condition \ref{condition(b)} from \dref{definevissoltopGD}. Similarly, $u_\ast = \liminf_*u^\ep$  satisfies the symmetrical supersolution condition.
\end{lemma}

\begin{proof}
As before we only prove that $u^\ast$ satisfies the subsolution condition \ref{condition(b)} for \( u^\ast \), and the supersolution condition \ref{condition(b)} for \( u_\ast \) follows from a symmetric argument. By a similar argument to the proof of the previous lemma, we may start with $\psi$ that is a smooth function crossing $u^*$ strictly (by standard perturbations) from above at $(x_0,t_0)$ and satisfies the following conditions in $U_r(x_0)\times (t_0-r,t_0+r)$
\be\label{eq.lemmacaseasuperheatoftestpsi2}
 \partial_t \psi>\Delta \psi, \ \pt_1\psi< 0 \quad \textup{and} \quad \pt_t\psi>0.
\ee
Here $U_r(z):=B_r(z)\cap \overline{\R_+^d}$. By the local stability condition \ref{condition(a)}
\be\label{eq.gradientoftouchinginconditionb}
\gd\psi(x_0,t_0)=-\nu e_1 \ \textup{ for some }\nu>0.
\ee

In the following we do a sequence of modifications on $\psi$ and $r$ so that all the above conditions are preserved and in $U_r(x_0)\times (t_0-r,t_0+r)$
\begin{enumerate}
    \item[(i)] 
    \(0<-\pt_1\psi \ll \pt_t \psi\),
    \item[(ii)] and $\psi$ is strictly superharmonic.
\end{enumerate}
We obtain (i) by adding $\psi$ with $-cx_1$ for 
\[
c=r+\sup_{U_r(x_0)\times (t_0-r,t_0+r)} \pt_1 \psi.
\]
As $r\rta0^+$, by smoothness of $\psi$, the modified functions $\psi-cx_1$ satisfy all the previous properties and also condition (i) above.

We obtain (ii) by replacing $\psi$ by
\[
\psi^\mu:= \psi + \mu x_1 - \frac{1}{\mu} x_1^2 
\]
for a small $\mu>0$. When $\mu>0$ is sufficiently small, $\psi^\mu$ is strictly superharmonic and all previous properties are preserved whenever $r<\mu^2/2$. In the following we omit the symbol $\mu$.

By Lemma \ref{l.stabilityofcrossing}, there is a subsequence, not relabeled, of  $\ep\rta0^+$, $c_\ep\rta0$ and a sequence of points $(x_\ep,t_\ep)\in U_r'(x_0)\times(t_0-r,t_0+r)$ so that $\psi_\ep:=\psi+c_\ep$ crosses $u^\ep$ from above at $(x_\ep,t_\ep)$ in $U_r(x_0)\times(t_0-r,t_0+r)$. Moreover, $(x_\ep,t_\ep)\rta (x_0,t_0)$. By using \eqref{eq.gradientoftouchinginconditionb} and condition (i) above, we obtain for some $a>0$ and all $\ep,r>0$ that are sufficiently small
\be\label{eq.wellseparatednessoftransfer}
\max_{x\in \overline{U_r(x_0)}} \psi_\ep(x,t_0-r) + ar< \min_{x\in \pd {U}_r(x_0)} \psi_\ep(x,t_\ep).
\ee
In the following context, we without loss assume that $\psi_\ep $ crosses $u^\ep$ at $(x_\ep,t_\ep)$ strictly from above by using standard perturbations.

Now we use energy minimization to construct a supersolution of the $\ep$ problem. First, we define $u=\xi^\ep$ to be the global energy minimizer to the following energy for some $\delta>0$ to be chosen:
\[
 \int_{U_r^+(x_0)} \frac{1}{2}|\gd u|^2 + \int_{U_r'(x_0)} \int_0^{u(x')} \lb f(v/\ep)-\delta \rb dv dx',
\]
subject to the boundary condition $u(x) = \psi_\ep(x,t_\ep)$ on $\pd U_r(x_0)$. The minimizer $\xi^\ep$ then satisfy the following problem
\be
\begin{cases}\label{eq.anintermediatexiproblem21}
\Delta \xi^\ep = 0 & \text{in } U_r^+(x_0), \\
\partial_1 \xi^\ep = f\left(\frac{\xi^\ep}{\ep}\right) - \delta & \text{on } U_r'(x_0), \\
\xi^\ep(x) = \psi_\ep(x,t_\ep) & x \in \pd U_r(x_0).
\end{cases}
\ee

We claim that for small $\ep, \ \delta>0$ the function $\xi^\ep$ crosses $u^\ep$ from above at some $(z_\ep,s_\ep) \in U_r'(x_0) \times(t_0-r,t_\ep) $. This will finish the proof because by the viscosity solution condition of $u^\ep$ and the smoothness of $\xi^\ep$
\[
\pt_1 \xi^\ep (z_\ep,s_\ep) \ge f(u^\ep(z_\ep,s_\ep))= f(\xi^\ep(z_\ep,s_\ep)),
\]
contradicting the boundary condition of $\xi^\ep$ at $(z_\ep,s_\ep)$ according to \eqref{eq.anintermediatexiproblem21}.

To prove the crossing property, we first observe that according to Lemma \ref{l.localuniformconvergenceofglobalminimizers}, we see that \( \xi^\ep \) converges uniformly to \( \xi \) in \( \overline{U_r(x_0)}\) solving
\be\label{eq.anintermediatexiproblem}
\begin{cases}
\Delta \xi = 0 & \text{in } U_r^+(x_0), \\
\partial_1 \xi =- \delta & \text{on }U_r'(x_0), \\
\xi(x) = \psi(x,t_0)& x \in \pd U_r(x_0).
\end{cases}
\ee
Combining \eqref{eq.wellseparatednessoftransfer} and \eqref{eq.anintermediatexiproblem}, we obtain
\[
\begin{split}
   \max_{x\in \overline{U_r(x_0)}} \lb u^\ep(x,t_0-r)-\xi^\ep(x) \rb &\le  \max_{x\in \overline{U_r(x_0)}} \lb\psi_\ep(x,t_0-r)-\xi^\ep(x)\rb\\
    &< \min_{x\in \pd U_r(x_0)} \psi_\ep(x,t_\ep)- \min_{x\in \overline{U_r(x_0)}}\xi^\ep(x)-a r\\
    &< \min_{x\in \pd U_r(x_0)} \psi(x,t_0)- \min_{x\in \overline{U_r(x_0)}}\xi(x)-a r + o_\ep(1)\\
    &\le -a r + o_\ep(1)\\
    &<0
\end{split}
\]
when $\ep>0$ is small. This shows that 
\be\label{eq.boundaryorderingtechnicalconditionb}
\xi^\ep(x) > u^\ep(x,t) \textup{ for }(x,t)\in \ppp \lb U_r(x_0)\times(t_0-r,t_\ep) \rb.
\ee
On the other hand, when $\delta>0$ is small we have for $x\in U_r'(x_0)$
\[
\pt_1\xi(x)=-\delta> \pt_1 \psi(x,t_0),
\]
and because $\xi$ is harmonic and $\psi$ is superharmonic we obtain by strong comparison principle that $\xi<\psi$ on $U_r(x_0)$ and therefore for small $\ep>0$
\be\label{eq.reverorderingatendtimexeptepconditionb}
\xi^\ep(x_\ep,t_\ep)<\psi_\ep(x_\ep,t_\ep)=u^\ep(x_\ep,t_\ep).
\ee
Let $s_\ep$ be the infimum of $s\in (t_0-r,t_\ep)$ so that $u^\ep(x,s)-\xi^\ep(x,s)$ is positive for some $x\in \overline{U_r(x_0)}$. By \eqref{eq.boundaryorderingtechnicalconditionb} and \eqref{eq.reverorderingatendtimexeptepconditionb}, $s_\ep\in (t_0-r,t_\ep)$. By continuity, $\xi^\ep$ will cross $u^\ep$ from above at some $(z_\ep,s_\ep)\in U_r(x_0)\times(t_0-r,t_\ep)$, but by strong comparison principle, $z_\ep$ will not belong to $U_r^+(x_0)$ and hence $z_\ep\in U_r'(x_0)$ as we desired.

\end{proof}

\subsection{Dynamic slope condition: laminar case}\label{subsec.dy2}
In this subsection we prove that the half relaxed limits $u^*$, $u_*$ satisfy  \dref{definevissoltopGD} \ref{condition(c)}.
\begin{lemma}\label{l.conditionc}
     The function $u^\ast$ defined in \eqref{eq.uupperstar} satisfies \dref{definevissoltopGD} \ref{condition(c)}. Similarly, $u_\ast$ defined in \eqref{eq.ulowerstar} satisfies the supersolution version of \dref{definevissoltopGD} \ref{condition(c)}.
\end{lemma}

\begin{proof}
As before we only prove that $u^\ast$ satisfies the subsolution condition  \dref{definevissoltopGD} \ref{condition(c)} for \( u^\ast \). The supersolution condition \dref{definevissoltopGD} \ref{condition(c)} for \( u_\ast \) follows from a symmetric argument. By a similar argument to the proof of the previous two lemmas, we may start with $\psi=\psi(x_1,t)$ that is a smooth spatially one-variable function, crossing $u^*$ from above and satisfies the following conditions in $U\times (t_0-r,t_0+r)$
\be\label{eq.lemmacaseasuperheatoftestpsi3}
 \partial_t \psi>\pt_1^2 \psi, \ \pt_1\psi< \max f \ \textup{ and } \   \pt_t\psi>0.
\ee
By using Lemma \ref{l.stabilityofcrossing} we also have $\ep\rta0^+, c_\ep\rta0$ and a sequence of points $(x_\ep,t_\ep)\in U'\times(t_0-r,t_0+r)$ so that $\psi_\ep:=\psi+c_\ep$ crosses $u^\ep$ from above at $(x_\ep,t_\ep)$ in $U\times(t_0-r,t_0+r)$. The domain $U\csubset B_1^+\cup B_1'$ is relatively open. 

Furthermore, by the assumption $\psi>u^*$ on $\pd U\times\{t=t_0\}$, we also have, combining $\pt_t\psi>0$ in \eqref{eq.lemmacaseasuperheatoftestpsi3}, that for some $c>0$ independent of $\ep$
\be\label{eq.psiepstrictorderconditionc}
\psi_\ep(x_1,t_\ep)\ge c+u^\ep(x,t) \textup{ for } (x,t)\in \ppp \lb U \times (t_0-r, t_\ep)\rb.
\ee
For a small $\mu>0$, we replace $U$ by
\[
U_\mu:= U \cap \{x_1 < \mu\}.
\]
This replacement preserves all the previous properties and when $\mu>0$ is chosen sufficiently small we also have
\be\label{eq.linearderivativecontrolinconiditoonc}
\psi (x_1,t) \le \psi(0,t) + (\max f - \beta) x_1,
\ee
for some small $\beta>0$ and all $x\in U_\mu$, $t\in [t_0-r,t_0+r]$. In the following we will omit the symbol $\mu$.

Let \( r_\ep \in \R \) be the maximal number such that
\[
r_\ep \le u^\ep(x_\ep, t_\ep) \quad \text{and} \quad f\left(\frac{r_\ep}{\ep}\right) = \max f =: M.
\]
By the periodicity of \( f \), we have
\[
|r_\ep - u^\ep(x_\ep, t_\ep)| \le \ep.
\]
We construct
\[
\phi_\ep(x_1, t) \equiv \phi_\ep(x_1) := (M - \ep) x_1 - \ep x_1^2 + r_\ep-\ep.
\]
Then we have by \eqref{eq.psiepstrictorderconditionc} and \eqref{eq.linearderivativecontrolinconiditoonc}
\[
\phi_\ep(x_1, t) \ge \psi_\ep(x_1,t_\ep) -2\ep > u^\ep(x,t) \textup{ for }(x,t)\in\ppp \lb U\times(t_0-r,t_\ep) \rb
\]
when $\ep >0$ is sufficiently small. On the other hand, we know that
\[
\phi_\ep(0,t_\ep)=r_\ep -\ep< u^\ep(x_\ep,t_\ep).
\]
By a similar argument to the last part of the previous lemma, we see that $\phi_\ep$ crosses $u^\ep$ from above at some $(z_\ep,s_\ep)\in U' \times(t_0-r,t_\ep)$, which is impossible because otherwise by the viscosity solution condition of $u^\ep$
\[
\pt_1\phi_\ep(0,s_\ep)\ge f (u^\ep(z_\ep,s_\ep)/\ep) = f(\phi_\ep(0,s_\ep)/\ep)=f(r_\ep/\ep+1) = \max f,
\]
which contradicts the definition of $\phi_\ep$. 

\end{proof}

\subsection{Proof of the parabolic homogenization} \label{subsec.proofoftheorem1.1}

Now we can combine the elements above to prove Theorem \ref{t.homogenizationofparabolic}.

\begin{proof}[Proof of Theorem \ref{t.homogenizationofparabolic}]
Let $u^*$ and $u_*$ be the upper and lower half-relaxed limits of $u_\ep$, as defined in \eqref{eq.uupperstar} and \eqref{eq.ulowerstar} respectively. By Lemma \ref{l.conditiona}, Lemma \ref{l.conditionb} and Lemma \ref{l.conditionc}, $u^*$ is a subsolution and $u_*$ is a supersolution of \eqref{eq.pNNep} in the sense of \dref{definevissoltopGD}. By Lemma \ref{l.boundaryorderingg} and the comparison principle \tref{comparisonprinciple}, $u_*\geq u^*$. Since $u_* \leq u^*$ by definition, then $\bar{u}:= u_* = u^*$ is continuous on $\overline{D_\infty^+}$, is a viscosity solution of \eqref{eq.pGDwithmotionlaw}, and $u^\ep$ converge locally uniformly on $\overline{D_\infty^+}$ to $\bar{u}$, .

\end{proof}

\section{The parabolic comparison principle}\label{section.comparisonprinciple}

In this section, we discuss the proof of comparison principle \tref{comparisonprinciple} for \eqref{eq.pGDwithmotionlaw}. In Section \ref{subsec.contactsets}, we prove the openness of the facets/contact sets, which is a crucial step because these sets consist of points where the Neumann condition becomes degenerate. In Section \ref{subsec.proofofcomparisonprinciple}, we first give a sketch of proof for Theorem \ref{t.comparisonprinciple} due to the length, and then provide a whole proof.

\subsection{Contact sets}\label{subsec.contactsets}
Because of the gradient degeneracy of the boundary condition
\[
\pt_1 u \in [m 1_{\{\gd' u = 0\}}, M 1_{\{\gd' u =0\}}],
\]
it turns out to be useful to consider the points $x\in D_\infty'$, at which the solution $u$ satisfies $\pt_1u(x)\ne0$. As is often the case in viscosity solution theory, we make this precise for weak solutions via sub and supersolution touching conditions.

\begin{definition}
    Given a subsolution $u$ to \eqref{eq.pGDwithmotionlaw} as defined in \dref{definevissoltopGD}, we partition $D_{\infty}'$ as 
    \be
    D_{\infty}' = \Ca_-(u) \sqcup \Na_-(u) \sqcup \Gamma_-(u),
    \ee
    where
    \be\label{eq.definitionforCa}
          \Ca_-(u):=\lma (x_0,t_0)\in D_\infty' \ ; \     {\begin{aligned}& \textup{ there is a smooth function } \phi \\
         & \textup{ crossing }u \textup{ from above at }(x_0,t_0),\\
         &\quad \textup{and} \quad \pt_1 \phi(x_0,t_0)<0.    \end{aligned}}
         \rma
    \ee
    and 
    \be\label{e.non-contact-fb-defn}
\Na_-(u):=D_\infty'\setminus\overline{\Ca_-(u)}\ 
  \quad \textup{and} \quad \Gamma_-(u):= D_\infty'\setminus\lb \Ca_-(u)\cup \Na_-(u)\rb.
    \ee
The sets $\Ca_+(v), \Na_+(v)$ and $\Gamma_+(v)$ corresponding to a supersolution $v$ are defined symmetrically. We will call $\Ca_-(u)$ and $\Ca_+(v)$ {contact sets} (or {facets}) for $u$ and $v$ respectively.

\end{definition}

\begin{remark}
    The terminology contact set is used due to (a somewhat distant) relationship with the contact set in the thin obstacle problem. See also \cite{feldman2024regularitytheorygradientdegenerate} for more details on this connection. 
\end{remark}

Notice that if the contact set $\Ca$ is open, then the interface $\Gamma$ will be the common boundary of the two disjoint open sets $\Ca$ and $\Na$. Unfortunately, in contrast to the elliptic case \cite{feldman2024regularitytheorygradientdegenerate}*{Lemma 2.8}, for general subsolutions (or supersolutions) we are unable to show that the contact sets defined above are open in $D_{\infty}'$. This is due to the lack of regularity in time. In the following lemma we show that $\Ca_-(u)$ is open when $u$ is Lipschitz in time.

In the following we write $U\csubset B_1^+\cup B_1'$ a relatively open domain and we also write
\[
U^+=U\cap B_1^+\quad \text{and} \quad U'=U\cap B_1'.
\]

\begin{lemma}
 \label{l.opennessofcontactsetwhentimelipschitzx}
Suppose that $u$ is a subsolution of \eqref{eq.pGDwithmotionlaw} on $ D_\infty^+\cup D_\infty' $ with $$\|\partial_tu\|_{L^\infty(U\times (t_1,t_2))} < + \infty.$$ Then $\Ca_-(u) \cap U'\times(t_1,t_2)$ is relatively open in $D_\infty'$ and $\nabla'u = 0$ on $\Ca_-(u) \cap U'\times(t_1,t_2)$. A symmetric result holds for supersolutions.

\end{lemma}

{In the proof we will show that for any $(x_0+\ep z,t_0+\ep h)$ near enough to $(x_0,t_0) \in \mathcal{C}_-(u)$ we can create another test function with negative inward normal derivative touching $u$ from above \emph{exactly} at $(x_0+\ep z,t_0+\ep h)$.  Similar ideas have appeared before in \cite{feldman2024regularitytheorygradientdegenerate}*{Lemma 2.8}, \cite{changlara2017boundary}*{Lemma 3.1} and \cite{feldman2019free}*{Proof of Theorem 5.3, Step 3}. All of those examples were elliptic, the parabolic analogue is trickier since the time variable needs to be treated in a distinct way from the spatial variables. In order to force the touching test function to touch at a specific space-time point we need to use the Lipschitz hypothesis. The test function is created by bending upwards in the tangential and temporal directions the linearization of $u$ at $(x_0,t_0)$. By a Lipschitz bending in time we can force the touching time to be exactly $t_0+\ep h$. By the subsolution condition the new test function can only touch $u$ from above where its tangential derivative is zero, this will force the spatial location of the touching point to be $x_0+\ep z$. }
\begin{proof}

Let $(x_0,t_0)\in U'\times(t_1,t_2)$ be a point in $\Ca_-(u)$. By the definition there is a smooth function $\phi$ that crosses $u$ from above at $(x_0,t_0)$ with $\beta:=\pt_1 \phi(x_0,t_0)<0$. By the subsolution condition of $u$, specifically \dref{definevissoltopGD} \ref{condition(a)}, we have $\gd' \phi(x_0,t_0) =0$. Also call $\alpha:=\pt_t \phi(x_0,t_0)$. Call $L>0$ to be a uniform upper bound for the time derivative $\pt_t u$, so that, in particular, $|\alpha|\leq L$.

Consider
\[
\Tilde{\phi}(x,t):=\bca
\phi(x,t) & \textup{ if }t\le t_0\\
\phi(x,t_0) + L(t-t_0) & \textup{ if }t>t_0.
\eca
\]
Then $\Tilde{\phi}$ touches $u$ from above at $(x_0,t_0)$. In the following we replace $\phi$ by $\Tilde{\phi}$.

For convenience we consider the following rescalings for small $\ep>0$
\[
u^\ep(y,s) := \frac{u(x_0+\ep y  , t_0+\ep s)- u(x_0,t_0)}{\ep} \ \hbox{ and } \ \phi^\ep (y,s) := \frac{\phi( x_0+\ep y ,  t_0+ \ep s)- \phi(x_0,t_0)}{\ep}.
\]

When $\ep>0$ is chosen small, we have the following differential inequality (in the sense of viscosity solutions) and upper half flatness condition with $\alpha(s):=\alpha\min\{s,0\} + L\max\{s,0\}$
\[
\bca
\ep \pt_s u^\ep \leq \Delta_y u^\ep  & \textup{ in }B_1^+ \times (-1,1)\\
\textup{subsolution condition as in \dref{definevissoltopGD}} & \textup{ on }B_1'\times(-1,1)\\
u^\ep(y,s) \le \beta y_1 + \alpha(s) + \omega(\ep) & \textup{ on }\overline{B_1^+}\times [-1,1],
\eca
\]
where $\omega(\ep) \rta 0^+$ as $\ep \rta 0^+$.

In the following we construct a family of smooth test functions $v_{\tau,h,z}$ that, when $0<\ep\ll\delta$, touch $u^\ep$ from above at exactly $(z,h)\in B_1^+ \times (-1,1)$ for some appropriate parameter $\tau\in[0,1]$. We define
\[
\begin{split}
    v_{\tau,h,z}(y,s)&:=\alpha(s) + \frac{\beta}{2} y_1+\delta|y' -z|^2  - d\delta  y_1^2+3L|s-h|  \\
    &\quad-  \lb\delta |z|^2 +3L|h|\rb \tau + \omega(\ep) (1-\tau),
\end{split}
\]
where $\delta< \min\{1,|\beta|\}/(100d)$. Notice that for $|s|\le 1$, $|y|\le1$ and when $0<\ep\ll \delta$
\[
\Delta_y v_{\tau,h,z} =-2\delta< -4L\ep\le \ep \pt_s v_{\tau,h,z},
\]
in the sense of viscosity solutions. For $(y,s)\in \ppp \lb B_1^+\times(-1,1) \rb \cup \overline{ B_1^+} \times \{s=1\}$, if $|z|<1/4$ and $|h|<\frac{\delta}{32L}$ we have $v_{1,h,z}> u^\ep$ because
\[
\begin{split}
    v_{1,h,z}- \lb \beta y_1 +\alpha( s) + \omega(\ep) \rb &=\lb 3L|s-h| -3L|h|\rb-\frac{\beta}{2} y_1 + \delta|y'-z|^2 - d\delta y_1^2 - \delta|z|^2 - \omega(\ep)\\
&\ge 3L|s|-6L|h|+ \lb -\frac{\beta}{2} - (d+1)\delta\rb y_1 + \delta |y'|^2\\
&\quad  -2\delta y'\cdot z + \delta y_1^2 - \omega(\ep)\\
&\ge \min\{3L, \delta\}-6L|h| - 2\delta |z| -\omega(\ep)\\
&\ge \frac{\delta}{4} - 6L|h| \\
&>0.
\end{split}
\]
On the other hand we have $v_{1,h,z}(0,0)=0$, which means that the graph of $v_{\tau,h,z}$ intersects with $u^\ep$ as $\tau $ approaches 1 from the negative side. By the half flatness condition we know that 
\[
v_{0,h,z} \ge \alpha (s )+ \frac{3\beta}{4} y_1  + \omega(\ep)
\]
stays above $u^\ep$, which combining the strong comparison principle and the above two properties implies that there exists a maximal $\tau^\ast=\tau^*(z,h)\in [0,1]$ such that $v_{\tau^\ast,h,z}$ touches $u^\ep$ from above at some $$(y_0,s_0)\in   B_1'\times (-1,1).$$  We claim that the touching point can only be $$(y_0,s_0)=(z,h).$$ If $s_0\ne h$ then
\[
|\pt_s v_{\tau^\ast,h,z}(y_0,s_0)| \ge 2L >L,
\]
where this inequality is interpreted as a supersolution condition in the viscosity sense, in other words the lower bound inequality holds on the temporal component of elements of the subdifferential. This violates that $u^\ep$ is $L$-Lipschitz regular in time. On the other hand, by the strong comparison principle we have $y_0\in B_1'$ and if $y_0\ne z$, then
\[
\gd' v_{\tau^\ast,h,z}(y_0,h) = 2\delta( y_0-z) \ne 0,
\]
which contradicts the subsolution condition of $u^\ep$ at $(y_0,h)$ (See condition \ref{condition(a)} in \dref{definevissoltopGD}) since $\pt_1 v_{\tau^\ast,h,z}(y_0,h)=\frac{\beta}{2}<0$.  This implies that $\Ca_-(u)\cap U\times (t_1,t_2)$ is open.

To show that $\gd' u =0$ in $\Ca_-(u)\cap U\times (t_1,t_2)$, we observe that $v_{\tau^*,h,z}(\cdot,h)$ touches $u(\cdot, h)$ from above exactly at $z$, which is the spatial minimum of a parabola and by arbitrariness of $z\in \Ca_-(u)\cap U\times\{h\}$, we see that $u(\cdot,h)$ is $C^{1,1}$ from one side and has (tangential) gradient 0 everywhere on $\Ca_-(u)\cap U\times\{h\}$, which forces $\gd'u=0$ on $\Ca_-(u)\cap U\times\{h\}$.
\end{proof}

\subsection{The comparison principle for the homogenized parabolic problem}\label{subsec.proofofcomparisonprinciple}

In order to exhibit the core ideas we present the sketch of proof of Theorem \ref{t.comparisonprinciple}. In order to present a clear sketch we will need to allow some slightly incorrect statements, which will then be clarified in the detailed proof.

\begin{proof}[Sketch of proof of Theorem \ref{t.comparisonprinciple}]
We give a sketch of proof under the assumption that the subsolution $u$ and the supersolution $v$ are smooth. First we can perturb $u$ to 
\be\label{eq.strictifyinsketch}
u_\mu(x,t):= u(x,t)-2\mu+\mu x_1 - \frac{\mu}{T-t},
\ee
for an arbitrary end time $T>0$ and a small $\mu>0$. In particular this forces the maximum of $u_\mu-v$ on $\overline{D_T^+}$ to occur on the interior $D_T^+$ and not at the end time $t = T$. By the maximum principle for strict subsolutions of the heat equation, the maximum point can only occur on $B_1'\times(0,T)$.

Let $(x_0,t_0)\in B_1'\times(0,T)$ be a maximum point of $u_\mu - v$, then by a proper vertical translation of $u_\mu$ we may assume that $u_\mu$ touches $v$ from below at $(x_0,t_0) \in B_1'\times(0,T)$. Derivative tests imply
\[\partial_1 (u_\mu -v)(x_0,t_0) \leq 0\quad \textup{and} \quad  \partial_t(u_\mu - v)(x_0,t_0) = 0.\]
Using Lemma \ref{l.opennessofcontactsetwhentimelipschitzx}, there are only three cases:

\begin{center}
\begin{minipage}{0.5\textwidth} 
\begin{enumerate}
    \item[Case 1:] $(x_0,t_0)\not\in \Ca_-(u)\cup \Ca_+(v)$.
    \item[Case 2:] $(x_0,t_0) \in \Ca_-(u) \cap \Ca_+(v)$.
    \item[Case 3:] $(x_0,t_0)\in \Ca_-(u) \Delta \Ca_+(v)$.
\end{enumerate}
\end{minipage}
\end{center}

Case 1 can be excluded since at the touching point $(x_0,t_0)\in \ D_\infty' \setminus (\Ca_-(u)\cup \Ca_+(v))$ 
\[ 
\partial_1 u + \mu =\partial_1u_\mu\le\partial_1 v \leq 0   ,
\]
which shows that $\pt_1 u(x_0,t_0)<0$, contradicting the assumption that $(x_0,t_0)\not\in \Ca_-(u)$, which is defined in \eqref{eq.definitionforCa}.

Now we show that Case 2 is impossible by using the (transversal) dynamic slope condition \dref{definevissoltopGD} \ref{condition(b)}. Observe that by the definition of $\Ca_-(u)$ and $\Ca_+(v)$ respectively, see \eqref{eq.definitionforCa},
\be\label{eq.formalproofuvnormalderivativeatcaca}
\pt_1 u({x}_0,t_0)<0 \quad \textup{and} \quad \pt_1 v({x}_0,t_0) >0.
\ee
By the (transversal) dynamic slope condition \dref{definevissoltopGD} \ref{condition(b)}, we must have $\pt_t v(x_0,t_0) \ge 0$ because
\[
\hbox{if $\pt_t v ({x}_0,t_0)<0$ then $\pt_1 v({x}_0,t_0)\le 0$}
\]
which is a contradiction of \eqref{eq.formalproofuvnormalderivativeatcaca}. On the other hand, if
\[
0 \leq \pt_t v ({x}_0,t_0) = \pt_t  u_\mu({x}_0,t_0)
\]
then 
\[\pt_t  u({x}_0,t_0) = \pt_t  u_\mu({x}_0,t_0)+\frac{\mu}{(T-t_0)^2} \geq \frac{\mu}{(T-t_0)^2} > 0\]
and the (transversal) dynamic slope condition \dref{definevissoltopGD} \ref{condition(b)} implies
\[\partial_1u({x}_0,t_0)  \geq 0\]
again contradicting \eqref{eq.formalproofuvnormalderivativeatcaca}. This finishes Case 2.

We claim that Case 3 can be reduced to Case 1 and 2 by using the conditions in \dref{definevissoltopGD}. We only argue with the case $(x_0,t_0)\in \Ca_-(u)\setminus \Ca_+(v)$ as the other one is symmetrical. By Lemma \ref{l.opennessofcontactsetwhentimelipschitzx},  $\Ca_-(u)$ is open in $D_\infty'$. Let $\Omega$, open in $B_1'$, be the (relatively open) connected component of $\Ca_-(u)\cap\{t=t_0\}$ containing $(x_0,t_0)$. We claim that there is another point
$$
(\Tilde{x}_0,t_0)\in \pt'\Omega \cap B_1'\subset \Gamma_-(u)
$$ 
such that $u$ touches $v$ from below also at $(\Tilde{x}_0,t_0)$. Otherwise $u<v$ on $\pt'\Omega \cap B_1' \times\{t_0\}$. 

Now we argue {imprecisely} for the sketch and think of $u=u(x_1, t)$ as a spatially one-dimensional function near $\Omega$. Notice that under this assumption $u_\mu=u_\mu(x_1,t)$ is also spatially 1D. 

If 
$$
\pt_t u_\mu(x_0,t_0)=\pt_t v(x_0,t_0)<0,
$$
then by the condition \dref{definevissoltopGD} \ref{condition(c)} of $v$ as a supersolution we have
\[
\pt_1 u_\mu(x_0,t_0)=\pt_1 u(0,t_0) + \mu\le m
\]
contradicting the strict subsolution condition \dref{definevissoltopGD} \ref{condition(a)} of $u$. If otherwise $$\pt_t u_\mu(x_0,t_0)=\pt_t v(x_0,t_0)\ge0 $$ then
\[
\pt_t u(x_0,t_0) = \pt_t u_\mu (x_0,t_0) +\frac{\mu}{(T-t_0)^2}>0,
\]
which, by the condition \dref{definevissoltopGD} \ref{condition(b)} of the subsolution $u$ we know that 
\[
\pt_1 u(x_0,t_0)=\pt_1 u(0,t_0)\ge 0
\]
contradicting the assumption $\pt_1 u(x_0,t_0)<0$ because $(x_0,t_0)\in \Ca_-(u)$. This finishes the proof of the existence of $\Tilde{x}_0$.

Now to prove the claim that Case 3 can be reduced to Case 1 and 2, we observe that $(\Tilde{x}_0,t_0)\in \Gamma_-(u)$ is either contained in $\Ca_+(v)$ or outside $\Ca_+(v)$. If $(\Tilde{x}_0,t_0) \not\in \Ca_+(v)$ then we are in Case 1. If $(\Tilde{x}_0,t_0) \in \Ca_+(v)\cap \Gamma_-(u)$ then we can reduce to Case 2 because $\Ca_+(v)$ is open by Lemma \ref{l.opennessofcontactsetwhentimelipschitzx} and 
by definition of $\Gamma_-(u)$ we can choose points $(\hat{x}_0,t_0)$ in $\Ca_-(u)$ that has small distance to $(\Tilde{x}_0,t_0)$ and
\[
u_\mu(\hat{x}_0,t_0)=u_\mu(\Tilde{x}_0,t_0)=v(\Tilde{x}_0,t_0)=v(\hat{x}_0,t_0),
\]
where we have used the fact that $u$ and $v$ depend only on time on each component of $\Ca_-(u)$ and $\Ca_+(v)$ respectively, which implies that $u$ touches $v$ from below at $(\hat{x}_0,t_0)$ and 
\[
(\hat{x}_0,t_0)\in \Ca_-(u)\cap \Ca_+(v).
\]

\end{proof}

\begin{proof}[Proof of Theorem \ref{t.comparisonprinciple}]

Let us first introduce the strategy of the proof. The proof begins by regularizing the sub/supersolutions $u$ and $v$ via parabolic tangential sup/inf-convolutions to ensure Lipschitz regularity in time and tangential spatial variables. These regularized functions are replaced by their {caloric lifts} (heat equation solutions with the same boundary data), preserving sub/supersolution properties. Next we consider the maximum value \(M\) of the regularized $u-v$, with a penalizing term \(-2\mu+\mu x_1 - \frac{\mu}{T-t}\) to force strict ordering of derivatives and ensure the maximum point exists and only occurs on $B_1'\times(0,T)$. At potential maximum points, we show that the regularized $u$ and $v$ are differentiable and hence we can test the boundary condition of $u$ and $v$ in different cases (which were described above in the proof sketch) depending on whether the maximum point is in the contact set of $u$ and/or $v$. 

\vspace{0.2cm}

\noindent\emph{Step 1: Regularization of sub and supersolutions}

Since constants are solutions to \eqref{eq.pGDwithmotionlaw}, we may without loss assume that $u$ and $v$ are bounded by considering $\min\{v,K\}$ and $\max\{u,-K\}$ for a large $K>0$. For $\delta>0$ and $\overline{U}=\overline{D_{\infty}^+}$ we consider the parabolic tangential sup-convolution $\Ta^\delta u$ of $u$ and inf-convolution $\Ta_\delta v$ of $v$ as defined in Definition \ref{def.tptconvolution}. For any fixed end time $T>0$, and a small parameter $\theta>0$ we define 
$$
\Oa:=B_{1-\theta/2}^+ \times (\theta/2,T] \quad \textup{and} \quad \Oa':=B_{1-\theta/2}' \times (\theta/2,T].
$$
The functions $\Ta^\delta u$ and $\Ta_\delta v$ satisfy the following properties:
\begin{enumerate}[label=(\roman*)]
    \item \label{condition(i-1)} \emph{Viscosity solution conditions}: By Lemma \ref{l.touchingpropertyofconvolution} and Corollary \ref{cor.subsolutionsforregularizations}, for any small $\theta>0$ there is a small $\delta_0>0$ such that, for all $0<\delta<\delta_0$,  $\Ta^\delta u$ is a subsolution and $ \Ta_\delta v$ is a supersolution of \eqref{eq.pGDwithmotionlaw} on 
    $$
    \Oa\cup  \Oa'= \lb B_{1-\theta/2}^+\cup B_{1-\theta/2}'\rb \times (\theta/2,T].
    $$ 

    \item \label{condition(i)} \emph{Lipschitz regularity and semi-convexity in time and tangential variables}: By Lemma \ref{l.lipschitzofconvolution} and boundedness of $u$ and $v$, both $\Ta_\delta v$ and $\Ta^\delta u$ are Lipschitz in time $t$ and in the tangential variable $x'$ for all $(x,t)=(x_1,x',t)\in \overline{\Oa}$. The Lipschitz constant is independent of $(x,t)\in \overline{\Oa}$. Moreover, by the definition of sup/inf-convolutions, the function $\Ta_\delta v(x_1,x',t)$ is {$\frac{1}{\delta}$-}semi-convex in $(x',t)$ and $\Ta^\delta u$ is {$\frac{1}{\delta}$-}semi-concave in $(x',t)$ for any fixed $x_1\ge0$.

    \item \label{condition(ii)} \emph{Boundary ordering}: Define the fattened boundary
      $$
J_\theta:=\overline{B_1^+}\times[0,T]\setminus \lb B_{1-\theta}^+\cup B_{1-\theta}'\rb\times(\theta,T].
    $$ 
    By Lemma \ref{l.convergenceofsupconvolution} and the upper semicontinuity of $u-v$, we have
    $$
\limsup_{\theta\rta0,\delta\rta0}\max_{J_\theta}\lb \Ta^\delta u - \Ta_\delta v \rb\le\max_{\ppp D_T^+} (u - v).
    $$
    In particular, if $$\max_{\ppp D_T^+} (u - v)<0,$$ then for sufficiently small $\theta,\delta>0$ we have
  \be\label{eq.propertyboundaryorderinginproofofcomparisonprincipoke}
  \max_{J_\theta}\lb \Ta^\delta u - \Ta_\delta v \rb<0.
  \ee
  \item \label{condition(iii)} \emph{Caloric lifts}: By Lemma \ref{l.caloriclift} and Definition \ref{def.caloriclifts}, we consider on the domain $\Oa$ the caloric lifts of $\Ta^\delta u$ and $\Ta_\delta v$:
\[
\Tilde{u}^\delta := \Ha \Ta^\delta u \ \textup{ and } \   \Tilde{ v}_\delta := \Ha \Ta_\delta v, \textup{ on }{\Oa}={B}_{1-\theta/2}^+\times(\theta/2,T].
\]
By Lemma \ref{l.preserveviscocaloriclift}, both $\Tilde{u}$ and $\Tilde{v}$ are also sub and supersolutions to \eqref{eq.pGDwithmotionlaw} respectively.

By Lemma \ref{l.continuityuoflift} the lifts $\Tilde{u}^\delta$ and $\Tilde{v}_\delta$ are continuous on $\Oa\cup \Oa'$. Moreover, because the restrictions of $\tilde{u}$ and $\tilde{v}$ on $\Oa'$ are Lipschitz, by Lemma \ref{l.boundarylipschitzheatequation}, both $\tilde{u}$ and $\tilde{v}$ are Lipschitz near $\Oa'$.

In general $\tilde{u}$ and $\tilde{v}$ may not be continuous on the parabolic boundary $\ppp \Oa$. However, by Lemma \ref{l.caloriclift}, the upper semicontinuous envelope of $\tilde{u}^\delta$ on $\overline{\Oa}$ coincides with $\Ta^\delta u$ on $\ppp \Oa$. Similar holds for $\tilde{v}$. This implies that
\be\label{eq.ordingofliftedcaloricfunction}
\max_{\ppp \Oa}(\tilde{u}^\delta-\tilde{v}_\delta) = \max_{\ppp \Oa} (\Ta^\delta u -\Ta_\delta v )\le  \max_{J_\theta}\lb \Ta^\delta u - \Ta_\delta v \rb.
\ee

\item \label{condition(iv)} \emph{Contact sets}: Denote the contact sets for the subsolution $\tilde{u}^\delta$ and for the supersolution $\tilde{v}_\delta$, respectively,
    \be
    \Ca_-:= \Ca_-(\tilde{u}^\delta)\subset \Oa' \quad \hbox{and} \quad \Ca_+:= \Ca_+(\tilde{v}_\delta)\subset \Oa'.
    \ee
    By property \ref{condition(i)}, \ref{condition(iii)} and Lemma \ref{l.opennessofcontactsetwhentimelipschitzx}, both $\Ca_-$ and $\Ca_+$ are open relative to $\Oa'$. Similar to above in \eref{non-contact-fb-defn}, we also consider the non-contact sets and free boundaries: define $\Na_\pm:= \mathcal{O}' \setminus \overline{\Ca_\pm}$ and $\Gamma_\pm := \mathcal{O}' \setminus (\Ca_\pm \cup \Na_\pm)$ so that $\mathcal{O}'$ is decomposed as a disjoint union in two ways
    \[
   \Oa'=\Ca_\pm \sqcup \Na_\pm \sqcup\Gamma_\pm.
    \]
\end{enumerate}

\vspace{0.2cm}

\noindent\emph{Step 2: Comparison with auxiliary perturbations}

In this step we make perturbations to allow strictness of some of the inequalities arising from derivative tests. It suffices to show, for parameters satisfying $0<\delta\ll\theta\ll \mu \ll1\ll T$, that
\be\label{eq.centrialinequality}
M:=\max_{(x,t)\in \overline{\Oa}} \lb\Tilde{u}^\delta(x,t) -\Tilde{v}_\delta(x,t)  -2\mu+ \mu x_1 -\frac{\mu}{T-t}\rb\le 0,
\ee
Indeed, if the above inequality holds then because 
$$
\Tilde{u}^\delta -\Tilde{v}_\delta \ge \Ta^\delta u -\Ta_\delta v\ge u-v \textup{  on }\overline{\Oa},
$$
we also have
\be\label{eq.intermediateinequalitythatfinishestheproof1}
\max_{(x,t)\in \overline{\Oa}}\lb {u}(x,t) -{v}(x,t)  -2\mu+ \mu x_1 -\frac{\mu}{T-t}\rb \le 0.
\ee
On the other hand, by property \ref{condition(ii)} we have for sufficiently small $\theta>0$
\[
\sup_{(x,t)\in \overline{D_T^+}\setminus \overline{\Oa}}\lb {u}(x,t) -{v}(x,t)  -2\mu+ \mu x_1 -\frac{\mu}{T-t}\rb \le 0.
\]
Combining \eqref{eq.intermediateinequalitythatfinishestheproof1}, we obtain
\[
\max_{(x,t)\in \overline{D_T^+}}\lb {u}(x,t) -{v}(x,t)  -2\mu+ \mu x_1 -\frac{\mu}{T-t}\rb \le 0.
\]
Sending $\mu \to 0$ implies that $u\le v$ for $t<T$ . Then sending $T \to \infty$ we conclude that $u\le v$ on the whole $\overline{D_{\infty}^+}$.

To prove \eqref{eq.centrialinequality}, we argue by contradiction, assume that $M>0$.  The maximum is achieved at some point $(x_0,t_0) \in \overline{\Oa}=\overline{B_{1-\theta/2}^+}\times[\theta/2,T]$. By applying \eqref{eq.propertyboundaryorderinginproofofcomparisonprincipoke} and \eqref{eq.ordingofliftedcaloricfunction}, we obtain that for any small $\mu>0$ there are small $\theta_0,\delta_0>0$ such that for all $0<\theta<\theta_0$ and $0<\delta<\delta_0$
\[
\tilde{u}^\delta -\tilde{v}_\delta < 0 \textup{ on } \ppp \Oa.
\]
On the other hand, we have
$$
\Tilde{u}^\delta(x,t) -\Tilde{v}_\delta(x,t)  -2\mu+ \mu x_1 -\frac{\mu}{T-t}\rta -\infty
$$ 
as $t\rta T^-$, and therefore the maximum point
\[
(x_0,t_0)\in \lb{B_{1-\theta/2}^+\cup B_{1-\theta/2}'}\rb\times(\theta/2,T).
\]
Since $\Tilde{u}^\delta-2\mu+\mu x_1-\frac{\mu}{T-t}$ is a strict subsolution to the heat equation and $\Tilde{v}_\delta$ is a supersolution to the heat equation in the interior $\Oa$ then
\[(x_0,t_0)\in  B_{1-\theta/2}'\times(\theta/2,T).\]

Define
\[
U(x, t) := \Tilde{u}^\delta(x,t) - 2\mu +  \mu x_1 - \frac{\mu}{T - t} - M
\quad \text{and} \quad 
V(x, t) := \Tilde{v}_\delta(x,t) .
\]
By previous arguments we have established that \( U \) touches \( V \) from below at 
\[
(x_0, t_0) \in B_{1-\theta/2}' \times (\theta/2, T) .
\]

\vspace{0.2cm}

\noindent \emph{Step 3: Differentiability at touching (maximum) points}

In the previous steps we have regularized $u$ and $v$ to obtain $U$ and $V$, however, the regularization is only in the tangential directions and we need some additional argument to establish differentiability in the normal direction. To address this, we first observe that by the construction, $\restr{U}{\Oa'}$ and $\restr{V}{\Oa'}$ are respectively semi-concave and semi-convex on $\Oa'$, and therefore both of them are $C^{1,1}$ in $(x',t)$ variables at the touching point $(x_0,t_0)$. As $U$ and $V$ differ from $\tilde{u}$ and $\tilde{v}$ by addition of a smooth function near $(x_0,t_0)$, we have
\[
\textup{both }\restr{\tilde{u}}{\Oa'}\ \textup{ and } \  \restr{\tilde{v}}{\Oa'} \textup{ are }C^{1,1}\textup{ at }(x_0,t_0).
\]
We now apply Lemma \ref{l.differentiabilityatc11point} to $\tilde{u}$ and $\tilde{v}$ to obtain that both $\tilde{u}$ and $\tilde{v}$ are differentiable both in space and time at $(x_0,t_0)$. By conditions \ref{condition(a)} and \ref{condition(b)} in \dref{definevissoltopGD}, the derivatives satisfy
\be\label{eq.localstabilityatx0t0bydifferentiability}
\pt_1 \tilde{u}(x_0,t_0) \ge L_*(\gd' \tilde{u}(x_0,t_0)) ,
\ee
and
\be\label{eq.dynamicslopebatx0t0bydifferentiability}
\textup{If }\pt_t \tilde{u}(x_0,t_0)>0 \textup{ then }\pt_1 \tilde{u}(x_0,t_0)\ge0.
\ee
A symmetric result holds for $V=\tilde{v}$. 

As $U$ is a smooth perturbation of $\tilde{u}$  we know that $U$ also differentiable at $(x_0,t_0)$, and by \eqref{eq.localstabilityatx0t0bydifferentiability} and \eqref{eq.dynamicslopebatx0t0bydifferentiability}, we have
\be\label{eq.localstabilityatx0t0bydifferentiabilityforU}
\pt_1 U(x_0,t_0) \ge L_*(\gd' U(x_0,t_0)) +\mu ,
\ee
and
\be\label{eq.dynamicslopebatx0t0bydifferentiabilityforU}
\textup{If }\ \pt_t U(x_0,t_0)\ge -\frac{\mu}{(T-t_0)^2} \ \textup{ then } \ \pt_1 U(x_0,t_0)\ge \mu.
\ee
On the other hand, because $U$ touches $V$ from below at $(x_0,t_0)$, we have the following formulae
\be\label{eq.formulaefortouchingUV}
\pt_1 (U-V)(x_0,t_0)\le0, \ \gd'(U-V)(x_0,t_0)=0\ \textup{ and } \  \pt_t (U-V)(x_0,t_0)=0. 
\ee

\vspace{0.2cm}

\noindent\emph{Step 4: Case analysis on the location of the touching points}

By the decompositions of $\mathcal{O}'$ described Step 1 item \ref{condition(iv)}, one of the following holds:
\begin{center}
\begin{varwidth}{\linewidth}
\begin{enumerate}
    \item[Case 1.] \( (x_0, t_0) \not \in \Ca_-\cup \Ca_+\).
    \item[Case 2.] \( (x_0, t_0) \in \Ca_- \cap \Ca_+ \).
    \item[Case 3.] \((x_0, t_0) \in \Ca_-\Delta \Ca_+\).
\end{enumerate}
\end{varwidth}
\end{center}
We rule out each possibility case by case, finally obtaining a contradiction of the existence of an interior maximum point and establishing \eqref{eq.centrialinequality}. The arguments will follow the sketch presented earlier in the section, now filling in the missing technical details.

Case 1. In this case \( (x_0, t_0) \not \in \Ca_-\cup \Ca_+\). By definition of the contact sets $\Ca_-$ and $\Ca_+$, the functions $\tilde{u}$ and $\tilde{v}$ satisfy Neumann sub and supersolution conditions, respectively, at $(x_0,t_0) \not\in \Ca_-\cup\Ca_+$. By the differentiability of $\tilde{u}$ and $\tilde{v}$ at $(x_0,t_0)$, as discussed in Step 3,
\[
\pt_1 \tilde{u}(x_0,t_0) \ge 0 \ge \pt_1 \tilde{v}(x_0,t_0).
\]
This implies, by \eqref{eq.localstabilityatx0t0bydifferentiabilityforU}, that
\[
\pt_1 U(x_0,t_0) \ge \mu >0,
\]
while, on the other hand, by \eqref{eq.formulaefortouchingUV},
\[
\pt_1 U(x_0,t_0) \le \pt_1 V(x_0,t_0)=\pt_1 \tilde{v}(x_0,t_0) \le 0,
\]
giving a contradiction.

Case 2. Now $(x_0,t_0)\in \Ca_-\cap \Ca_+$. Call
\[
b:=\partial_t U(x_0,t_0) = \partial_tV(x_0, t_0).
\]
Either $b \geq 0$ or $b< 0$. If $b\ge 0$, then by \eqref{eq.dynamicslopebatx0t0bydifferentiabilityforU},
\[
\pt_1 U(x_0,t_0) \ge \mu, \textup{ or equivalently }\pt_1\tilde{u}(x_0,t_0) \ge 0,
\]
contradicting the fact that $(x_0,t_0)\in \Ca_-$. Indeed, for any smooth function $\phi$ crossing $\tilde{u}$ from above at $(x_0,t_0)$, we know that $\phi(\cdot,t_0)$ touches $\tilde{u}(\cdot,t_0)$ from above at $(x_0,t_0)$ and therefore
\[
\pt_1\phi(x_0,t_0) \ge \pt_1 \tilde{u}(x_0,t_0) \ge 0,
\]
which contradicts the definition of $\Ca_-$ in \eqref{eq.definitionforCa}. If $b<0$ then by \eqref{eq.dynamicslopebatx0t0bydifferentiability}, or rather the symmetric supersolution statement for $V=\tilde{v}$,
\[
\pt_1 V(x_0,t_0)=\pt_1 \tilde{v}(x_0,t_0) \le 0,
\]
which contradicts $(x_0,t_0)\in \Ca_+$.

Case 3. In the following we finish the proof by showing that Case 3 can be reduced to Case 1 and Case 2. To see this we focus on the case that the touching point 
$$
(x_0, t_0) \in \Ca_- \cap \left( \Na_+ \cup \Gamma_+ \right)  
$$
as the other case is symmetrical. It suffices to prove that there exists another touching point 
\be\label{eq.claiminproofofcpthatcase3}
(\Tilde{x}_0, t_0) \in \partial \Ca_- \cap \left( B_{1-\theta/2}' \times \{ t_0\} \right) \subset  \Gamma_- ,
\ee
at which \( U \) touches \( V \) from below. Let us first show that the existence of the point in \eqref{eq.claiminproofofcpthatcase3} implies a contradiction. The new touching point $(\Tilde{x}_0, t_0)$ either belongs to \(\Na_+\cup\Gamma_+\) or $\Ca_+$. By Case 1 above $(\Tilde{x}_0, t_0)$ does not belong to \(\Na_+\cup\Gamma_+\). Next suppose that $(\Tilde{x}_0, t_0)\in\Ca_+$. By property \ref{condition(i)} and Lemma \ref{l.opennessofcontactsetwhentimelipschitzx}, $\tilde{u}(\cdot,t_0)$, $\tilde{v}(\cdot,t_0)$ and therefore $U(\cdot,t_0)$, $V(\cdot,t_0)$ are constant on each component of $\Ca_-\cap\{t=t_0\}$ and $\Ca_+\cap\{t=t_0\}$ respectively. Therefore $V(\cdot,t_0)$ is constant in a (tangential) neighborhood of $\Tilde{x}_0$, while in any small (tangential) neighborhood of $(\Tilde{x}_0,t_0)$ there is a $(\hat{x}_0,t_0)\in \Ca_-$ such that 
\[
U(\hat{x}_0,t_0)=U(\Tilde{x}_0,t_0).
\]
This implies that $U$ also touches $V$ from below at $(\hat{x}_0,t_0)\in \Ca_-\cap\Ca_+$. This is Case 2 which we have already shown that it cannot occur.

Now we return to prove the existence of the touching point in \eqref{eq.claiminproofofcpthatcase3}.  This is where  \dref{definevissoltopGD} \ref{condition(c)} comes into play. In fact, we show a slightly stronger result that, if we denote \( \Omega_{t_0} \subset \Ca_- \cap \{t=t_0\} \) as the component (which by Lemma \ref{l.opennessofcontactsetwhentimelipschitzx} is relatively open in $B_{1-\theta/2}' \times \{t_0\}$) of the latter set that contains \( (x_0,t_0) \), then there must be a touching point 
\be\label{eq.claiminproofofcpthatcase3slightlystronger}
(\tilde{x}_0,t_0)\in \pt'\Omega_{t_0} \cap (B_{1-\theta/2}' \times \{t_0\}) \subset \Gamma_-.
\ee
To see this, we argue by contradiction and assume that such a point does not exist. Notice that the relative boundary satisfies the containment
\[
\partial' \Omega_{t_0} \subset \lb\Gamma_- \cap \{t = t_0\}\rb \cup \pt'B_{1-\theta/2}' \times \{t_0\}.
\]
By property \ref{condition(ii)}, \ref{condition(iii)} in Step 1 and the strict perturbations in Step 2, we know that $$
U < V \textup{ on }\pt'B_{1-\theta/2}' \times \{t_0\}. 
$$ 
Therefore, if \( (\Tilde{x}_0, t_0) \) does not exist as in \eqref{eq.claiminproofofcpthatcase3slightlystronger}, then we reduce to the following condition
$$
U < V \textup{ on }\partial' \Omega_{t_0} . 
$$ 
As we have pointed out in the sketch, the main difficulty here is to view $U$ as a spatially 1D function, which is definitely not one in general. To avoid this issue we need to construct a 1D function by using the fact that $\gd' U=0$ on $\Ca_-$.

We first observe that since \( \Ca_- \) is open in \( \Oa'\) and $t_0<T$, there exists a small \( r_0 > 0 \) and a modulus of continuity \( \omega(r)\downarrow 0 \) such that for any \( r_0 > r > 0 \),
\[
 \left\{ x \ ; \ (x, t_0) \in \Omega_{t_0}, \ \dist((x, t_0), \partial' \Omega_{t_0}) > r \right\} \times (t_0 - \omega(r), t_0 + \omega(r)) \csubset \Ca_-.
\]
We write for $r,h>0$ the following spatial domains
\[
\Omega^r:=\left\{ x \ : \ (x, t_0) \in \Omega_{t_0}, \ \dist((x, t_0), \partial' \Omega_{t_0}) > r \right\} \ \textup{ and } \   \Omega_h^r:=\{0<x_1<h\}\times \Omega^r
\]
When \( r_0 > 0 \) is sufficiently small, we still have \( U < V \) on \( \partial' \left( \Omega^r \times\{t_0\} \right) \) by the continuity of \( U- V \). By the strong comparison principle {for the heat equation}, \( U < V \) in  \(\Oa\) and then, {in particular}, we have reduced \eqref{eq.claiminproofofcpthatcase3slightlystronger} to
\be\label{eq.crucialstrictnessforcomparison1}
U < V \textup{ on } \pd \Omega_h^r \times \{t_0\}.
\ee

Second, we show that \( U \) is continuously differentiable on
\[
\Omega^r\times\{t_0\} \subset \Omega_{t_0},
\]
for all small $r>0$. Indeed, for a fixed \( y \in \R^{d-1} \) such that
\[
 (x_0 + y, t_0) \in \Omega^r\times\{t_0\}, 
\]
we have
\[
U(x, t) - U(x + y, t) \equiv 0 \quad \text{for } x \in B_{r}'(x_0) \times (t_0 - \omega(r), t_0 + \omega(r)).
\]
Moreover, \( U(x, t) - U(x + y, t) \) solves the heat equation in the interior 
$$
B_{r}^+(x_0) \times (t_0 - \omega(r), t_0 + \omega(r)). 
$$
If we denote
\[
F_r:=B_{r}^+(x_0) \times (t_0 - \omega(r), t_0 + \omega(r)),
\]
we have by the standard Dirichlet boundary regularity estimate of heat equation (or also interior estimate such as Lemma \ref{l.interiorregularityofheatequation}), that for some $C>0$ independent of $y$
\[
\norm{U(x,t)-U(x+y,t)}_{C_{x,t}^{1,\alpha}(\overline{F_{r/2}})} \le C \norm{U(x,t)-U(x+y,t)}_{L_{x,t}^\infty(F_r)}.
\]
This implies that $U$ is differentiable at $(x_0+y,t_0)$ if and only if it is differentiable at $(x_0,t_0)$ and the latter was already proved in Step 3. Moreover, the derivatives of $U$ at $(x_0+y,t_0)$ is continuous with respect to $y$ because $U$ is continuous on $\Oa'$ by Lemma \ref{l.continuityuoflift}.

Third, we notice that because on \( \Omega_{t_0} \) we have $\gd' U=0$, by \eqref{eq.localstabilityatx0t0bydifferentiabilityforU}
\[
\partial_1 U \ge L_*(\gd' U)+\mu = m +  \mu .
\]
On the other hand, let $\ep>0$ small and consider the domain
\[
K_\ep:=\Omega_\ep^r\times(t_0-\ep,t_0+\ep).
\]
By the continuous differentiability of $U$ on $\Omega^r\times\{t_0\}$ and that $\gd'\restr{U}{\Ca_-} =0$, we have 
\be\label{eq.importantlowerflatnessincpccase3}
V\ge U \ge U(x_0,t_0) + \lb m + \mu\rb x_1 + b (t-t_0) - o(\ep), \textup{ on }\overline{K_\ep},
\ee
where 
$$
b=\pt_tU(x_0,t_0)=\pt_tV(x_0,t_0).
$$ 

Now, we give the construction of the test functions in different cases of the time derivative $b$.  If \( b \ge 0 \), then we obtain a contradiction to the assumption that $(x_0,t_0)\in \Ca_-$ by the same argument in Case 2. If \(  b < 0 \), we claim that the following function
\[
H(x, t) := U(x_0, t_0) + b(t - t_0) +  (m + \mu/2) x_1 - \frac{|b|}{100} |t - t_0|
\]
touches \( V \) from below at \( (x_0, t_0) \) in \( \overline{K_\ep} \) for sufficiently small \( \ep > 0 \). By comparison principle of heat equations, it suffices to show that \( H \le V \) on the parabolic boundary $\pp K_\ep$, as \( H \) is a subsolution to the heat equation. Indeed:
\begin{itemize}
    \item On \( \pt K_\ep \cap \{x_1 = 0\} \), by $\frac{1}{\delta}$-semi-concavity of $U$ in the time variable,
    \[
    H = U(x_0, t_0) + b(t - t_0) - \frac{|b|}{100} |t - t_0| \le U(x_0, t_0) + b(t - t_0) - \frac{1}{2\delta} |t - t_0|^2 \le U(x,t)
    \]
    for \( |t - t_0| \le \ep \), and $\ep>0$ smaller, if necessary, depending on $|b|>0$ and on $\delta>0$. 

    \item On \( K_\ep \cap \{x_1 \ge \ep/2\} \), by \eqref{eq.importantlowerflatnessincpccase3},
    \[
    H \le U(x_0, t_0) + b(t - t_0) + \left( m +  \mu \right) x_1 - \frac{1}{4} \mu \ep \le U(x,t) \le V(x,t)
    \]
    when \( \ep>0 \) is small enough.

    \item When \( t = t_0 - \ep \), by \eqref{eq.importantlowerflatnessincpccase3} again,
    \[
    H = U(x_0, t_0) + b(t - t_0) +  (m +\mu/2) x_1 - \frac{|b|}{100} \ep \le U(x,t) \le V(x,t)
    \]
    when \( \ep>0 \) is small enough. Again we are using $|b|>0$.

    \item On $\pd \Omega_\ep^r\times(t_0-\ep,t_0+\ep)$, by \eqref{eq.crucialstrictnessforcomparison1} and the continuity of $U-V$,
\[
V(x,t)>   U(x_0, t_0)+ \left( m + \mu\right) x_1 + b(t-t_0) \ge H,
\]
for \( \ep > 0 \) sufficiently small.
\end{itemize}
Combining the above discussions, we conclude that the one-dimensional test function \( H \) touches \( V \) from below at \( (x_0, t_0) \) in \( K_\ep \). By the laminar dynamic slope condition \dref{definevissoltopGD} \ref{condition(c)} of $V=\tilde{v}$, 
\[
m \geq \partial_1 H(x_0, t_0)=m +\mu/2,
\]
which is a contradiction.

This concludes the proof of \eqref{eq.claiminproofofcpthatcase3}, which, also, concludes the proof that Case 3 cannot occur, showing \eqref{eq.centrialinequality} and concluding the proof of comparison.

\end{proof}

\section{Special cases exhibiting facets}\label{section.specialcases}

In this section we study special solutions to \eqref{eq.pGDwithmotionlaw} and its steady states that will certainly exhibit facets (or contact set). This justifies that the homogenized equation \eqref{eq.pGDwithmotionlaw} does not reduce to the (trivial) standard Neumann problems. 

First we show, in \sref{strong-subs-bdry-max}, that any minimal supersolution to the elliptic problem \eqref{eq.generalhomogenizedequidef} satisfies a ``boundary maximum principle" on an open subset of $B_1'$. Solutions of Neumann problems
\[ \Delta u = 0 \ \hbox{ in } \ B_1^+ \ \hbox{ with } \ \partial_1 u = 0 \ \hbox{ on } \ B_1', \ \hbox{ and } \ u = g \ \hbox{ on } \ \pd B_1^+\]
do not, in general, have this property. {Specifically the set of $g \in C(\partial B_1^+)$ for which the Neumann solution fails the boundary maximum principle is non-empty and open in the uniform topology.}

In \sref{monotone-shift}, we study the viscosity solutions to \eqref{eq.pGDwithmotionlaw} that satisfy a specific type of time-monotone Dirichlet boundary condition. We show that they converge to an extremal steady state to the elliptic problem \eqref{eq.generalhomogenizedequidef} as time goes to infinity. By the discussions in Section \ref{s.strong-subs-bdry-max}, these steady states are generally not Neumann steady states. This indicates the general existence of facets/contact sets in the parabolic problem \eqref{eq.pGDwithmotionlaw}.

\subsection{Strong subsolutions and boundary maximum principle}\label{s.strong-subs-bdry-max}

In this subsection, we study the strong subsolutions arising in condition \ref{condition(3)} of Theorem \ref{t.characterizeminimalsuperslutiontohomoextremalsolutions}. We show that if $\max f$ is sufficiently large, then the corresponding minimal supersolutions to \eqref{eq.generalhomogenizedequidef} satisfies a \emph{boundary maximum principle} on an open subset of $B_1'$. This property is not generally satisfied by Neumann solutions. This phenomenon indicates that facets exist generally instead of merely in some isolated examples.

\begin{definition}
    For $K\in \R\cup \{+\infty\}$, say that $u$ is a \emph{$K$-strong subsolution} to the gradient degenerate problem \eqref{eq.generalhomogenizedequidef} if $u$ is a subsolution of
    \begin{equation}\label{e.neumann-strongsubprop}
        \Delta u \geq 0 \ \hbox{ in } \ B_1^+ \ \hbox{ with } \ \partial_1u \geq 0 \ \hbox{ on } \ B_1',
    \end{equation}
   and for any smooth function \( \eta(x) = \eta(x_1) \) (depending only on \( x_1 \)) that touches \( u \) from above at \( z \in B_1' \), the following holds:  If there exists a neighborhood \( \Omega \subset \mathbb{R}^d \) containing \( z \), with \( \Omega \cap \overline{B_1^+} \csubset B_1^+ \cup B_1' \), such that  
              \begin{itemize}
                  \item \( \eta \geq u \) in \( \Omega \cap \overline{B_1^+} \),
                  \item \( \eta > u \) on \( \partial \Omega \cap \overline{B_1^+} \),
              \end{itemize}
              then the derivative satisfies \( \partial_1 \eta(z) = \partial_1 \eta(0) \geq K \). We simply call a $+\infty$-strong subsolution a strong subsolution. Define $K$-strong supersolutions symmetrically.
\end{definition}

\begin{definition}
    Say that $u \in \textup{USC}(\overline{B_1^+})$ satisfies the \emph{boundary maximum principle} if for any $U\csubset B_1'$
    \[
    \max_{\overline{U}} u = \max_{\pt' U} u.
    \]
    Define, similarly, the boundary minimum principle for $u \in \textup{LSC}(\overline{B_1^+})$.
\end{definition}

\begin{remark}
    
For a generic boundary data $g\in C(\pt B_1\cap \{x_1\ge0\})$ we can find a continuum of solutions to \eqref{eq.generalhomogenizedequidefequiform} with only the gradient degenerate boundary condition $$\pt_1 u |\gd' u|=0$$ that are homogenization limits. Indeed, we consider the homogenization limits of the semi-linear problem \eqref{eq.homogenizationNeumann} with 
\[
f(u):=|K|\sin(2\pi u).
\]
Then we define for each $K\in \R$ the solution $u_K$ to be the maximal subsolution to \eqref{eq.generalhomogenizedequidefequiform} if $K>0$, or the minimal supersolution if $K<0$. Notice that as $\avg{f}=0$, $u_K$ is exactly the solution to the Neumann boundary condition $\pt_1 u_K=0$ when $K=0$.

\end{remark}

\begin{lemma}[\cite{feldman2024regularitytheorygradientdegenerate}*{Section 8.1}]\label{l.strongsubsolutionequaltomaxp}
    A subsolution to \eref{neumann-strongsubprop} is further a $+\infty$-strong subsolution if and only if it satisfies the boundary maximum principle on $B_1'$.
\end{lemma}

Let us discuss the implications of the $K$-strong subsolution property for $K<+\infty$.

\begin{lemma}\label{l.lstrongandthelipschitzconstant}
    Suppose a $K$-strong subsolution $u$ is Lipschitz on $B_1^+$ with Lipschitz constant $0<S< K$. Then $u$ satisfies the boundary maximum principle on $B_1'$. 
\end{lemma}

\begin{proof}
It suffices to show that $u$ is a $+\infty$-strong subsolution by Lemma \ref{l.strongsubsolutionequaltomaxp}.

Suppose there is a smooth function of the form $\eta(x)\equiv \eta(x_1)$ that touches $u$ from above at $z\in B_1'$ and for some open domain $\Omega\csubset \R^d$ containing $z$ such that $\Omega\cap \overline{B_1^+} \csubset B_1^+\cup B_1'$, $\eta \ge u$ in $\Omega\cap \overline{B_1^+}$, $\eta>u$ on $\pt \Omega \cap \overline{B_1^+}$.

On the other hand, by the Lipschitz continuity, we have
\[
 u(x_1,x')\le S x_1 + u(0,x') , \ \textup{ for }(x_1,x')\in \overline{\Omega'}\times [0,\delta],
\]
for some small $\delta>0$. Because $\eta>u$ on $\pt \Omega \cap \overline{B_1^+}$, we know that 
\[
u(0,x')<   \eta(0) = u(z) , \ \textup{ for }(0,x')\in \pt'{\Omega'}.
\]
This implies 
\[
u(x_1,x')< (S+\ep)x_1 + u(z) , \ \textup{ for }(x_1,x')\in \pt\lb{\Omega'}\times (-1,\delta)\rb\cap \overline{B_1^+}, 
\]
for any small $\ep>0$ such that $S+\ep<K$, contradicting the $K$-strong subsolution property of $u$.

\end{proof}

\begin{corollary}\label{cor.boudnarymaximumprincipleasmaxfgrows}
There is a finite nondecreasing function $\ell:[0,1)\rta [0,+\infty)$ with $\ell(1^-)=+\infty$ so that for any $r\in (0,1)$ if $\max f > \ell(r)$ then any minimal supersolution $u$ to \eqref{eq.generalhomogenizedequidef} with $\osc_{B_1^+} u \le 1$ satisfies the boundary maximum principle on $B_r'$.
\end{corollary}

\begin{proof}
    By Lemma 3.1 in \cite{feldman2024regularitytheorygradientdegenerate}, any minimal supersolution $u$ of \eqref{eq.generalhomogenizedequidef} with $\osc_{B_1^+} u \leq 1$ is Lipschitz in $B_r^+\cup B_r'$ with Lipschitz constant at most $C(d,r)$ independent of $\max f$. Define $\ell(r):= C(d,r) + 1$. Then, assuming $\max f > \ell(r)$ as in the statement, then $u$ is a $(\max f)$-strong subsolution in $\overline{B_r^+}$ having Lipschitz constant strictly smaller than $\ell(r)<\max f$. Then \lref{lstrongandthelipschitzconstant} implies that $u$ satisfies the boundary maximum principle.
\end{proof}

\begin{proof}[Proof of Proposition \ref{prop.generalexistenceoffacets}]

 By Remark \ref{r.normalizeaveragezero} we focus on the case $\avg{f}=0$. We argue with $\osc_{\pd B_1^+} g \le \ep$ for $\ep>0$ small and $\max f>0$ fixed. Notice that we can replace solutions $u$ by $u/\ep$ so that we can just focus on the case $\osc_{\pd B_1^+} g < 2$ and $\max f/\ep $ being large.

For each $g\in C(\pd B_1^+)$ define $v^g$ to be the unique solution of
    \[
    \bca
   \Delta v^g =0 & \textup{ in }B_1^+\\
   \pt_1 v^g = 0 & \textup{ on }B_1'\\
   v^g= g  & \textup{ on }\pd B_1^+.
    \eca
    \]
    We start with an example where the boundary maximum principle fails for $v^g$. Let $h(x) : = x_1$. Since $h(x)=x_1$ is a strict subsolution of the above Neumann problem $v^h(x) > x_1$ in $B_1^+$. In particular $v^h$ attains its positive maximum value on $B_1^+$ at some $x' \in B_1'$ with $|x'| < 1$. In particular $v^h$ fails to satisfy the boundary maximum principle in $B_r'$ for $|x'| < r_0 < 1$ and $r_0$ sufficiently close to $1$.
    
    Define $\mathcal{F}\subset C(\pd B_1^+)\cap\{\osc_{\pd B_1^+} g < 2\}$ to consist of all data $g$ such that $v^g$ does not satisfy the boundary maximum principle on $B_{r_0}'$ with $r_0$ defined as above. Notice that $\mathcal{F}$ is open, since failing the boundary maximum principle on some specific $U\csubset B_1'$ is an open condition in the uniform topology, and $g \mapsto v^g$ is continuous in uniform topology on $C(\pd B_1^+)$ by comparison principle. The set $\mathcal{F}$ is nonempty we we have already established that $v^h \in \mathcal{F}$, note that $0 \leq v^h \leq 1$ by comparison principle so that the oscillation condition is satisfied.

    By Corollary \ref{cor.boudnarymaximumprincipleasmaxfgrows} the minimal supersolutions $u_g$ of \eqref{eq.generalhomogenizedequidef} with boundary data $g\in \mathcal{F}$ satisfy the boundary maximum principle on $B_{r_0}'$ when $\max f/\ep$ is sufficiently large. If $\Ca_{u_g}=\emptyset$, then, by Lemma \ref{l.partitionofB1prime}, $\pt_1 u_g =0$ on $B_1'$ and hence by uniqueness $u_g=v^g$, violating the assumption that $v^g$ does not satisfy the boundary maximum principle on $B_{r_0}'$.
\end{proof}

\subsection{Extremal steady states and parabolic flows with monotone shift}\label{s.monotone-shift}

In this subsection, we discuss the steady states of the homogenized flow having a specific monotone Dirichlet boundary condition.

\begin{definition}\label{def.monotoneshift}
On the parabolic boundary $\ppp D_{\infty}^+$, we call a continuous function $g$ to be a \emph{monotone shift} if $g$ is monotone in time and
\[
\pi_g:=\lim_{T\rta\infty} g(\cdot,T) 
\]
exists in the sense of uniform convergence, and either
\be\label{eq.upshift}
\min_{x\in \pd B_1^+} \pi_g(x) - \max_{x\in \overline{B_1^+}} g(x,0)  >0,
\ee
or
\be\label{eq.downshift}
\min_{x\in \overline{B_1^+}} g(x,0)- \max_{x\in \pd B_1^+} \pi_g(x)   >0.
\ee
We call $g$ an up-shift if in the case \eqref{eq.upshift} and a down-shift if in the case \eqref{eq.downshift}.
\end{definition}

We show that the long-time limit of solutions to the parabolic flow \eqref{eq.pGDwithmotionlaw} under monotone shift boundary data is an extremal solution of the elliptic equation \eqref{eq.generalhomogenizedequidef}.

\begin{theorem}\label{t.relationbetweenextremsteadystatesandparabolicflow}
      Let $g \in C(\ppp D_{\infty}^+)$ a monotone up-shift (resp. down-shift) as in Definition \ref{def.monotoneshift}. Let $u$ be the solution to \eqref{eq.pGDwithmotionlaw} with boundary data $g$ on $\ppp D_{\infty}^+$. Then $u(\cdot,t)$ converges uniformly on $\overline{B_1^+}$ to $u_0$ the minimal supersolution of \eqref{eq.generalhomogenizedequidef} with boundary condition $u_0=\pi_g$ on $\pd B_1^+$ (resp. maximal subsolution).
\end{theorem}

\begin{proof}[Proof of Theorem \ref{t.relationbetweenextremsteadystatesandparabolicflow}]
The case for down-shift is symmetrical, so it suffices to consider the case that $g$ is an up-shift.

By comparison principle, we know that as $g$ is a bounded function on $\ppp D_{\infty}^+$, the functions $u(\cdot,t)$ are uniformly bounded as $t\rta\infty$. We define for $(x,t)\in \overline{B_1^+}\times[-1,1]$ 
\[
v^\ast(x,t)\equiv v^\ast(x):= {\limsup_{T\rta\infty}}^\ast u(x,t+T)\quad \textup{and} \quad
v_\ast(x,t)\equiv v_\ast(x):=\underset{T\rta\infty}{{\liminf}_\ast} u(x,t+T).
\]
Then $v^\ast$ is upper semicontinuous and $v_\ast$ is lower semicontinuous with $v^\ast \ge v_\ast$ in $\overline{B_1^+}$. Furthermore, by Lemma \ref{l.boundaryorderingg}, $v^\ast =v_\ast=\pi_g$ on $\pd B_1^+$. 

Let $\psi(x,t)$ be a smooth function that touches $v^\ast$ from above at $t=0$ and $x_0\in B_1^+\cup B_1'$, then because $v^\ast$ is constant in time, $\pt_t \psi(x_0,0)=0$. If $\Delta\psi(x_0,0)<0$, then in a small neighborhood of $(x_0,0)$, we have $\Delta\psi < \pt_t \psi$ making $\psi$ a strict supersolution of the heat equation. By Lemma \ref{l.stabilityoftouching}, we can find a sequence of large $T\rta\infty$ and $C_T=o_T(1)$ such that $\psi(\cdot,t-T)+C_T$ touches $u$ from above at $t_T\in[T-1,T+1]$ and some $x_T\in B_1^+\times B_1'$ such that $|x_T-x_0|+|t_T-T|=o_T(1)$. By the strict supersolution property of $\psi$, we know that $x_T\in B_1'$, and hence we have the condition for all large $T$,
\[
\pt_1 \psi(x_T,t_T-T) \ge L_\ast(\gd' \psi(x_T,t_T-T)),
\]
which implies that $v^\ast$ is a subsolution to the equation \eqref{eq.generalhomogenizedequidef}. Similarly, we can show that $v_\ast$ is a supersolution to \eqref{eq.generalhomogenizedequidef}. 

To show that $v_\ast=v^\ast$ is a minimal supersolution to \eqref{eq.generalhomogenizedequidef}, it suffices to show the condition \ref{condition(3)} in Theorem \ref{t.characterizeminimalsuperslutiontohomoextremalsolutions}. By Theorem \ref{t.characterizeminimalsuperslutiontohomoextremalsolutions}, we know that 
\[
v_\ast \ge \Tilde{v} \ge \min_{x\in \pd B_1^+} \pi_g(x)>\max_{x\in \overline{B_1^+}} g(x,0),
\]
where $\Tilde{v}$ is the minimal supersolution to \eqref{eq.generalhomogenizedequidef} with boundary data $\pi_g$. Moreover, for every fixed $x\in \pd B_1^+$ we have by monotonicity in time
\[
\pi_g(x) \ge g(x,t), t\ge0.
\]
It is not difficult to check that $\Tilde{v}$ is itself a stationary solution to the parabolic flow \eqref{eq.pGDwithmotionlaw}. According to the comparison principle \ref{t.comparisonprinciple}, this shows that $\Tilde{v} \ge u$ on the whole space-time domain $\overline{D_{\infty}^+}$. In particular, $$v^\ast \ge v_\ast\ge \Tilde{v} \ge u.$$ 
Now let $\phi(x_1,t)$ be a spatially 1-variable smooth function that touches $v^\ast$ from above at $(x_0,0)\in  B_1'\times\{0\}$ in $\overline{B_1^+}\times[-1,1]$ that satisfies for some open spatial domain $\Omega\csubset \R^d$ containing $x_0$ such that $\Omega\cap \overline{B_1^+}\csubset B_1^+\cup B_1'$ we have $\phi \ge v^\ast$ in ${\Omega}\cap \overline{B_1^+}$ and $\phi> v^\ast +\delta$ on $\lb\pt\Omega\rb\cap \overline{B_1^+}$ for some small $\delta>0$. In fact, by the inequality $v^\ast \ge u$ we know that 
\[
\min_{t\in[-1,1]}\phi(\cdot,t)> u +\delta, \ \textup{on} \lb\pt\Omega\cap \overline{B_1^+}\rb\times[0,\infty).
\]
We finish the proof by showing that $\pt_1 \phi(x_0,0)\ge \max f$. We argue by contradiction and assume that $\pt_1 \phi(x_0,0)< \max f$. To that end, we replace $\phi$ and $\Omega$ by 
$$\phi(x_1,t)+\mu x_1 +\mu t^2 -\frac{1}{\mu} x_1^2 \quad\textup{and}\quad\Omega\cap\{x_1\le \mu^2/4\}$$
for some small $\mu>0$, which ensures that $\phi$ is a supersolution of the heat equation in $\Omega\times[-t_\mu,t_\mu]$ for some small $t_\mu>0$ depending on $\mu$ and all the previous properties are preserved. Again by Lemma \ref{l.stabilityoftouching}, we can find a sequence of large $T\rta\infty$ and constants $C_T=o_T(1)$ such that $\phi_T:=\phi(\cdot,t-T)+C_T$ touches $u$ from above at $$(x_T,s_T)\in B_1'\times(T-t_\mu/2,T+t_\mu/2)\textup{ in }\lb\overline{\Omega}\cap \overline{B_1^+}\rb\times [T-t_\mu/2,T+t_\mu/2]$$
with $|x_T-x_0|+|s_T-T|=o_T(1)$. Moreover, we have when $T>0$ is large, 
\[
\phi_T(\cdot,T)>u+\delta/2, \ \textup{ on }\lb\pt\Omega\cap \overline{B_1^+}\rb\times[0,\infty),
\]
and by \eqref{eq.upshift} 
\[
\phi_T(\cdot,T)>u, \ \textup{ on }\lb\overline{\Omega}\cap\overline{B_1^+}\rb\times\{0\}.
\]
For $\eta>0$ small, we know that 
\[
p(x):=(\max f) x_1 + \phi_T(0,T)
\]
is a stationary solution to \eqref{eq.pGDwithmotionlaw}, $p(x) \ge \phi_T(x,T)$ for $x\in \overline{\Omega}\cap \{x_1\le \eta
\}=:\overline{\Omega}_\eta$ with strict inequality when $x_1>0$, and $p(x)$ also touches $u(\cdot,T)$ from above at $x_T$. However, this is a contradiction to the comparison principle \ref{t.comparisonprinciple} because by the construction $p>u$ on the parabolic boundary $\ppp \lb \Omega_\eta\cap B_1^+ \rb\times(0,T+\eta)$ for some sufficiently small $\eta>0$.

\end{proof}

\appendix

\section{Some regularity estimates}\label{appendix.someestimates}

Here we present several lemmas that are frequently used. 
\begin{lemma}\label{l.usefulholderestimate}
If a bounded function $u$ is harmonic in $B_1^+$, and in the sense of viscosity 
\[
\pt_1 u^\ast \le M\quad \textup{and} \quad  \pt_1 u_\ast \ge -M, \ \textup{ on }B_1'
\]
for some constant $M>0$, then \(u\in C_{\textup{loc}}^\alpha(B_1^+\cup B_1')\), and there is a constant $C>0$ independent of $u$ such that
\[
\norm{u}_{C^\alpha(B_{1/2}^+)} \le C \norm{u}_{L^\infty(B_1^+)}.
\]
Furthermore, if $u^*=u_*=0$ on $\pt B_1\cap\{x_1\ge0\}$ then $u\in C^\alpha(B_1^+)$ and
\[
\norm{u}_{C^\alpha(B_{1}^+)} \le C \norm{u}_{L^\infty(B_1^+)}.
\]
\end{lemma}

\begin{remark}
    The same result fails if one relaxes to the case that $u$ is merely bounded and $u^\ast$ is subharmonic and $u_\ast$ is superharmonic. This is because it is unlikely that $u^\ast=u_\ast$ in the measure-theoretic sense in the interior $B_1^+$.
\end{remark}

\begin{lemma}\label{l.usefulC1alphaestimate}
    If a continuous function $u$ is harmonic in $B_1^+$, and in the sense of viscosity 
\[
\pt_1 u = h, \ \textup{ on }B_1'
\]
for some $h\in C_{\loc}^\alpha({B_1'})$ and $\alpha>0$, then \(u\in C_{\textup{loc}}^{1,\alpha}(B_1^+\cup B_1')\), and there is a constant $C>0$ independent of $u$ such that
\[
\norm{u}_{C^{1,\alpha}(\overline{B_{1/2}^+})} \le C \norm{u}_{L^\infty(B_1^+)}.
\]
\end{lemma}

We have a similar parabolic version for the above estimates.

\begin{lemma}\label{l.usefulholderestimatepara}
If a bounded function $u$ satisfies $ \pt_t u = \Delta u$ in $D_{T}^+=B_1^+\times (0,T]$ for some end time $T>0$, and in the sense of viscosity 
\[
\pt_1 u^\ast \le M, \ \pt_1 u_\ast \ge -M, \ \textup{ on }D_{T}'
\]
for some constant $M>0$, then there is some $\alpha>0$, \(u\in C_{\textup{loc}}^{\alpha,\alpha/2}(D_{T}^+\cup D_{T}')\), and for any $0<h<T$ there is a constant $C>0$ independent of $u$ such that
\[
\norm{u}_{C^{\alpha,\alpha/2}(B_{1/2}^+\times (h,T))} \le C \norm{u}_{L^\infty(B_1^+\times (0,T))}.
\]
\end{lemma}

\begin{lemma}\label{l.usefulC1alphaestimatepara}
If a bounded continuous function $u$ satisfies $ \pt_t u =\Delta u$ in $B_1^+\times (0,T]$ for some end time $T>0$, and in the sense of viscosity
\[
\pt_1 u = h, \ \textup{ on }D_{T}'
\]
for some $h\in C_{\loc}^\alpha({D_{T}'})$ and $\alpha>0$, then \(u\in C_{\textup{loc}}^{1+\alpha,1/2+\alpha/2}(D_{T}^+\cup D_{T}')\), and for any $0<\tau<T$ there is a constant $C>0$ independent of $u$ such that
\[
\norm{u}_{C^{1+\alpha,1/2+\alpha/2}\lb B_{1/2}^+\times (\tau,T]\cup B_{1/2}'\times (\tau,T]\rb} \le C \norm{u}_{L^\infty(B_1^+\times (0,T))}.
\]
\end{lemma}

Lemma \ref{l.usefulC1alphaestimate} and the continuous version of Lemma \ref{l.usefulholderestimate} can be found in \cite{milakis2006regularity}. Here we present a proof for Lemma \ref{l.usefulholderestimate}. The proof of Lemma \ref{l.usefulholderestimatepara} is similar to that of Lemma \ref{l.usefulholderestimate}, and the proof of Lemma \ref{l.usefulC1alphaestimatepara} can be found in \cite{ladyzhenskaja1968linear}. 

\begin{proof}[Proof of Lemma \ref{l.usefulholderestimate}]

The proof is essentially similar to the continuous case as in \cite{milakis2006regularity}, and here we mainly explain the issue of discontinuity near $B_1'$. It suffices to show that for some $0<\mu<1$, we always have for small $r>0$
\[
\osc_{B_{r/2}^+} u \le \mu \osc_{B_r^+} u + Mr.
\]
To show this inequality, we consider the infimum $\Tilde{w}$ of all continuous solutions $w$ to the equation
\[
\bca
\Delta w =0 & \textup{ in }B_r^+\\
w\ge u^\ast & \textup{ on }\pt B_r \cap \{x_1\ge0\}\\
\pt_1 w \le -M &\textup{ on }B_r'.
\eca
\]
Notice that $\Tilde{u}:=\Tilde{w}-h$ with $h$ the classical solution to
\be\label{eq.classicalsolutionofholderharmonic}
\bca
\Delta h =0 & \textup{ in }B_r^+\\
h= 0,& \textup{ on }\pd B_r\\
\pt_1 h = M &\textup{ on }B_r',
\eca
\ee
gives the infimum of continuous solutions $p$ to
\[
\bca
\Delta p =0 & \textup{ in }B_r^+\\
p\ge u^\ast & \textup{ on }\pd B_r\\
\pt_1 p \le 0 &\textup{ on }B_r',
\eca
\]
which, after doing even reflection, by standard Perron's method gives a harmonic function $\Tilde{u}$ in $B_r$ such that $\Tilde{u}=u^\ast$ on $\pt B_r$. Similarly, we can define $\Bar{u}$ to be the harmonic function in $B_r$ such that $\Bar{u}=u_\ast$ on $\pt B_r$.

By $L^\infty$ theory of harmonic functions and the continuity of $u$ on $\pt B_r \cap \{x_1>0\}$, we know that
\[
\Tilde{u}=\Bar{u}=:v, \ \textup{ on }B_r^+\cup B_1'.
\]
This shows that
\[
|u-v| \le |h|, \ \textup{ on }B_r^+,
\]
which proves the claim.

In the case that $u^*=u_*=0$ on $\pt B_1\cap\{x_1\ge0\}$, we observe that by comparison principle, $h \le u_*\le u^* \le -h$ globally on $\overline{ B_1^+}$. It then suffices to show that $h$ is in $C^\alpha(B_1^+)$ for some $\alpha>0$. Observe that $h=\Tilde{h}-Mx_1$, where $\Tilde{h}$ can be evenly reflected to be a harmonic function on $B_1$ that satisfies
\[
\Tilde{h}= M |x_1| \ \textup{ on }\pt B_1.
\]
By \cite{Pavlovic2007}, we know that $\Tilde{h}$ is in $C^\alpha(B_1)$ and hence $h\in C^\alpha(B_1^+)$.
\end{proof}

\section{Tangential regularization}\label{appendix.tangentialregularization}
In this section, we introduce the {tangential regularization} procedure used in the parabolic comparison principle. A tangential regularization is a inf/sup-convolution involving only the tangential $x'$ (and time $t$ if in the parabolic case) variables with an additional harmonic (or caloric) lift. A tangential regularization is preferred in the analysis of nonlinear Neumann problems like \eqref{eq.pGDwithmotionlaw}, \eqref{eq.generalhomogenizedequidef} and \eqref{eq.generalhomogenized2}, because the commonly used doubling variable method does not apply due to the lack of a uniform obliqueness condition \cites{BARLES1999191,UsersGuide}, and the also commonly used standard inf/sup-convolutions does not work either due to similar issues.

\subsection{Tangential regularization for the elliptic case}
In the elliptic case for an upper semicontinuous function $u$ on $\overline{B_1^+}$, we extend $u=-\infty$ outside $\overline{B_1^+}$ and define for $\ep>0$ the following tangential sup-convolution
\[
\Ta^\ep u(x):=\sup\lma u(y) - \frac{1}{2\ep}|x-y|^2 \ ; \ x_1=y_1\rma.
\]
For a lower semicontinuous function $v$ on $\overline{B_1^+}$, we define similarly the tangential inf-convolution
\[
\Ta_\ep v(x):=\inf\lma v(y) + \frac{1}{2\ep}|x-y|^2 \ ; \ x_1=y_1\rma.
\]
The \emph{elliptic tangential regularization} $\Ha^\ep u$ and $\Ha_\ep v$ of $u$ and $v$ are respectively the harmonic lifts of $\Ta^\ep u$ and $\Ta_\ep v$, that is, we define $\Ha^\ep u$ as the infimum of all continuous harmonic functions on $\overline{B_1^+}$ that is above $\Ta^\ep u$ and $\Ha_\ep v$ the supremum of all continuous harmonic functions below $\Ta_\ep v$. The properties of the elliptic tangential regularization can be found in \cite{feldman2024regularitytheorygradientdegenerate}.

\subsection{Tangential regularization for the parabolic case}
Let us now focus on the tangential regularization for the parabolic case.
\subsubsection{Parabolic tangential inf-/sup-convolution}
\begin{definition}\label{def.tptconvolution}
Let $U\subset\R^d\times\R$ be a space-time domain, $u \in USC(\overline{U})$ and is finite on $\overline{U}$. {Usually we can extend by $u = -\infty$ outside $\overline{U}$ and this extension is still upper semicontinuous.} Define for $\ep>0$ {and all $(x,t) \in \R^d \times \R$} the parabolic tangential sup-convolution
\[
\Ta^\ep u( x, t):=\sup \lma u(y,s) - \frac{1}{2\ep}|x-y|^2 -\frac{1}{2\ep}(t-s)^2 \ ; \ x_1=y_1, \ (y,s)\in U  \rma.
\]
Similarly, for $v \in LSC(\overline{U})$ bounded from below and not identically $+\infty$, we define the parabolic tangential inf-convolution $\Ta_\ep v = -\Ta^\ep(-v)$.

\end{definition}

\begin{lemma}\label{l.lipschitzofconvolution}
    Let $u$ be upper semicontinuous and bounded in $\overline{U}$. For every $\ep>0$, the function $\Ta^\ep u$ is finite everywhere on $\R^d\times\R$, upper semicontinuous and $\frac{1}{2\ep}$-semi-convex in $(x_1,x',t)$ for any fixed $x_1$. Moreover, for any $R>0, 0<\ep<1$ we have the following Lipschitz estimate: if $x_1=y_1$ and $|x|^2+|y|^2+|t|^2+|s|^2\le R^2$ then
    \[
    |\Ta^\ep u(x,t)-\Ta^\ep u(y,s)| \le C_\ep \lb R + \norm{u}_{L^\infty(U)}\rb\lb|x-y|+|t-s|\rb.
    \]
\end{lemma}

\begin{proof}
The proof of the pointwise finiteness and semi-convexity in $(x_1,x',t)$ for fixed $x_1$ is the same as the standard sup-convolutions. 

Similar to the standard sup-convolution, by boundedness of $u$, for sufficiently small $\ep>0$ there is a $r_\ep>0$ such that at any $(x_0,t_0)=(x_{0,1},x_0',t_0)\in \overline{U}$ there will be a $(x_\ep,t_\ep)=(x_{0,1},x_\ep',t_\ep)\in \overline{U}\cap B_{r_\ep}(x_0,t_0)$ such that
\[
\Ta^\ep u(x_0,t_0) = u(x_\ep, t_\ep) - \frac{1}{2\ep}|x_0'-x_\ep'|^2 - \frac{1}{2\ep }(t_0-t_\ep)^2.
\]
Notice that if $|u|\le M$ then $r_\ep = (100M\ep)^{1/2}$ will suffice because outside the ball $B_{r_\ep}(x_0,t_0)$ we have $|x_0'-x_\ep'|^2+(t_0-t_\ep)^2 > 100M \ep$ and hence $\Ta^\ep u(x_0,t_0)<-M\le u(x_0,t_0)$, which is impossible.

To show the upper semicontinuity we consider any converging sequence of points $(x_n,t_n)=(x_{n,1},x_n',t_n)\rta (x_0,t_0)=(x_{0,1},x_0',t_0)$ in $\overline{U}$. By the above discussion, there will be a sequence $(x_{n,\ep},t_{n,\ep})= (x_{n,1},x_{n,\ep}',t_{n,\ep})\in \overline{U}\cap B_{r_\ep}(x_n,t_n)$ such that
\[
\Ta^\ep u(x_n,t_n) = u(x_{n,\ep},t_{n,\ep}) - \frac{1}{2\ep}|x_n'-x_{n,\ep}'|^2 - \frac{1}{2\ep}(t_{n}-t_{n,\ep})^2.
\]
By compactness of the sequence $(x_{n,\ep},t_{n,\ep})$ for each fixed $\ep$, we know that in a convergent subsequence $n_k$ with $\lim_{k\rta\infty}(x_{n_k,\ep},t_{n_k,\ep}) = (\Tilde{x}_{\ep},\Tilde{t}_\ep)$
\[
\limsup_{k\rta\infty}\Ta^\ep u(x_{n_k},t_{n_k}) \le u(\Tilde{x}_{\ep},\Tilde{t}_\ep) - \frac{1}{2\ep} |\Tilde{x}_{\ep}'-x_0'|^2 -\frac{1}{2\ep}(\Tilde{t}_\ep-t_0)^2\le \Ta^\ep u(x_0,t_0),
\]
which implies that $\Ta^\ep u$ is upper semicontinuous at $(x_0,t_0)$.

To show the Lipschitz estimate, we write 
\[
\Ta^\ep u(x,t) = u(x_\ep, t_\ep)-\frac{1}{2\ep}|x'-x_\ep'|^2-\frac{1}{2\ep}(t-t_\ep)^2,
\]
and because for $(y,s)$ such that $y_1=x_1$
\[
\Ta^\ep u(y,s) \ge u(x_\ep, t_\ep)-\frac{1}{2\ep}|y'-x_\ep'|^2-\frac{1}{2\ep}(s-t_\ep)^2,
\]
we have
\[
\begin{split}
  \Ta^\ep u(x,t)-  \Ta^\ep u(y,s) &\le \frac{1}{2\ep}\lb|y'-x_\ep'|^2+(s-t_\ep)^2-|x'-x_\ep'|^2 -(t-t_\ep)^2 \rb\\
  &\le C_\ep \lb R+\norm{u}_{L^\infty(U)}\rb\lb |x-y|+|t-s|\rb.
\end{split}
\]

\end{proof}

\begin{lemma}\label{l.touchingpropertyofconvolution}
     Let $u$ be upper semicontinuous and bounded in $\overline{U}$.  Let $\phi$ be a smooth function crossing $\Ta^\ep u$ from above (strictly) at $(x_0,t_0)=(x_{0,1},x_0',t_0)\in \overline{U}$, and let $(x_\ep,t_\ep)=(x_{0,1},x_\ep',t_\ep)\in \overline{U}$ be a point close to $(x_0,t_0)$ such that
    \[
    \Ta^\ep u(x_0,t_0) = u(x_\ep,t_\ep) -\frac{1}{2\ep}|x_\ep'-x_0'|^2 -\frac{1}{2\ep}(t_\ep-t_0)^2.
    \]
    Then $\phi^\ep(x,t):=\phi(x+x_0-x_\ep,t+t_0-t_\ep) -\frac{1}{2\ep}|x_\ep'-x_0'|^2 -\frac{1}{2\ep}(t_\ep-t_0)^2$ crosses $u$ from above (strictly) at $(x_\ep,t_\ep)$ with 
    \[
   \lim_{\ep\rta0} \frac{1}{2\ep}|x_\ep'-x_0'|^2 +\frac{1}{2\ep}(t_\ep-t_0)^2=0.
    \]
  Moreover, we have
    \[
\gd \phi^\ep(x_\ep,t_\ep)=\gd \phi(x_0,t_0)=\frac{1}{\ep}(x_\ep'-x_0')\quad \textup{and} \quad 
 \pt_t \phi^\ep(x_\ep,t_\ep)= \pt_t \phi(x_0,t_0)\le\frac{1}{\ep}(t_\ep-t_0).
\]

\end{lemma}

\begin{proof}
    The proof is done by a simple adaptation of the arguments in \cite{calder2018lecture}*{Proposition 8.6}.
\end{proof}

\begin{corollary}\label{cor.subsolutionsforregularizations}
Let $U=( B_1^+\cup B_1') \times [0,\infty)$, and $u$ be a subsolution to \eqref{eq.pGDwithmotionlaw} on $U$, then for every $\theta>0$ there is small $\delta_0>0$ so that for all $0<\delta<\delta_0$, the parabolic tangential regularization $\Ta^\delta u$ is a subsolution to \eqref{eq.pGDwithmotionlaw} on 
\[
(B_{1-\theta}^+\cup B_{1-\theta}') \times [\theta,\infty).
\]
    
\end{corollary}

\begin{proof}
    By Lemma \ref{l.touchingpropertyofconvolution}, any smooth function crossing $\Ta^\delta u$ from above will also cross $u$ from above with a small additive constant and small-distance translations in time and tangential directions.
\end{proof}
 
\begin{lemma}\label{l.convergenceofsupconvolution}
    Let $u$ be upper semicontinuous in $\overline{U}$, then the upper half relaxed limit
    \[
     {\limsup_{\ep\rta0}}^\ast \Ta^\ep u = \bca
u,& \textup{ in }\overline{U}\\
-\infty,& \textup{ elsewhere.}
     \eca
    \]
In particular, for any compact subset $K\csubset \R^d\times \R$ we have
    \[
    \lim_{\ep\rta 0} \max_{(x,t)\in K} \Ta^\ep u(x,t) = \max_{(x,t)\in K} u(x,t).
    \]
\end{lemma}

\begin{proof}
   This convergence property is a corollary of Proposition 3.7 in \cite{UsersGuide}.
\end{proof}

\subsubsection{Caloric lift}

Let $D_{T}=B_1\times(0,T)$ be a space-time cylinder and $D_{T}^+:= B_1^+\times (0,T]$ the positive interior. For any $g\in C\lb \pp D_{T}^+\rb$ there is a solution $u \in C(\overline{D_T^+}) \cap C^\infty(D_T^+)$ of
\[
\bca
\pt_t u = \Delta u   & \textup{ in }D_{T}^+\\
u=g &\textup{ on }\pp D_{T}^+.
\eca
\]
We show a similar result above a subsolution to $\Delta u \ge \pt_t u$.
\begin{lemma}\label{l.caloriclift}
Let $v$ be a bounded upper semicontinuous function on $\overline{D_{T}^+}$ and a subsolution to $\pt_t v \leq \Delta v $ in $D_{T}^+$, then there is a unique $u\ge v$ on $\overline{D_{T}^+}$ such that $u$ solves $ \pt_t u = \Delta u$ in the classical sense in $D_{T}^+$ and 
\[
\limsup_{D_T^+\ni (y,s)\rta (x,t)} \ u(y,s) = v(x,t)
\]
for any $(x,t)\in \pp D_T^+$.

\end{lemma}

\begin{proof}
We define
   \[
   u(x):=\inf\lma p(x) \ ; \ p\in C\lb\overline{D_{T}^+}\rb, \ \pt_t p \ge \Delta p, \quad \textup{and} \quad p\ge v \rma.
   \]
By standard Perron's method we know that $u$ satisfies the required properties.
\end{proof}

\begin{definition}\label{def.caloriclifts}
We call $u$ the caloric lift of $v$ on $\overline{D_{T}^+}$, and denote
\[
\Ha v := u.
\]
For a supersolution $w$ to $\pt_t w \geq \Delta w $ we can define similarly $\Ha w:= - \Ha(-w)$.
\end{definition}

An important property of caloric lift is that it preserves the sub/supersolution conditions of \eqref{eq.pGDwithmotionlaw}.

\begin{lemma}\label{l.preserveviscocaloriclift}
    Suppose $u$ is a (upper semicontinuous) subsolution to \eqref{eq.pGDwithmotionlaw} on $\overline{D_T^+}$ in the sense of \dref{definevissoltopGD}, then its caloric lift $\Ha u$ is also a subsolution to \eqref{eq.pGDwithmotionlaw}.
\end{lemma}

\begin{proof}
    In the interior $\Ha u$ satisfies the heat equation in the classical sense, so we only have to worry about the boundary condition. By Lemma \ref{l.caloriclift}, we know that $\Ha u =u$ on the boundary $\pp D_T^+$ in the sense of semicontinuity, and because $\Ha u \ge u$ in the interior, any smooth function crossing $\Ha u$ from above at the boundary point will also cross $u$ from above, and therefore $\Ha u$ inherits all the viscosity subsolution conditions of $u$.
\end{proof}

Let us now discuss some initial regularity of caloric lifts.

\begin{lemma}\label{l.continuityuoflift}
    Let $v$ be as described in the previous lemma. If $\restr{v}{\pp D_{T}^+}$ is continuous at $(x_0,t_0)\in \pp D_{T}^+$, then the corresponding caloric lift $\Ha v$ is continuous at $(x_0,t_0)$.
\end{lemma}

\begin{proof}
     Let $\omega(s)$ be a modulus of continuity of $\restr{v}{\pp D_{T}^+}$ at $(x_0,t_0)$ and define, for $(y,s) \in \pp D_{T}^+$,
    \[h_\pm(y,s):=v(x_0,t_0) \pm \omega(|y-x_0|+|s-t_0|).\]
    {By definition $h_+ \geq v \ge h_-$ on $\pp D_{T}^+$.}
    Extend $h_\pm$ to the parabolic interior $D_{T}^+$ by solving $\Delta h_\pm = \pt_t h_\pm$. This gives us a continuous upper barrier $h_+$ and lower barrier $h_-$ of $u$ at $(x_0,t_0)$, making $u$ continuous at $(x_0,t_0)$.
\end{proof}

We use the following interior estimate.

\begin{lemma}\label{l.interiorregularityofheatequation}
    Let $u$ be a bounded solution to the heat equation $\pt_t u = \Delta u +f$ in $B_1\times(0,T]$ for some $f\in C^{\alpha,\alpha/2}\lb\overline{B_1}\times[0,T]\rb$, then there exists a constant $C>0$ depending only on $\alpha$, $0<h<T$ and dimension $d$ such that
    \[
    \norm{u}_{C^{2+\alpha,1+\alpha/2}\lb B_{1/2}\times(h,T]\rb} \le C \lb \norm{u}_{L^\infty(B_1\times(0,T])}+ \norm{f}_{C^{\alpha,\alpha/2}({B_1\times(0,T]})}\rb.
    \]
\end{lemma}
\begin{proof}
    See \cite{ladyzhenskaja1968linear}.
\end{proof}

The following lemmas are useful in the proof of Theorem \ref{t.comparisonprinciple}. This first result allows us to show that caloric lifts of Lipschitz (in time) boundary data are Lipschitz (in time).

\begin{lemma}
    \label{l.boundarylipschitzheatequation}
    Let $u$ be a bounded solution to the heat equation $\pt_t u = \Delta u $ in $U\times(t_1,t_2)$ for some relative open domain $U\csubset B_1^+\cup B_1'$ and $\restr{u}{\overline{U'}\times[t_1,t_2]}$ is Lipschitz, then for any relative open subdomain $V\csubset U$ and small $r>0$ there is
    \[
    \norm{\pt_t u}_{L^\infty(V\times(t_1+r,t_2-r))}<\infty.
    \]
\end{lemma}

\begin{proof}
   Let $g$ be a Lipschitz extension of $\restr{u}{\overline{U'}\times[t_1,t_2]}$ to the whole $\R^d\times\R$ with the same Lipschitz constant. We solve for $h$ an auxiliary function satisfying
    \[
    \bca
\pt_t h = \Delta h & \textup{ in }U^+\times(t_1,t_2]\\
h = g & \textup{ on } \pt U\times[t_1,t_2]\\
\Delta h(\cdot,t_1) = 0 & \textup{ in }U^+.
    \eca
    \]
Suppose $L>0$ is the Lipschitz constant for $g$, then by comparison principle with the testing function $h(x,t_1)\pm L (t-t_1)$, we know that
\be\label{eq.lemmac9h}
|h(x,t)-h(x,t_1)| \le L(t-t_1).
\ee
On the other hand, by applying the maximum principle on $h(x,t+s)-h(x,t)$ for $s>0$, combining \eqref{eq.lemmac9h} and the Lipschitz regularity of $g$  we know that 
\[
\norm{\pt_t h}_{L^\infty(U\times(t_1,t_2))}<\infty.
\]
Now the proof follows by considering the decomposition
\[
u=(u-h)+ h,
\]
where $u-h$ is smooth near $U'$ as it satisfies the zero Dirichlet boundary condition on $U'$, while $h$ is Lipschitz in time.
    
\end{proof}

This next result shows that a solution of the heat equation is (quantitatively) differentiable at a point where the boundary data is $C^{1,1}$.

\begin{lemma}\label{l.differentiabilityatc11point}
    Let $U\csubset B_1^+\cup B_1'$ be a relatively open domain. Suppose $v$ is a lower semicontinuous supersolution to the heat equation on $U^+\times(0,T)$ and $u\in C(U\times(0,T))$ is a subsolution to the heat equation. If $v$ touches $u$ from above at $(x_0,t_0)\in U'\times(0,T)$ and both $\restr{u}{U'\times(0,T)}$,$\restr{v}{U'\times(0,T)}$ are $C^{1,1}$ at $(x_0,t_0)$, then there is a sequence of radius $r_k\rta0^+$ and smooth functions $\psi_k,\xi_k$ on 
    on
\[
\left( B_{r_k}^+(x_0) \cup B_{r_k}'(x_0) \right) \times (t_0 - r_k^2, t_0 + r_k^2),
\]
such that:
\begin{center}
\begin{varwidth}{\linewidth}
\begin{enumerate}
  \item[(A)] \( \xi_k \) touches \( U \) from above at \( (x_0, t_0) \),
    \item[(B)] \( \psi_k \) touches \( V \) from below at \( (x_0, t_0) \),
    \item[(C)] \( \partial_t (\xi_k - \psi_k)(x_0, t_0) = 0 \),
    \item[(D)] and
    \[
    \lim_{k \to \infty} \left| \nabla_x (\xi_k - \psi_k)(x_0, t_0) \right| \to 0.
    \]
\end{enumerate}
\end{varwidth}
\end{center}
In particular, if $u=v\in C(U\times(0,T))$ satisfies the heat equation in $U^+\times(0,T)$ then $u$ is differentiable at $(x_0,t_0)$.
\end{lemma}

\begin{proof}
 
To define \( r_k, \xi_k, \psi_k \), we first choose a small number \( r_1 = r >0\) such that
\[
u \le v \quad \text{in } \overline{B_r^+(x_0)} \times [t_0 - r^2, t_0 + r^2].
\]
By the $C^{1,1}$ regularity of $\restr{u}{U'\times(0,T)}$ and $\restr{v}{U'\times(0,T)}$ at $(x_0,t_0)$, we can find two polynomials on \(\pt\R_+^d\times\R \) of the form
\[
P(x', t) = u(x_0,t_0) + p \cdot (x' - x_0') + b(t - t_0) - C \left( |x' - x_0'|^2 + (t - t_0)^2 \right),
\]
and
\[
Q(x', t) = u(x_0,t_0) +  p \cdot (x' - x_0') + b(t - t_0) + C \left( |x' - x_0'|^2 + (t - t_0)^2 \right),
\]
where $p\in \pt\R_+^d$, $b\in \R$ and $C>0$ such that 
$$Q \ge \restr{v}{U'\times(0,T)}\ge \restr{u}{U'\times(0,T)} \ge P$$
and since $u$ touches $v$ from below, we have
\begin{itemize}
    \item \( P \) touches \( \restr{v}{U'\times(0,T)} \) from below at \( (x_0, t_0) \),
    \item \( Q \) touches \( \restr{u}{U'\times(0,T)} \) from above at \( (x_0, t_0) \)
\end{itemize} 
respectively in the domain \( \overline{B_r^+(x_0)} \times [t_0 - r^2, t_0 + r^2] \).

We define \( \xi_1 \) and \( \psi_1 \) to be the caloric lifts (See Definition \ref{def.caloriclifts}) in \( B_r^+(x_0) \times (t_0 - r^2, t_0 + r^2) \) that share the same boundary data \( u \) on
\[
\overline{B_r^+(x_0)} \times [t_0 - r^2, t_0 + r^2] \setminus \left( B_r^+(x_0) \cup B_1'(x_0) \right) \times (t_0 - r^2, t_0 + r^2],
\]
with
\[
\restr{\xi_1}{\{x_1 = 0\}} = Q \quad \text{and} \quad \restr{\psi_1}{\{x_1 = 0\}} = P.
\]
By the comparison principle for heat equations, \( \psi_1 \) and \( \xi_1 \) satisfy the conditions (A) and (B). The condition (C) is satisfied because of boundary regularity of heat equations near $(x_0,t_0)$. Moreover, this construction can be repeated for any shrinking radius \( r_k \to 0^+ \) with \( r_k < r_1 = r \). To complete the construction and check condition (D), we verify the convergence of the gradients of \( \psi_k, \xi_k \). This is achieved by observing that the function
\[
H_k(y, s) := \frac{(\psi_k - \xi_k)(r_k y + x_0, r_k^2 s + t_0)}{r_k}
\]
satisfies the following heat equation:
\[
\begin{cases}
\partial_s H_k = \Delta_y H_k & \text{in } B_1^+ \times (-1, 1), \\
H_k = 2C \left( r_k |y'|^2 + r_k^3 s^2 \right) & \text{on } B_1' \times (-1, 1], \\
H_k \equiv 0 & \text{on } \overline{B_1^+} \times [-1, 1] \setminus \left( B_1^+ \cup B_1' \right) \times (-1, 1].
\end{cases}
\]
By Lemma \ref{l.interiorregularityofheatequation}, we have
\[
\left| \nabla_x (\xi_k - \psi_k)(x_0, t_0) \right| = \left| \nabla_y H_k(0, 0) \right| = o_k(1),
\]
which completes the construction and condition (D) is checked.

\end{proof}

\bibliographystyle{plainnat}
\bibliography{ref}

\end{document}